\documentclass[hidelinks,11pt]{article}

\newcommand{\bol}{\boldsymbol}

\newcommand{\ney}{\boldsymbol{y}}                          
\newcommand{\nex}{\boldsymbol{x}}           
\newcommand{\bnex}{\bold{x}}

\newcommand{\ner}{\boldsymbol{r}}

\newcommand{\de}{\,\mathrm{d}}                               
\newcommand{\e}{\operatorname{e}}                               
                     
\newcommand{\inc}{\mathrm{inc}}

\newcommand{\andtext}{\quad\mbox{and}\quad}

\newcommand{\p}{\partial}

\newcommand{\real}{\mathrm{Re}\,}    
                                   
\newcommand{\imag}{\mathrm{Im}\,}

\newcommand{\lf}{\left}
\newcommand{\rg}{\right}

\newcommand{\R}{\mathbb{R}}       
\newcommand{\C}{\mathbb{C}}

\newcommand{\nor}{\boldsymbol n}

\usepackage{multirow}
\usepackage{amsmath}
\usepackage{amsfonts}
\usepackage{amsmath}
\usepackage{amssymb}  
\usepackage{graphicx}
\usepackage{caption}
\usepackage{mathrsfs}
\usepackage{upgreek}
\usepackage{amsthm}
\usepackage{subfig}
\usepackage{booktabs}
\usepackage{authblk}

\newtheorem{theorem}{Theorem}[section]
\newtheorem{lemma}[theorem]{Lemma}

\newtheorem{remark}[theorem]{Remark}

\title{Smoothed Combined Field Integral Equations for \\
 Exterior Helmholtz  Problems }
\author[1]{Carlos P\'erez-Arancibia\thanks{cperezar@mit.edu}}
\affil[1]{\small{Department of Mathematics, Massachusetts Institute of Technology}}
\date{\today}
\begin{document}
\maketitle

\begin{abstract}
This paper presents smoothed combined field integral equations for the  solution of Dirichlet and Neumann exterior Helmholtz problems. The integral equations introduced in this paper are smooth in the sense that they only involve continuously differentiable integrands in both Dirichlet and Neumann cases. These integral equations coincide with the well-known combined field equations and are therefore uniquely solvable for all frequencies. In particular, a novel regularization of the hypersingular operator is obtained, which, unlike regularizations based on Maue's integration-by-parts formula, does not give rise to involved Cauchy principal value integrals. 
The smoothed integral operators and layer potentials, on the other hand, can be numerically evaluated at target points that are arbitrarily close to the boundary without severely compromising their accuracy. A variety of numerical examples in two spatial dimensions that consider three different Nystr\"om discretizations for smooth domains and domains with corners---one of which is based on direct application of the trapezoidal rule---demonstrates the effectiveness of the proposed integral approach. In certain aspects, this work extends to the uniquely solvable Dirichlet and Neumann combined field integral equations, the ideas presented in the recent contribution R. Soc. Open Sci. 2(140520), 2015.
   

\end{abstract}

\section{Introduction}\label{sec:intro}

As is well known,  boundary integral equation (BIE) methods, such as boundary element methods~\cite{bonnet1999boundary,Sauter2010} as well as  Nystr\"om methods~\cite{Bruno:2001ima,Kress:1990vm,Kress:1995,KUSSMAUL:1969,MARTENSEN:1963}, provide several advantages over methods based on volume discretization of the computational domain, such as finite difference~\cite{taflove1995computational} and finite element methods~\cite{JIN:2002}, for  the solution of exterior Helmholtz problems. For example, BIE methods can easily handle unbounded domains and radiation conditions at infinity without recourse to approximate absorbing/transparent boundary conditions for truncation of the computational domain~\cite{Givoli2013}. Additionally, BIE methods are based on discretization of the relevant physical boundaries, and they therefore give rise to linear systems of reduced dimensionality---which, although dense, can be efficiently solved by means of accelerated iterative linear algebra solvers~\cite{Bebendorf2008,Bruno:2001ima,Greengard1998,phillips1997precorrected}.

One of the main issues associated with the use of BIE methods is the numerical evaluation of the challenging singular, weakly-singular and nearly-singular integrals that are inherent to the integral operators and layer potentials upon which BIE methods are based on. In two spatial dimensions, for example,  the single-layer, double-layer and adjoint double-layer operators feature weak $O(\log|\nex-\ney|)$ kernel singularities, while the hypersingular operator  features a much more stronger $O(|\nex-\ney|^{-2})+O(\log|\nex-\ney|)$ kernel singularity as $\ney\to\nex$, where $\nex$ and $\ney$ denote points on the (assumed smooth) boundary. As is known, however,  application of the standard regularization procedure, which was originally proposed by Maue~\cite{Maue:1949in}, enables the hypersingular operator to be expressed in terms of a Cauchy principal value integral that exhibits a $O(|\nex-\ney|^{-1})+O(\log|\nex-\ney|)$ kernel singularity as~$\ney\to\nex$. Nearly-singular integrals arise, on the other hand, when integral operators and layer potentials are evaluated at target points close to but not on the boundary of the domain. All these issues greatly hinder the use of BIE methods, as numerical evaluation of integral operators and layer potentials requires special treatment of the kernel singularities by means of specialized quadrature rules and/or semi-analytical techniques for which there is a vast literature that will not be reviewed here; cf.~\cite{atkinson1997numerical,Barnett:2014tq,bonnet1999boundary,Bruno:2001ima,Bruno:2012dx,COLTON:2012,dominguez2014nystrom,Kress:1995,Kress:1990vm,klockner2013quadrature,saranen2013periodic,Sauter2010}.

This paper presents uniquely solvable BIEs for the solution of exterior Helmholtz problems with Dirichlet and Neumann boundary conditions, that arise from  electromagnetic scattering by perfectly conducting obstacles in two spatial dimensions. The BIEs introduced in this contribution, dubbed \emph{smoothed combined field integral equations}, coincide with the common combined field integral equations~(CFIEs)~\cite{Burton1971Application,panish1965question,brakhage1965dirichletsche,leis1965dirichletschen} in both Dirichlet and Nuemann cases. Unlike the common CFIEs, however, they are given in terms of operators expressed as integrals of continuously differentiable functions (they indeed exhibit a mild singularity of the form $O(|\nex-\ney|^2\log|\nex-\ney|)$ as $\ney\to\nex$). In particular, we hereby introduce a novel regularization of the  hypersingular operator that involves neither Cauchy principal value nor weakly singular integrals. The proposed smoothing procedure is also utilized to regularize nearly-singular integrals that arise from evaluation of layer potentials at target points near the boundary, and from integral operators that result from integral formulations of problems involving two or more obstacles close to each other. 

Our smoothing procedure relies on the existence of certain homogenous solutions $p_j$, $j=0,\ldots,N$ ($N>0$) of the Helmholtz equation, that we referred to as~\emph{smoothing functions}. Such functions allow a sufficiently smooth density  $\varphi$ to be expressed as  $\varphi(\ney)=\sum_{j=0}^N\p_s^j\varphi(\nex)p_j(\ney|\nex)+O(|\nex-\ney|^{N+1})$ and $i\eta\varphi(\ney)=\sum_{j=0}^N\p_s^j\varphi(\nex)\p_n p_j(\ney|\nex)+O(|\nex-\ney|^{N+1})$ where $\nex$ is a given point on the boundary $\Gamma$ and where $\eta>0$ is a constant  (the symbols  $\p_s$ and $\p_n$ denote  tangential and normal derivatives on $\Gamma$).  Calling $K=K_1-i\eta K_2$  the kernel of  the Dirichlet or Neumann combined field operators, we thus  can write
\begin{equation*}\begin{split}
\int_{\Gamma}K(\nex,\ney)\varphi(\ney)\de s(\ney) =&\ \sum_{j=0}^{N}\p^j_s\varphi(\nex)\int_{\Gamma}\{K_1(\nex,\ney)p_{j}(\ney|\nex)-K_2(\nex,\ney)\p_np_{j}(\ney|\nex)\}\de s(\ney)\\
&\ +\int_{\Gamma}\{K_1(\nex,\ney)\rho_1(\ney|\nex)-K_2(\nex,\ney)\rho_2(\ney|\nex)\}\de s(\ney)\label{eq:intro}\end{split}
\end{equation*}
where   $\rho_1(\ney|\nex)=\varphi(\ney)-\sum_{j=0}^N\p_s^j\varphi(\nex)p_j(\ney|\nex)$ and $\rho_2(\ney|\nex)=i\eta\varphi(\ney)-\sum_{j=0}^N\p_s^j\varphi(\nex)\p_np_j(\ney|\nex)$. As it turns out, Green's third identity provides closed-form expressions for the boundary integrals inside the sum. Therefore,  the operator  $\int_{\Gamma}K(\nex,\ney)\varphi(\ney)\de s(\ney) $ can be easily evaluated by integrating the smoothed mildly singular functions $K_j(\nex,\ney)\rho_{j}(\ney|\nex)$, $j=1,2$---which satisfy $K_j(\nex,\ney)\rho_{j}(\ney|\nex)=O(|\nex-\ney|^{N+1}\log|\nex-\ney|)$---wherever on $\Gamma$ the tangential derivatives $\p_s^j\varphi(\nex)$, $j=1,\ldots,N,$ exist. In this paper we present a smoothing procedure that considers functions $p_j$, $j=0,1$, that are obtained explicitly as linear combinations of plane waves.

A smoothing procedure similar in nature to the one presented here was originally introduced in~\cite{Klaseboer:2012ks} for the solution of the Laplace equation and was later extended in~\cite{sun2015boundary} to the Helmholtz equation in three spatial dimensions. Both contributions consider integral equations derived from direct use of Green's third identity. As such, the associated smoothed integral equations for the Helmholtz equation suffer from spurious resonances in both Dirichlet and Neumann cases~\cite{COLTON:1983}.  The smoothing procedure introduced in those references, on the other hand, which provides a smoothing factor that turns weakly singular integrands (in three-dimensions) into bounded but discontinuous functions, does not suffice for the regularization of the hypersingular operator that appears in the combined field integral equation for the Neumann problem.

The structure of this paper is as follows. Section~\ref{sec:prelim} presents the boundary value problems considered in this paper and reviews the definition and main properties of the layer potentials and boundary integral operators. Section~\ref{eq:reg_IE}, subsequently,  introduces the smoothed CFIE  formulations for both Dirichlet and Nuemann problems. Details on the construction of the smoothing functions are provided in Section~\ref{sec:reg_functions}. Finally, Section~\ref{sec:numerics} presents a variety of numerical examples in two spatial dimensions that include three different Nystr\"om discretizations for smooth domains and domains with corners.

\section{Preliminaries}\label{sec:prelim}
This paper  considers exterior Helmholtz boundary value problems that arise as an incident TE- or TM-polarized electromagnetic wave impinges on the  surface of an axially symmetric perfect electric conductor with cross section $\Omega\subset\R^2$ and boundary $\p\Omega=\Gamma$. In TE-polarization the scattered field $u_D:\R^2\setminus\Omega\to\C$ is solution of the exterior Dirichlet problem
\begin{equation}
\begin{array}{rclll}
\Delta u_D+k^2 u_D &=& 0&\mbox{in}& \R^2\setminus\overline\Omega,\medskip\\
 u_D &=& - u^\inc&\mbox{on}&\Gamma,\medskip\\
\multicolumn{1}{c}{\displaystyle\lim_{|\ner|\to\infty}\sqrt{|\ner|}\left(\frac{\p u_D}{\p |\ner|}-iku_D\right) }&=&0,
\end{array}\label{eq:DEP}
\end{equation}
where  $k\in \C$, $\imag k\geq 0$, $\real k>0$ denotes the wavenumber of the unbounded medium surrounding~$\Omega$. In TM-polarization, on the other hand,  the scattered field $u_N:\R^2\setminus\Omega\to\C$  is solution the exterior Neumann problem
\begin{equation}
\begin{array}{rclll}
\Delta u_N+k^2 u_N &=& 0&\mbox{in}& \R^2\setminus\overline\Omega,\medskip\\
\displaystyle\p_n u_N &=& -\p_n u^\inc&\mbox{on}&\Gamma,\medskip\\
\multicolumn{1}{c}{\displaystyle\lim_{|\ner|\to\infty}\sqrt{|\ner|}\left(\frac{\p u_N}{\p |\ner|}-iku_N\right)}&=&0,
\end{array}\label{eq:NEP}
\end{equation}
where the symbol $\p_n$ in~\eqref{eq:NEP} denotes the exterior normal derivative on the boundary $\Gamma$.  
 
 As is well-known~(cf.~\cite{COLTON:1983}) for a continuous boundary data the exterior Dirichlet~\eqref{eq:DEP} (resp. Neumann ~\eqref{eq:NEP}) admits a unique solution $u_D$ (resp. $u_N$) for all wavenumbers $k\in\C$, $\real k>0$, $\imag k\geq 0$.  
 
Given a density function $\varphi:\Gamma\to\C$ we define  the single- and double-layer potentials as
\begin{equation}
\mathcal S[\varphi](\ner) = \int_\Gamma G(\ner,\ney)\varphi(\ney)\de s({\ney})\andtext \mathcal D[\varphi](\ner) = \int_{\Gamma}\frac{\p G(\ner,\ney)}{\p \nor(\ney)}\varphi(\ney)\de s({\ney}),\quad \ner\in\R^2\setminus\Gamma,\label{eq:pot}
\end{equation}
respectively, where $G(\ner,\ney)=\frac{i}{4}H_0^{(1)}(k|\ner-\ney|)$ is the free-space Green function for the Helmholtz equation. Evaluation of the layer potentials~\eqref{eq:pot} and their  exterior normal derivatives on $\Gamma$ yields the jump relations~\cite{COLTON:1983}
\begin{equation}\begin{split}
\mathcal S[\varphi] = S[\varphi],\quad \p_n \mathcal S[\varphi] = -\frac{\varphi}{2} +K'[\varphi],\quad
\mathcal D[\varphi] = \frac{\varphi}{2}+K[\varphi]\andtext \p_n \mathcal D[\varphi] = N[\varphi],\\
\end{split}\label{eq:jump}
 \end{equation}
 which are expressed in terms of the single-layer ($S$), double-layer ($K$), adjoint double-layer $(K')$ and hypersingular $(H)$  operators. These operators are given by the integral expressions
 \begin{subequations}\begin{eqnarray}
S[\varphi](\nex) &=& \int_{\Gamma} G(\nex,\ney)\varphi(\ney)\de s(\ney),\label{eq:single} \\
K[\varphi](\nex) &=& \int_{\Gamma} \frac{\p G(\nex,\ney)}{\p\nor(\ney)}\varphi(\ney)\de s(\ney),\label{eq:double}\\
K'[\varphi](\nex) &=& \int_{\Gamma} \frac{\p G(\nex,\ney)}{\p\nor(\nex)}\varphi(\ney)\de s(\ney),\label{eq:adouble}\\
N[\varphi](\nex) &=& \mathrm{f.p.}\int_{\Gamma} \frac{\p^2 G(\nex,\ney)}{\p\nor(\nex)\p\nor(\ney)}\varphi(\ner)\de s(\ney),\label{eq:hyper}
\end{eqnarray}\label{eq:int_op}\end{subequations}
for  $\nex\in \Gamma$.
Note that the integral in the definition of the hypersingular operator~\eqref{eq:hyper} must be understood as a Hadamard finite-part integral which, upon integration by parts, can expressed as
\begin{equation}
N[\varphi](\nex) = k^2\int_\Gamma G(\nex,\ney)\nor(\nex)\cdot\nor(\ney)\varphi(\ney)\de s(\ney) + \mathrm{p.v.}\int_{\Gamma}\p_s G(\nex,\ney)\p_s\varphi(\ney)\de s(\ney)\label{eq:Maue}
\end{equation}
in terms of a Cauchy principal value integral and the tangential derivative $\p_{s}$ of the surface density~$\varphi$ on~$\Gamma$. The expression~\eqref{eq:Maue}  is sometimes called Maue's integration by parts formula~\cite{Maue:1949in} and can be interpreted as regularization of the integral operator in the sense that it involves integrands are ``smoother" (see Section~\ref{sec:intro}).

We assume, for the time being,  that the boundary of the PEC obstacle $\Gamma=\p\Omega$ admits a real analytic $2\pi$-periodic parametric representation
\begin{equation}\Gamma=\{\bold x(t): t\in[0, 2\pi)\},\label{eq:param}\end{equation}
where $|\bold x'(t)|\neq 0$ for all $t\in[0,2\pi)$. (This smoothness assumption on $\Gamma$ is relaxed in Section~\ref{sec:corners} where numerical examples for domains with corners are considered.) Utilizing  the boundary parameterization we define the Sobolev space $H^s(\Gamma)$, $s>0$, as the space of functions $\varphi\in L^2(\Gamma)$ such that $\varphi\circ \bold x\in H^s[0,2\pi]$, where $H^s[0,2\pi] = \{v\in L^2[0,2\pi]:\|v\|_s<\infty\}$ (see~\cite[Chapter 8]{kress2012linear} or~\cite[Section 5.3]{saranen2013periodic} for a more detailed definition of this space and its properties). In particular, the  integral operators~\eqref{eq:int_op}: $S:H^{s}(\Gamma)\to H^{s+1}(\Gamma)$, $K',K:H^{s}(\Gamma)\to H^{s+3}(\Gamma)$ and $N:H^{s}(\Gamma)\to H^{s-1}(\Gamma)$ are continuous for all $s>0$~\cite{Mclean2000Strongly}.

In order to solve the exterior boundary value problems we look for solutions given by the combined double- and single-layer potential~\cite{brakhage1965dirichletsche,leis1965dirichletschen,panish1965question} 
\begin{equation}
u(\ner) =(\mathcal D-i\eta\mathcal S)[\varphi](\ner), \qquad \ner\in\R^2\setminus\overline\Omega,\label{eq:CFP}
\end{equation}
 where the density function~$\varphi$ can be determined by matching the potential $u$ or its normal derivative $\p_n u$ with the appropriate boundary data on $\Gamma$. 
 
Therefore, in the case of Dirichlet problem~\eqref{eq:DEP} we  obtain the Dirichlet Combined Field Integral Equation (D-CFIE)
\begin{equation}
\left(\frac{I}{2}+K-i\eta S\right)[\varphi] = -u^\inc\quad \mbox{on}\quad \Gamma\label{eq:IE_Dirichlet}
\end{equation}
for an unknown density function $\varphi$, which was obtained evaluating  the potential~\eqref{eq:CFP} on~$\Gamma$ using the jump conditions~\eqref{eq:jump}.   Since $K$ and $S$ are compact operators on $H^s(\Gamma)$ we have that the combined field integral operator $K-i\eta S$ is also a compact operator on $H^s(\Gamma)$, and thus~\eqref{eq:IE_Dirichlet} is a Fredholm integral equation of the second second-kind.  Therefore, the well-posedness of the D-CFIE follows from the fact that \eqref{eq:IE_Dirichlet} admits at most one solution for all wavenumbers $\real k>0$, $\imag k\geq 0$, provided $\eta>0$~\cite[Theorem 3.3]{COLTON:1983}.

In the case of the Neumann problem~\eqref{eq:NEP}, on the other hand, evaluation of the normal derivative of the combined potential~\eqref{eq:CFP} on $\Gamma$ yields  the Neumann Combined Field Integral Equation (N-CFIE) 
\begin{equation}
\left(\frac{i\eta }{2}I+N-i\eta K'\right)[\varphi] = -\p_nu^\inc\quad \mbox{on}\quad \Gamma,\label{eq:IE_Neumann}
\end{equation}
for an unknown density function $\varphi$. Here,  $\frac{i\eta }{2}I+N-i\eta K':H^{s}\to H^{s-1}$ is not a compact operator on $H^s(\Gamma)$. As is known, however,  the N-CFIE admits at most one solution $\varphi$ for all wavenumbers $\real k>0$, $\imag k\geq 0$, provided $\eta>0$~\cite[Theorem 3.34]{COLTON:1983}.

\section{Smoothed combined field integral equation formulations}\label{eq:reg_IE}
In this section we introduce  smoothed  versions of the combined potential~\eqref{eq:CFP} and associated CFIEs~\eqref{eq:IE_Dirichlet} and~\eqref{eq:IE_Neumann}.

\subsection{Dirichlet problem}
Let us assume that we are given two  $C^\infty$-smooth functions $p_0(\:\cdot\:|\nex_0)$ and $p_1(\:\cdot\:|\nex_0)$ that satisfy
\begin{equation}\begin{gathered}
\Delta p_0(\ner|\nex_0) + k^2 p_0(\ner|\nex_0) =0, \qquad \ner\in\R^2, \\
p_0(\ner|\nex_0) =1,\quad  \p_n p_0(\ner|\nex_0) = i\eta,\quad \p_s p_0(\ner|\nex_0) =0\quad\mbox{and }\quad \p_s\p_n p_0(\ner|\nex_0) = 0\quad\mbox{at}\quad \ner=\nex_0,\label{eq:funct_p}\end{gathered}
\end{equation}
and
\begin{equation}\begin{gathered}
\Delta p_1(\ner|\nex_0) + k^2 p_1(\ner|\nex_0) =0, \qquad \ner\in\R^2, \\
p_1(\ner|\nex_0) =0,\quad  \p_n p_1(\ner|\nex_0) = 0,\quad \p_s p_1(\ner|\nex_0) =1\quad\mbox{and }\quad \p_s\p_n p_1(\ner|\nex_0) = i\eta\quad\mbox{at}\quad \ner=\nex_0,\label{eq:funct_q}\end{gathered}
\end{equation}
respectively, where $\nex_0$ is given a point on the boundary $\Gamma$.  Assume further that both functions $p_0(\:\cdot\:|\nex_0)$ and $p_1(\:\cdot\:|\nex_0)$ are given by certain linear combinations of plane waves. (Expressions for such functions are given in Section~\ref{sec:reg_functions}.) Therefore, Green's third identity together with standard stationary phase arguments yield the relations
\begin{subequations}\begin{eqnarray}
\mathcal D\lf[p_0(\:\cdot\:|\nex_0)\rg](\ner)-\mathcal S\left[\p_n p_0(\:\cdot\:|\nex_0)\right](\ner)&=&0,\qquad \ner\in\R^2\setminus\overline\Omega,\label{eq:GP}\\
\mathcal D\lf[p_1(\:\cdot\:|\nex_0)\rg](\ner)-\mathcal S\left[\p_n p_1(\:\cdot\:|\nex_0)\right](\ner)&=&0,\qquad \ner\in\R^2\setminus\overline\Omega.\label{eq:GQ}
\end{eqnarray}\end{subequations}


Multiplying~\eqref{eq:GP} and~\eqref{eq:GQ} through by $\varphi(\nex_0)$ and~$\p_s\varphi(\nex_0)$, respectively,  and  subtracting the resulting expressions from the combined potential~\eqref{eq:CFP},  we obtain 
\begin{equation}\begin{split}
u(\ner) =&\ \mathcal D\left[\varphi-\varphi(\nex_0)p_0(\:\cdot\:|\nex_0)-\p_s\varphi(\nex_0)p_1(\:\cdot\:|\nex_0)\right](\ner) \\
&\ -\mathcal S\lf[i\eta\varphi-\varphi(\nex_0)\p_np_0(\:\cdot\:|\nex_0)-\p_s\varphi(\nex_0)\p_np_1(\:\cdot\:|\nex_0)\rg](\ner),\quad \ner\in\R^2\setminus\overline\Omega,\ \nex_0\in\Gamma.\label{eq:mod_rep_formula}\end{split}
\end{equation}
Evaluating the potential~\eqref{eq:mod_rep_formula} on  $\Gamma$, using the jump conditions~\eqref{eq:jump}, we  get  
\begin{equation*}
\begin{split}
u(\nex) =&\ \frac{\varphi(\nex)-\varphi(\nex_0)p_0(\nex|\nex_0)-\p_s\varphi(\nex_0)p_1(\nex|\nex_0)}{2}+ K\left[\varphi-\varphi(\nex_0)p_0(\:\cdot\:|\nex_0)-\p_s\varphi(\nex_0)p_1(\:\cdot\:|\nex_0)\right](\nex)\\
&- S\lf[i\eta\varphi-\varphi(\nex_0)\p_np_0(\:\cdot\:|\nex_0)-\p_s\varphi(\nex_0)\p_np_1(\:\cdot\:|\nex_0)\rg](\nex)\quad\mbox{for all }\quad \nex,\nex_0\in\Gamma,
\end{split}
\end{equation*}
where the operators $K$ and $S$ are defined in~\eqref{eq:double} and~\eqref{eq:single}, respectively. 
Therefore, selecting $\nex_0=\nex$ in the relation above and using the identities $p_0(\nex|\nex)=1$ and $p_1(\nex|\nex)=0$,  we obtain  the Smoothed Dirichlet  Combined Field Integral Equation (SD-CFIE)
\begin{equation}
\left(K\circ R_D- S\circ R_S\right)[\varphi] = -u^\inc\quad \mbox{on}\quad \Gamma,\label{eq:IE_SDirichlet}
\end{equation}
for the unknown density function $\varphi$. Here,  the operators $R_D$ and $R_S$ are explicitly defined in terms of the smoothing functions $p_0$ and $p_1$  by	
\begin{equation}\begin{split}
R_D[\varphi|\nex](\ney)=&\ \varphi(\ney)-\varphi(\nex)p_0(\ney|\nex)-\p_s\varphi(\nex)p_1(\ney|\nex)\qquad\andtext \\
R_S[\varphi|\nex](\ney)=&\ i\eta\varphi(\ney)-\varphi(\nex)\p _n p_0(\ney|\nex)-\p_s\varphi(\nex)\p _n p_1(\ney|\nex).\label{eq:reg_op}
\end{split}
\end{equation}
The following lemma  establishes the essential property of $R_D$ and $R_S$:

\begin{lemma}\label{eq:lem_important}
Given $\nex\in\Gamma$, the operators  $R_D[\:\cdot\:|\nex]$ $R_S[\:\cdot\:|\nex]:H^s(\Gamma)\to H^{s}(\Gamma)$ introduced in~\eqref{eq:reg_op}, are well defined for all $s>3/2$. Furthermore, they satisfy 
\begin{equation}
R_D[\varphi|\nex](\ney) = O(|\nex-\ney|^2)\andtext R_S[\varphi|\nex](\ney) = O(|\nex-\ney|^2)\quad\mbox{as}\quad\ney\to\nex,\ \ (\ney\in\Gamma),\label{eq:estimate}\end{equation}
for all $s>5/2$.
\end{lemma}

\begin{proof}
Clearly  $R_D$ and $R_S$ admit the representations $R_D=I- p_0(\,\cdot\,|\nex)\delta_{\nex}-p_1(\,\cdot\,|\nex)\delta'_{\nex}$ and $R_S=i\eta I-\p _n p_0(\,\cdot\,|\nex)\delta_{\nex}-\p _n p_1(\,\cdot\,|\nex)\delta'_{\nex}$ in terms of the Dirac's distribution $\delta_{\nex}$ (supported at $\nex\in\Gamma$) and its derivative, both of which belong to $H^{-s}(\Gamma)$ for all $s>3/2$. Since $p_0(\:\cdot\:|\nex),\p_n p_0(\:\cdot\:|\nex)\in H^{s}(\Gamma)$ for all $s\in\R$, on the other hand,  it readily follows that $R^{(1)}_D[\varphi|\nex],R^{(1)}_S[\varphi|\nex]\in H^{s}(\Gamma)$, $s>3/2$, for any given point  $\nex\in\Gamma$.
 
We now prove  the asymptotic identities in~\eqref{eq:estimate}. Let  the parameter values $t,\tau\in[0,2 \pi)$ be such that $\nex=\bold x(t)$ and $\ney=\bold x(\tau)$, where $\bold x:[0,2\pi)\to\Gamma$ denotes the parametrization of the smooth curve~$\Gamma$, and let $\phi =\varphi\circ\bold x$ where $\varphi\in H^s(\Gamma)$, $s>5/2.$ Since for $s>5/2$ the density $\phi$ is a twice-continuously differentiable $2\pi$-periodic function on $[0,2\pi)$~\cite[Lemma 5.3.3]{saranen2013periodic}, we have that 
\begin{equation}
 \rho_D(\tau|t)=R_D[\varphi|\bold x(t)](\bold x(\tau))=\phi(\tau)-\phi(t)\tilde p_0(\tau|t)-|\bold x'(t)|^{-1}\phi'(t)\tilde p_1(\tau|t),\label{eq:param_1}\end{equation} with $\tilde p_0(\tau|t)=p_0(\bold x(\tau)|\bold x(t))$ and $\tilde p_1(\tau|t)=p_1(\bold x(\tau)|\bold x(t)),$
 is also a twice-continuously differentiable 2$\pi$-periodic function. Therefore, expansing $\rho_D(\tau|t)$ as a Taylor series around $\tau=t$, we get
\begin{equation}\begin{split}
\rho_D(\tau|t) =&\ \phi(t)+\phi'(t)(\tau-t)+\frac{\phi''(t)}{2}(\tau-t)^2+o(|\tau-t|^2)\\
&-\phi(t)\{ 1+O(|\tau-t|^2)\}
-\phi'(t)\{\tau-t+O(|\tau-t|^2)\}
=O(|\tau-t|^2)\ \mbox{ as }\  \tau\to t,
\end{split}\label{eq:taylor}\end{equation}
where we have utilized the identities
\begin{equation*}\begin{split}
\tilde p_0(t|t) =  p_0(\bold x(t)|\bold x(t)) =1,\qquad&\tilde p'(t|t) = |\bold x'(t)|\p_s p_0(\bold x(t)|\bold x(t)) =0,\\
\tilde p_1(t|t) =  p_1(\bold x(t)|\bold x(t)) =0,\qquad&\tilde q'(t|t) = |\bold x'(t)|\p_s p_1(\bold x(t)|\bold x(t)) =|\bold x'(t)|,
\end{split}\end{equation*}
that follow from the point conditions in~\eqref{eq:funct_p} and~\eqref{eq:funct_q} satisfied by  $p$ and $q$, respectively.

Similarly, using the identities 
\begin{equation*}\begin{split}
\p_n\tilde p_0(t|t) = \p_n p_0(\bold x(t)|\bold x(t)) =i\eta,\qquad&\p_n\tilde p'_0(t|t) = |\bold x'(t)|\p_s\p_n p_0(\bold x(t)|\bold x(t)) =0,\\
\p_n\tilde p_1(t|t) =  \p_np_1(\bold x(t)|\bold x(t)) =0,\qquad&\p_n\tilde p'_1(t|t) = |\bold x'(t)|\p_s \p_np_1(\bold x(t)|\bold x(t)) =i\eta |\bold x'(t)|,
\end{split}\end{equation*}  where 
$\p_n\tilde p_0(\tau|t)=\p_n p_0(\bold x(\tau)|\bold x(t))$ and  $\p_n\tilde p_1(\tau|t)=\p_n p_1(\bold x(\tau)|\bold x(t)),$
 it can be shown that  the function
 \begin{equation}\rho_S(\tau|t)=R_S[\varphi|\bold x(t)](\bold x(\tau))=i\eta\phi(\tau)-\phi(t)\p_n\tilde p_0(\tau|t)-|\bold x'(t)|^{-1}\phi'(t)\p_n\tilde p_1(\tau|t) \label{eq:param_2}\end{equation} satisfies
 \begin{equation}
\rho_S(\tau|t) =O(|\tau-t|^2)\quad\mbox{as}\quad \tau\to t.
\end{equation}
Therefore, finally, the identities in~\eqref{eq:estimate} follow from the fact that $O(|\bold x(t)-\bold x(\tau)|^2) = O(|\tau-t|^2)$.  The proof is now complete.
  \end{proof}
  
In order to illustrate the result of Lemma~\ref{eq:lem_important} we present  Figure~\ref{fig:reg_functions}   which displays the functions $\rho_D(\tau|t)$ and $\rho_S(\tau|t)$ defined in~\eqref{eq:param_1} and~\eqref{eq:param_2}, respectively, obtained by application of  $R_D$ and $R_S$ to a certain smooth density function~$\varphi$.  As can be observe in this figure, both functions $\rho_D(\tau|t)$ and $\rho_S(\tau|t)$ vanish quadratically along the line $\tau=t$.
\begin{figure}[h!]
\centering	
\includegraphics[scale=0.57]{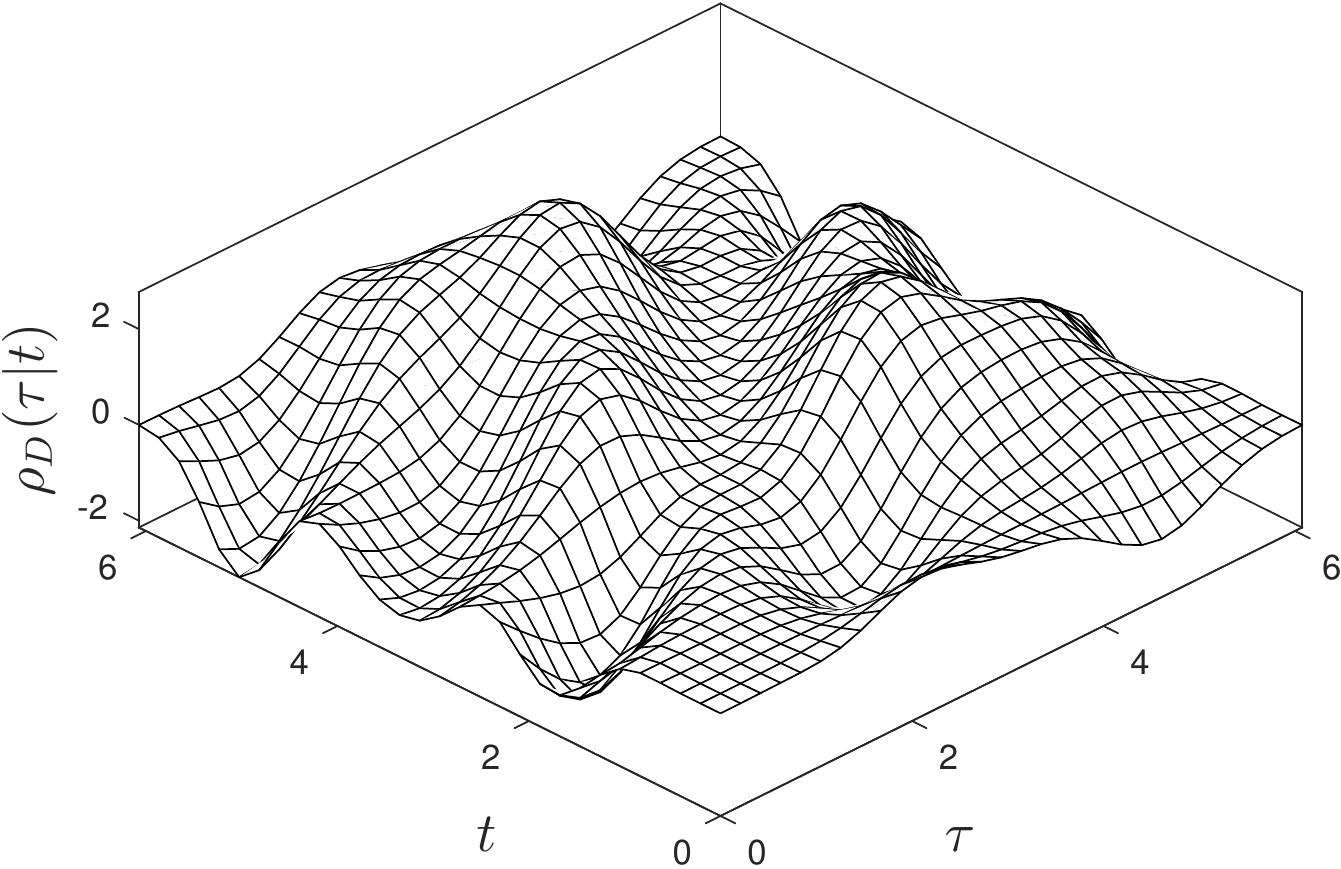}\quad
\includegraphics[scale=0.57]{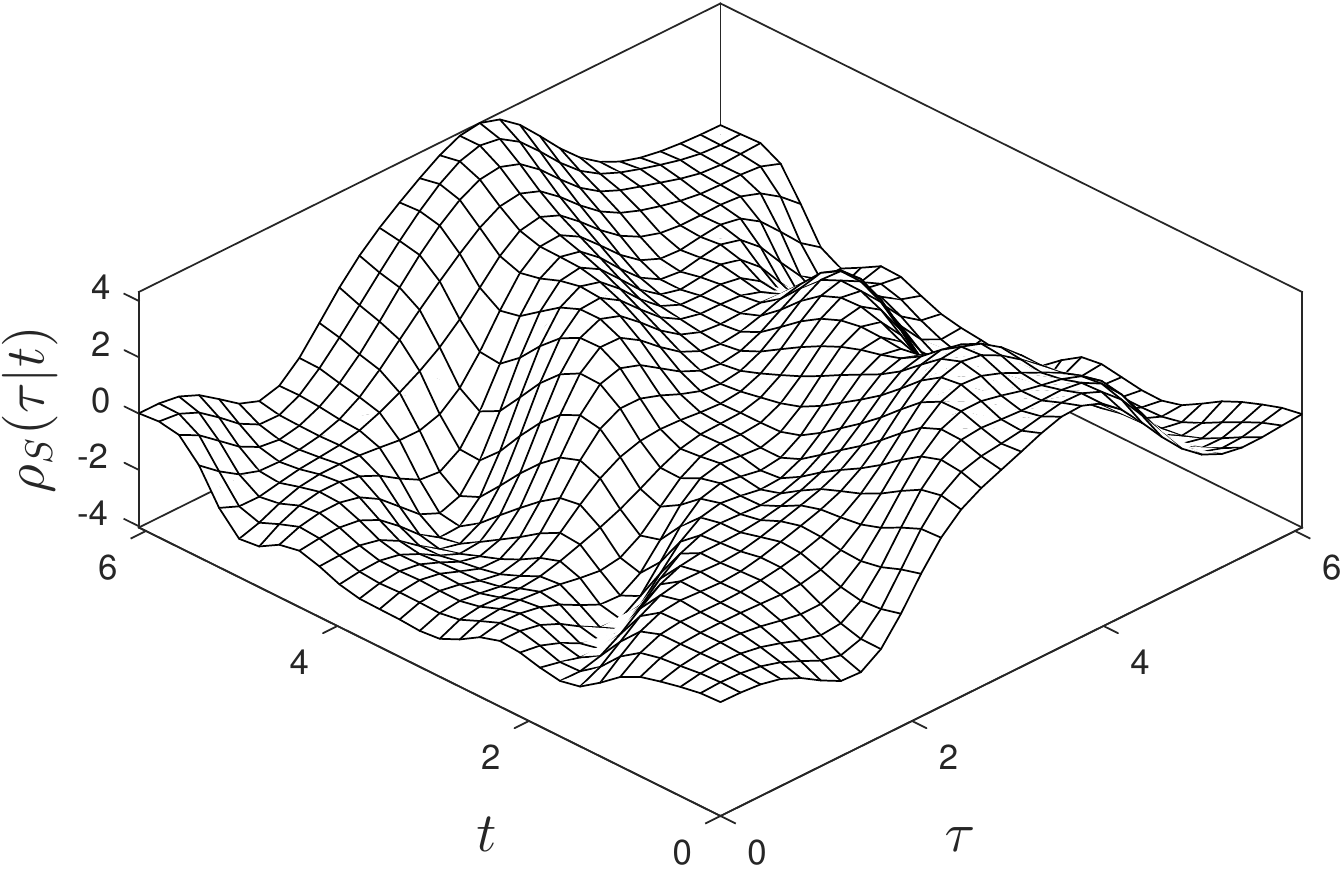}
\caption{Real part of the functions  $\rho_D(\tau|t)=R_D[\varphi|\bold x(t)](\bold x(\tau))$  (left)  and $\rho_S(\tau|t)=R_S[\varphi|\bold x(t)](\bold x(\tau))$  (right), where $\varphi(\bnex(t)) = \e^{ik\bnex(t)\cdot(\cos\pi/8,\sin\pi/8)}$,   $\Gamma=\{\bold x(t)=(\cos t,\sin t),t\in[0,2\pi)\}$ and $k=\eta = 2$.}\label{fig:reg_functions}
\end{figure}  
  
  
We are now in position to study the singular character of integral operator $K\circ R_D-S\circ R_S$ in the SD-CFIE~\eqref{eq:IE_SDirichlet}. As in the proof of Lemma~\ref{eq:lem_important} we utilize the parametrization $\bold x$ of the smooth curve  $\Gamma$ to define the $2\pi$-periodic function $\phi(t) = \varphi(\bold x(t))$ as well as the $2\pi$-biperiodic functions  $\rho_D(\tau|t)$ and $\rho_S(\tau|t)$ given  in~\eqref{eq:param_1} and~\eqref{eq:param_2}. Using these notations and letting $R=R(t,\tau)=|\bnex(t)-\bnex(\tau)|$ and  $\bold n(\tau)=(-x'_2(\tau),x'_1(\tau))/|\bold x'(\tau)|$ we have that  $v(t)=(K\circ R_D-S\circ R_S)[\varphi](\bnex(t))$ can be expressed as 
\begin{equation}
v(t) = \int_{0}^{2\pi}\lf\{L(t,\tau)\rho_D(\tau|t)-M(t,\tau)\rho_S(\tau|t)\rg\}\de\tau,\quad t\in[0,2\pi),\label{eq:op_param_Dir}
\end{equation}
in terms of the weakly-singular  kernels
\begin{equation}\begin{split}
L(t,\tau)=&\ \frac{ik}{4} \frac{H_1^{(1)}(kR)}{R}(\bold x(t)-\bold x(\tau))\cdot \bold n(\tau)|\bold x'(\tau)|=L_1(t,\tau)\log|t-\tau|+L_2(t,\tau),\\
M(t,\tau)=&\ \frac{i}{4} H_0^{(1)}(kR)|\bold x'(\tau)|= M_1(t,\tau)\log|t-\tau|+M_2(t,\tau),
\end{split}\label{eq:kernels_dirich}
\end{equation}
where the functions 
$$L_1(t,\tau) := -\frac{k}{2\pi}\frac{J_1(kR)}{R}(\bold x(t)-\bold x(\tau))\cdot \bold n(\tau)|\bold x'(\tau)|,\quad L_2(t,\tau):=L(t.\tau)-L_1(t,\tau)\log|t-\tau|$$
$$M_1(t,\tau) := -\frac{1}{2\pi}J_0(kR)|\bold x'(\tau)|\andtext M_2(t,\tau):=M(t.\tau)-M_1(t,\tau)\log|t-\tau|$$
can be properly defined at $\tau=t$ so that they are in fact $2\pi$-biperiodic analytic functions~\cite{COLTON:2012}. 

It thus follows  from~\eqref{eq:kernels_dirich}, Lemma~\ref{eq:lem_important}  and the fact that $L(t,\tau)=O(|t-\tau|\log|t-\tau|)$, that the integrands in~\eqref{eq:op_param_Dir}  satisfy
\begin{equation}\begin{split}
L(t,\tau) \rho_D(\tau|t) =  O(|t-\tau|^3\log|t-\tau|)\mbox{ and }M(t,\tau)\rho_S(\tau|t) =  O(|t-\tau|^2\log|t-\tau|)\ \mbox{ as }\ \tau\to t
\end{split}\label{eq:sing_Dirichlet}
\end{equation}
for $\varphi\in H^s(\Gamma)$ for $s>5/2$. As it turns out, it can be easily shown  that for $s>5/2$ the integrands in~\eqref{eq:op_param_Dir} are indeed continuously differentiable $2\pi$-biperiodic  functions (in $t$ and $\tau$).  

To illustrate the smoothness of $L(t,\tau) \rho_D(\tau|t)$ and $M(t,\tau) \rho_S(\tau|t)$ along the line $\tau=t$ we present Figure~\ref{fig:reg_integrands} which displays both functions for a given smooth density function $\varphi$.
\begin{figure}[h!]
\centering	
\includegraphics[scale=0.57]{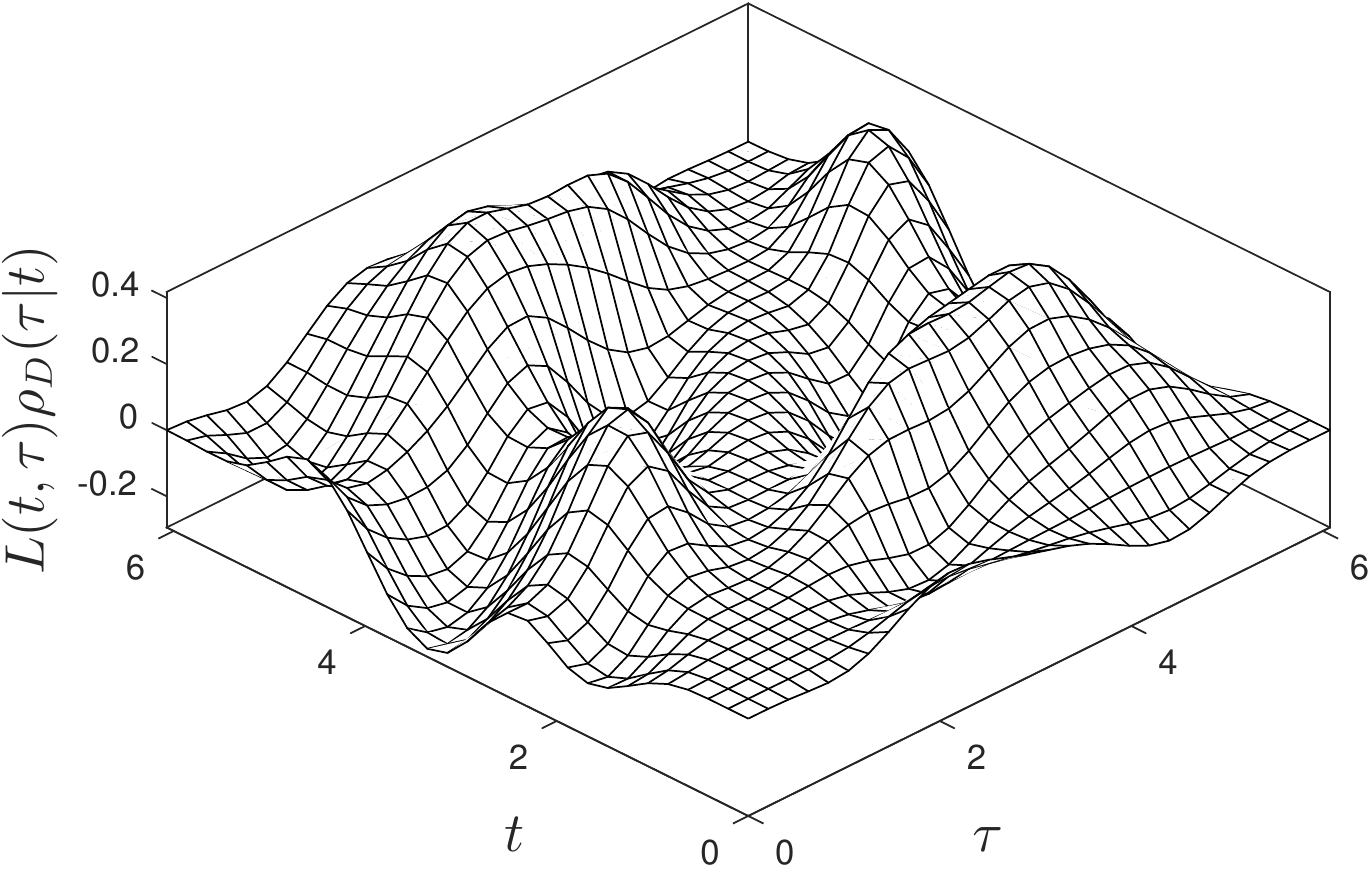}\quad
\includegraphics[scale=0.57]{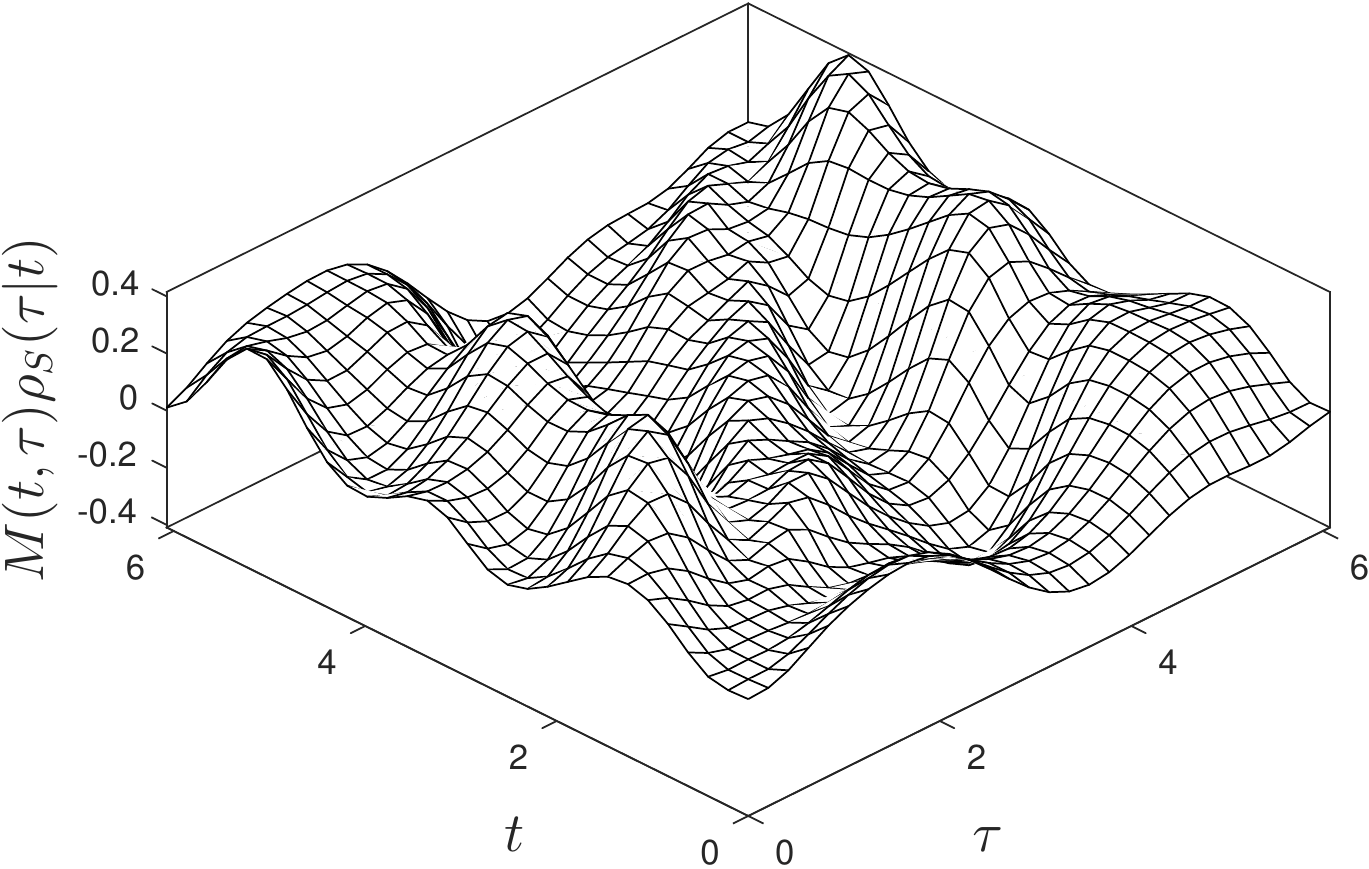}
\caption{Real part of the functions  $L(t,\tau)\rho_D(\tau|t)$  (left)  and $M(t,\tau)\rho_S(\tau|t)$  (right) for $\varphi(\bnex(t))=\phi(t) = \e^{ik\cos(t-\pi/8)}$, $\Gamma=\{\bold x(t)=(\cos t,\sin t), t\in[0,2\pi)\}$ and  $k=\eta = 2$.}\label{fig:reg_integrands}
\end{figure}  

We finish this section with Theorem~\ref{lem:equal} whose proof follows directly from the discussion above:
\begin{theorem}\label{lem:equal} Let $S$ and $K$ be the single- and double-layer operators  defined in~\eqref{eq:single} and~\eqref{eq:double}, respectively, and $R_D$ and $R_S$ be the smoothing operators defined in~\eqref{eq:reg_op}. Then, the identity 
$$\frac{I}{2}+K-i\eta S=K\circ R_D- S\circ R_S$$ holds true on $H^s(\Gamma)$ for all $s>3/2$. Therefore, in particular, the SD-CFIE~\eqref{eq:IE_SDirichlet} is uniquely solvable on $H^s(\Gamma)$, $s>3/2$ for all wavenumbers $k\in\C$, $\real k>0$, $\imag k\geq 0$, provided $u^\inc|_{\Gamma}\in H^s(\Gamma)$ and~$\eta>0$.
\end{theorem}
\subsection{Neumann problem} 
We now proceed to derive the  smoothed integral equation  for exterior Neumann problem~\eqref{eq:NEP}. 
 Evaluating the normal derivative  of the potential~\eqref{eq:mod_rep_formula} on~$\Gamma$ and using, once again, the jump conditions~\eqref{eq:jump}, we obtain 
\begin{equation*}
\begin{split}
\p_n u(\nex) =&\ -\frac{i\eta\varphi(\nex)-\varphi(\nex_0)\p_np_0(\nex|\nex_0)-\p_s\varphi(\nex_0)\p_np_1(\nex|\nex_0)}{2}\\
&\ +N\left[\varphi-\varphi(\nex_0)p_0(\:\cdot\:|\nex_0)-\p_s\varphi(\nex_0)p_1(\:\cdot\:|\nex_0)\right](\nex)\\
&- K'\lf[i\eta\varphi-\varphi(\nex_0)\p_np_0(\:\cdot\:|\nex_0)-\p_s\varphi(\nex_0)\p_np_1(\:\cdot\:|\nex_0)\rg](\nex),\quad \nex,\nex_0\in\Gamma,
\end{split}
\end{equation*}
where the integral  operators $N$ and $K'$ are defined in~\eqref{eq:adouble} and~\eqref{eq:hyper}, respectively. Therefore, 
selecting $\nex_0=\nex$  and using the fact that  $\p_n p_0(\nex|\nex)=i\eta$ and $\p_np_1(\nex|\nex)=0$, we obtain the Smoothed Neumann Combined Field Integral Equation (SN-CFIE) 
\begin{equation}
\lf(N\circ R_D-K'\circ R_S\rg)[\varphi] = -\p_n u^\inc\quad\mbox{on}\quad\Gamma,\label{eq:IE_SNeumann}
\end{equation}
for the unknown density function $\varphi$.

Let us now examine the singular character of the integral operator $N\circ R_D-K'\circ R_S$ in the SN-CFIE~\eqref{eq:IE_SNeumann}. Using the notations $\rho_D(\tau|t)$ and $\rho_S(\tau|t)$ introduced in~\eqref{eq:param_1} and~\eqref{eq:param_2},  we have that $v(t) =(N\circ R_D-K'\circ R_S)[\varphi](\bnex(t))$   can be expressed as
\begin{equation}
v(t) =  \int_{0}^{2\pi}\lf\{H(t,\tau)\rho_D(\tau|t)-W(t,\tau)\rho_S(\tau|t)\rg\}\de\tau,\quad t\in[0,2\pi),\label{eq:op_param_Neu}
\end{equation}
in terms of the integral kernels
\begin{equation}\begin{split}
H(t,\tau)=&\ \frac{ik}{4}\left\{kRH_0^{(1)}(kR)-2H_1^{(1)}(kR)\right\}\frac{(\bold x(t)-\bold x(\tau))\cdot \bold n(\tau)\ (\bold x(t)-\bold x(\tau))\cdot \bold n(t)}{R^3}|\bold x'(\tau)|\\
& +\frac{ik}{4}\frac{H_1^{(1)}(kR)}{R}|\bold x'(\tau)|\bold n(t)\cdot \bold n(\tau)=\frac{H_0(t,\tau)}{(t-\tau)^2}+H_1(t,\tau)\log(|t-\tau|) + H_2(t,\tau),\medskip\\
W(t,\tau)=&\ \frac{ik}{4} \frac{H_1^{(1)}(kR)}{R}(\bold x(\tau)-\bold x(t))\cdot \bold n(t)|\bold x'(\tau)|=W_1(t,\tau)\log|t-\tau|+W_2(t,\tau),
\end{split}\label{eq:kernels_neumann}
\end{equation}
where the functions
\begin{equation*}\begin{split}
H_0(t,\tau):=&\ \frac{1}{2\pi}\frac{(t-\tau)^2}{ R^2}\bold n(t)\cdot \bold n(\tau)|\bold x'(\tau)|,\\
 H_1(t,\tau)
:=&\ -\frac{k^2}{4\pi}\left\{J_0(kR)-2\frac{J_1(kR)}{kR}\right\}\frac{(\bold x(t)-\bold x(\tau))\cdot \bold n(\tau)\ (\bold x(t)-\bold x(\tau))\cdot \bold n(t)}{R^2}|\bold x'(\tau)|\\
& -\frac{k}{4\pi}\frac{J_1(kR)}{R}\bold n(t)\cdot \bold n(\tau)|\bold x'(\tau)|,\\
H_2(t,\tau):=&\ H(t,\tau)-\frac{H_0(t,\tau)}{(t-\tau)^2}+H_1(t,\tau)\log|t-\tau|,\\
W_1(t,\tau):=&\ \frac{k}{2\pi}\frac{J_1(kR)}{R}(\bold x(t)-\bold x(\tau))\cdot \bold n(t)|\bold x'(\tau)|\andtext W_2(t,\tau):=W(t,\tau)-W_1(t,\tau)\log|t-\tau|,
\end{split}\end{equation*}
can be properly defined at $\tau=t$ so that they are $2\pi$-biperiodic analytic functions~\cite{COLTON:2012}. 

 Therefore,  from~\eqref{eq:kernels_neumann}, Lemma~\ref{eq:lem_important} and the fact that $W(t,\tau)=O(|t-\tau|\log|t-\tau|)$, we obtain that  
\begin{equation}\begin{split}
H(t,\tau) \rho_D(\tau|t) =  O(1)+O(|t-\tau|^2\log|t-\tau|)\andtext W(t,\tau)\rho_S(\tau|t) =  O(|t-\tau|^3\log|t-\tau|)
\end{split}\label{eq:sing_Neumann}
\end{equation}
as $\tau\to t$ for $\varphi\in H^s(\Gamma)$ for $s>5/2$. 
 
 As in the case of the SD-CFIE, it can be easily shown that for $s>5/2$  the integrands in~\eqref{eq:op_param_Neu} are  continuously differentiable $2\pi$-biperiodic functions.
  
 Note  that, unfortunately, the diagonal values of the integrand in~\eqref{eq:op_param_Neu}, i.e., the limit values of the integrand as $\tau\to t$, depends on the second derivative of $\phi(t)=\varphi(\bnex(t))$ at $\tau=t$. More precisely, we have 
\begin{equation}
\lim_{\tau\to t}H(t,\tau)\rho_D(\tau|t) =H_0(t,t)\lim_{\tau\to t}\frac{\rho_D(\tau|t)}{(t-\tau)^2}= \frac{H_0(t,t)}{2}\left\{\phi''(t)-\phi(t)\tilde p_0''(t|t)-\frac{\phi'(t)}{|\bold x'(t)|}\tilde p_1''(t|t)\rg\},\label{eq:diag_term}
\end{equation}
where, as in the proof of Lemma~\ref{eq:lem_important}, we have used the notations $\tilde p_0(\tau|t)  =p_0(\bnex(\tau)|\bnex(t))$ and $\tilde p_0(\tau|t)  =p_1(\bnex(\tau)|\bnex(t))$.

To illustrate the smoothness of the integrand in the expression~\eqref{eq:op_param_Neu} we present Figures~\ref{fig:reg_integrands_Ne} which displays the functions  $H(t,\tau)\rho_D(\tau|t)$ and $W(t,\tau)\rho_S(\tau|t)$.
\begin{figure}[h!]
\centering	
\includegraphics[scale=0.57]{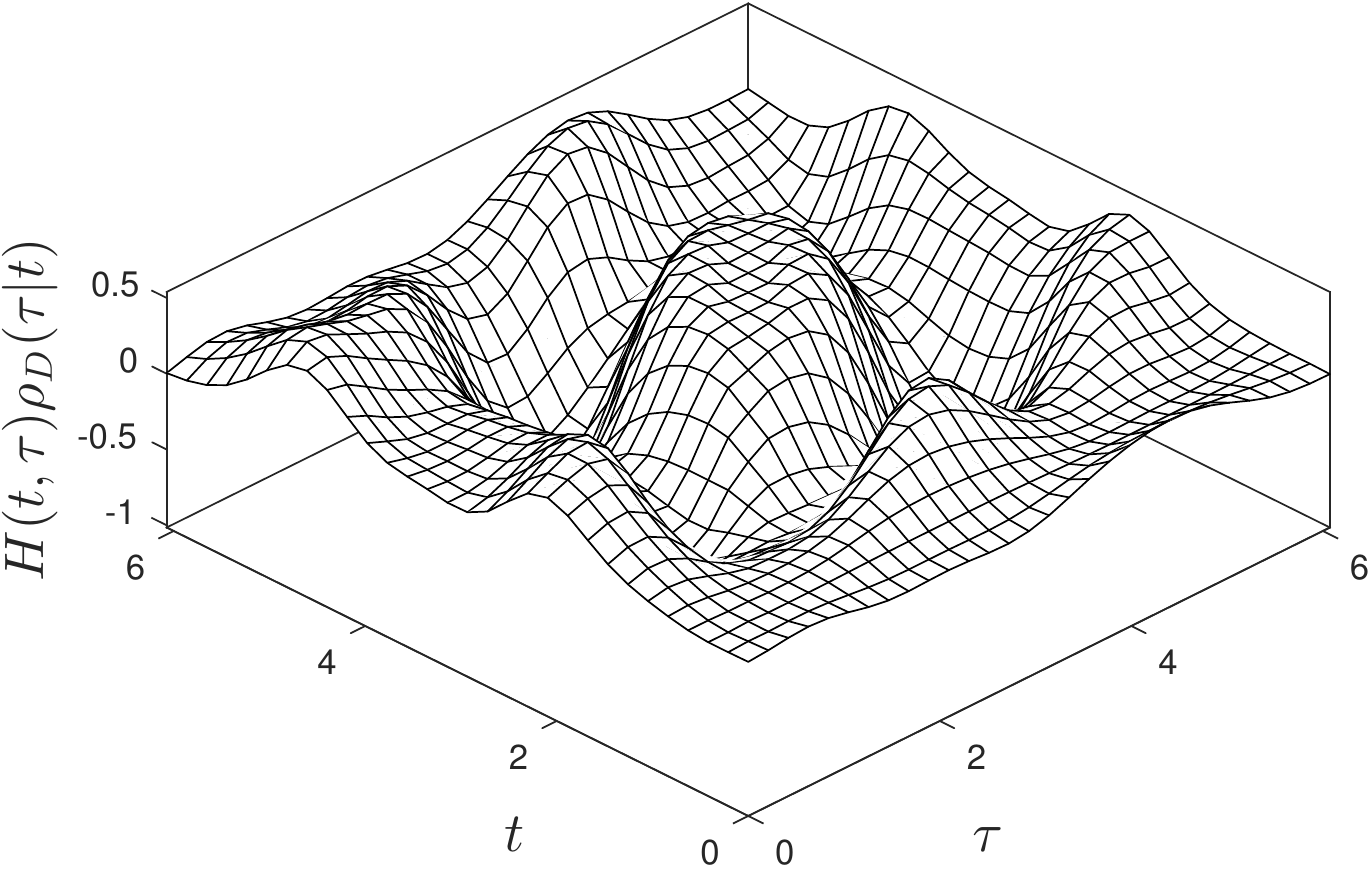}\quad
\includegraphics[scale=0.57]{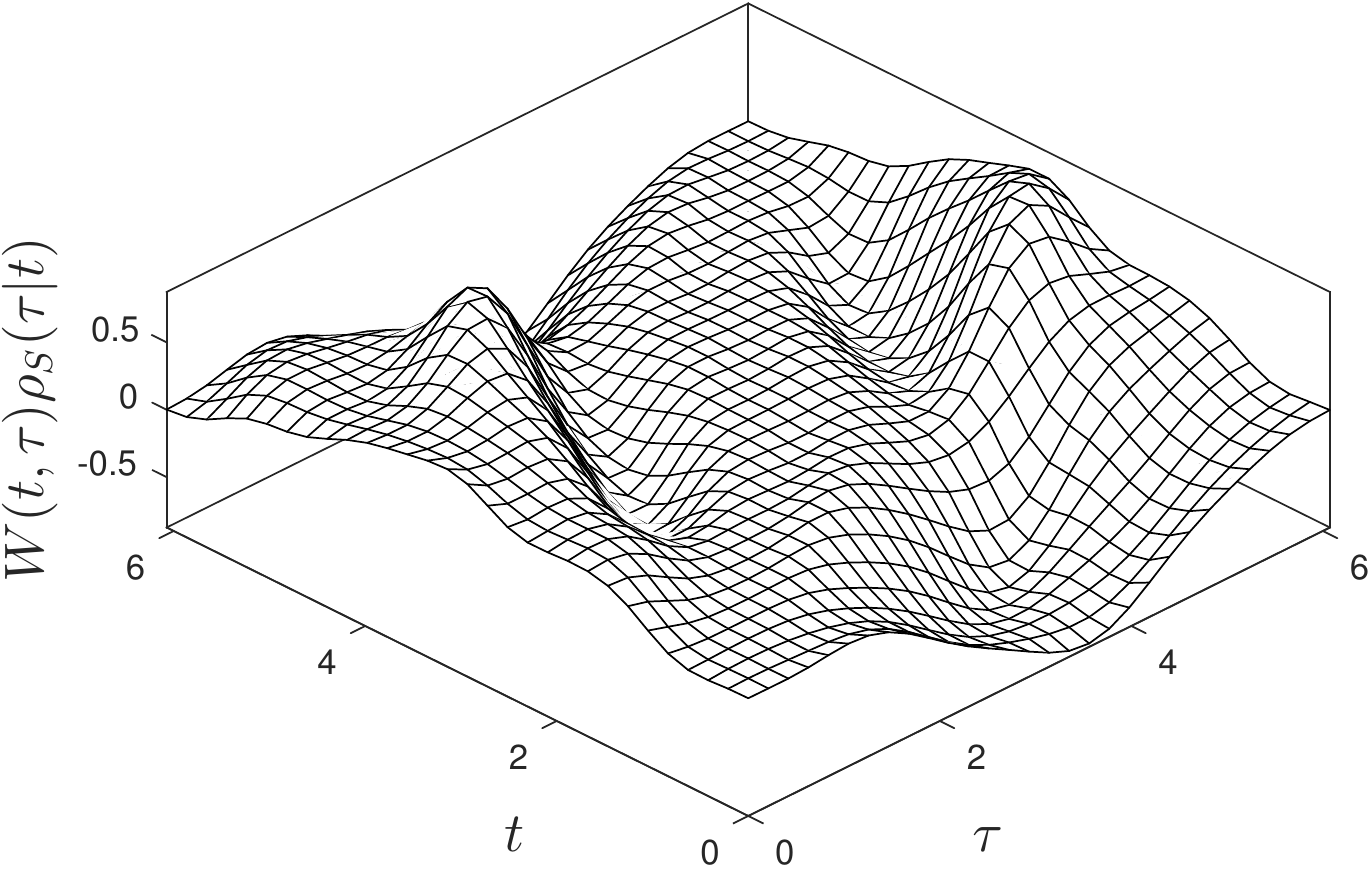}
\caption{Real part of the functions  $H(t,\tau)\rho_D(\tau|t)$  (left)  and $W(t,\tau)\rho_S(\tau|t)$ (right) for $\varphi(\bnex(t))=\phi(t) = \e^{ik\cos(t-\pi/8)}$, $\Gamma=\{\bold x(t)=(\cos t,\sin t), t\in[0,2\pi)\}$ and  $k=\eta = 2$.}\label{fig:reg_integrands_Ne}
\end{figure}  

We end this section with the following assertion that follows directly from the discussion above:
\begin{theorem}\label{lem:equal_N} Let $N$ and $K'$ be the hypersingular and adjoint double-layer operators defined in~\eqref{eq:single} and~\eqref{eq:double}, respectively, and $R_D$ and $R_S$ be the smoothing operators defined in~\eqref{eq:reg_op}. Then, the identity 
$$\frac{i\eta I}{2}+N-i\eta K'=N\circ R_D- K'\circ R_S$$ holds true on $H^s(\Gamma)$ for all $s>3/2$. Therefore, in particular, the SN-CFIE~\eqref{eq:IE_SDirichlet} admits at most one solution  $\varphi\in H^s(\Gamma)$, $s>3/2$, for all wavenumbers $k\in\C$, $\real k>0$, $\imag k\geq 0$, provided $\p_nu^\inc\in H^{s-1}(\Gamma)$ and $\eta>0$.
\end{theorem}

\begin{remark}
The hypersingular operator can be expressed as 
$$
N[\varphi](\nex) = N\lf[\varphi(\nex)-\varphi(\nex)p_0(\,\cdot\,|\nex)+\p_s\varphi(\nex)p_1(\,\cdot\,|\nex)\rg](\nex)+K'[\varphi(\nex)\p_np_0(\,\cdot\,|\nex)+\p_s\varphi(\nex)\p_np_1(\,\cdot\,|\nex)](\nex),
$$
for all $\varphi\in H^s(\Gamma)$, $s>5/2$,  and all  $\nex\in\Gamma$, in terms of the smoothing functions $p_0$ and  $p_1$ introduced~\eqref{eq:funct_p} and~\eqref{eq:funct_q}. It thus follows from the Lemma~\ref{eq:lem_important} and the smoothness of the kernel $W$ of the adjoint double-layer operator, that the hypersingular operator can be expressed in terms of integrals of continuous functions. 
\end{remark}

\subsection{Evaluation of combined potentials close to the boundary \label{sec:ner_field}}

Once the density function~$\varphi$ has been retrieved by solving the integral equation~\eqref{eq:IE_SDirichlet} or~\eqref{eq:IE_SNeumann}, depending on the boundary condition, the desired  solution of the corresponding boundary value problem~\eqref{eq:DEP} or~\eqref{eq:NEP} is given in terms of the combined potential~\eqref{eq:CFP} that can be evaluated everywhere in $\R^2\setminus\overline\Omega$.  From the definition of the single- and double-layer potentials~\eqref{eq:pot} it is clear that for any given target point~$\ner\in\R^2\setminus\overline\Omega$ the integrands involved in the definition of the combined potential~\eqref{eq:CFP} are, in principle, smooth functions of the source/integration point $\ney\in\Gamma$.  
In practice, however, numerical issues arise when the observation point $\ner$ lies ``near" the boundary (see discussion on the 5$h$-rule in~\cite[Remark 6]{Barnett:2014tq}). In this case the boundary integrands in~\eqref{eq:CFP} are still smooth functions of $\ney\in\Gamma$, but since both the Green function $i/4H_0^{(1)}(k|\ner-\ney|)$ and its normal derivative blow up as $O(\log|\ner-\ney|)$ and  $O(|\ner-\ney|^{-1})$ respectively,  when $\ner\to\ney\in\Gamma$,  large numbers of quadrature points are needed to properly resolve the nearly singular character of the boundary integrands. 

In this section we describe how to utilize the smoothing operators $R_D$ and $R_S$ introduced in Section~\ref{eq:reg_IE} above, to substantially mitigate the errors produced by the naive numerical approximation of nearly singular integrals that arise from evaluation of the combined potentials close to the boundary.

Selecting  $\nex_0=\bar\ner$ in~\eqref{eq:mod_rep_formula}, where $\bar\ner\in\Gamma$  is such that $|\ner-\bar\ner|=\min_{\nex\in\Gamma}|\ner-\nex|$, we obtain that the combined potential~\eqref{eq:CFP} can be expressed as 
\begin{equation}\begin{split}
u(\ner) = \mathcal D\lf[R_D[\varphi|\bar\ner]\rg](\ner)-\mathcal S\lf[R_S[\varphi|\bar\ner]\rg](\ner),\quad \ner\in\R^2\setminus\overline\Omega.\label{eq:reg_pot}
\end{split}\end{equation} 
It thus follows from the properties of the operators $R_D$ and $R_S$ established in Lemma~\ref{eq:lem_important} that for a density function $\varphi\in H^{s}(\Gamma),$ $s>5/2$, the integrands in the expression on the right-hand-side of~\eqref{eq:reg_pot} satisfy
\begin{equation*}
 G(\ner,\ney)R_S[\varphi|\bar\ner](\ney) = O(|\ney-\bar\ner|^2\log|\ner-\bar\ner|)\ \mbox{ and }\ \frac{\p G(\ner,\ney)}{\p n(\ney)}R_D[\varphi|\bar\ner](\ney) = O\lf(\frac{|\ney-\bar\ner|^2}{|\ner-\bar\ner|}\rg)
\end{equation*}
as $\ney,\ner\to\bar\ner$, $\ner\in\R^2$,  $\bar\ner,\ney\in\Gamma$.

It is demonstrated in Section~\ref{sec:numerics} through numerical examples that use of the smoothed potential~\eqref{eq:reg_pot} instead of the combined potential~\eqref{eq:pot}  improves considerably the numerical accuracy of the fields at target points near the boundary.

\subsection{Close obstacles\label{sec:multiple}}
An issue similar to the one described above in Section~\ref{sec:ner_field} arises in  scattering configurations involving two or more obstacles that are close to each other. 
 
Without loss of generality we let $\Omega$ be composed by two disjoint obstacles $\Omega_1$ and $\Omega_2$ with smooth boundaries $\Gamma_1$ and $\Gamma_2$ ($\Gamma=\Gamma_1\cup\Gamma_2$). Clearly, evaluation of any of the integral operators~\eqref{eq:int_op} on the curve $\Gamma_1$ entails integration on $\Gamma_2$ of a certain density function multiplied by the Green function or one of its normal derivatives with respect to the target point $\nex\in\Gamma_1$,  the source point $\ney\in\Gamma_2$, or both. If $\Gamma_1$ and $\Gamma_2$ are close to each other, that is, there are source points $\ney\in\Gamma_2$ that are close to target points $\nex\in\Gamma_1$, the relevant integrands in the~\eqref{eq:int_op}  become nearly singular as they portrait singularities of the form  $O(\log|\nex-\ney|)$ in the case of the single-layer, $O(|\nex-\ney|^{-1})$ in the case of the double-layer and adjoint double layer operators, and  $O(|\nex-\ney|^{-2})$  in the case of the hypersingular operator.

Following the ideas presented in Section~\ref{sec:ner_field} and  letting $\varphi|_{\Gamma_1}=\varphi_1$ and $\varphi|_{\Gamma_2}=\varphi_2$ denote the restriction of the density function to each one of the curves,  we have that the combined field integral operators at a point $\nex\in\Gamma_i$, $i=1,2$, can be expressed as
\begin{equation}\begin{split}
\lf(\frac{I}{2}+K-i\eta S\rg)[\varphi](\nex)=\int_{\Gamma_i}\!\lf\{\frac{\p G(\nex,\ney)}{\p n(\ney)}R_D[\varphi_i|\nex](\ney)-G(\nex,\ney)R_S[\varphi_i|\nex](\ney)\rg\}\de s(\ney)+\\
  \int_{\Gamma_j}\lf\{\frac{\p G(\nex,\ney)}{\p n(\ney)}R_D[\varphi_j|\bar\nex_j](\ney)- G(\nex,\ney)R_S[\varphi_j|\bar\nex_j](\ney)\rg\}\de s(\ney)
\end{split}\label{eq:dir_near}\end{equation}
in the case of the D-CFIE, and 
\begin{equation}\begin{split}
\lf(\frac{i\eta I}{2}+N-i\eta K'\rg)[\varphi](\nex) = \int_{\Gamma_i}\!\!\lf\{\frac{\p^2 G(\nex,\ney)}{\p n(\nex)\p n(\ney)}R_D[\varphi_i|\nex](\ney)-\frac{\p G(\nex,\ney)}{\p n(\nex)}R_S[\varphi_i|\nex](\ney)\rg\}\de s(\ney)+\\
  \int_{\Gamma_j}\lf\{\frac{\p^2 G(\nex,\ney)}{\p n(\nex)\p n(\ney)}R_D[\varphi_j|\bar\nex_j](\ney)- \frac{\p G(\nex,\ney)}{\p n(\nex)}R_S[\varphi_j|\bar\nex_j](\ney)\rg\}\de s(\ney)
\end{split}\label{eq:neu_near}\end{equation}
in the case of the N-CFIE,  where the point $\bar\nex_j\in\Gamma_j$, $j=1,2$, $j\neq i$, is such that $|\nex-\bar\nex_j|=\min_{\ney\in\Gamma_j}|\nex-\ney|$.

Therefore, it follows from Lemma~\ref{eq:lem_important} that  all the integrands on the right hand side of~\eqref{eq:dir_near} and~\eqref{eq:neu_near} remain bounded regardless the distance between the curves $\Gamma_1$ and $\Gamma_2$. In fact,  the integrands on~$\Gamma_j$ in~\eqref{eq:dir_near} satisfy 
\begin{equation*}\begin{split}
  G(\nex\ney)R_S[\varphi|\bar\nex_j](\ney) =&\ O(|\ney-\bar\nex_j|^2\log|\ner-\bar\nex_j|)\ \mbox{ and }\\
   \frac{\p G(\nex,\ney)}{\p n(\ney)}R_D[\varphi|\bar\nex_j](\ney) =&\ O\lf(\frac{|\ney-\bar\nex_j|^2}{|\nex-\bar\nex_j|}\rg),\end{split}
\end{equation*}
while the integrands on~$\Gamma_j$ in~\eqref{eq:neu_near} satisfy 
\begin{equation*}\begin{split}
 \frac{\p G(\ner,\ney)}{\p n(\nex)}R_S[\varphi|\bar\nex_j](\ney) =&\ O(|\ney-\bar\nex_j|^2\log|\nex-\bar\nex_j|)\quad \mbox{ and }\\
  \frac{\p^2 G(\nex,\ney)}{\p n(\nex)\p n(\ney)}R_D[\varphi|\bar\nex_j](\ney) =&\ O\lf(\frac{|\ney-\bar\nex_j|^2}{|\nex-\bar\nex_j|^2}\rg),
\end{split}\end{equation*}
as $\ney,\nex\to\bar\nex_j$. 

\section{Smoothing functions}\label{sec:reg_functions}
This section  presents  explicit expressions for the smoothing functions $p_0(\ner|\nex_0)$ and $p_1(\ner|\nex_0)$ introduced in~\eqref{eq:funct_p} and~\eqref{eq:funct_q}, respectively, in terms of linear combinations of planes waves (LCPW) of the form $\e^{ik\bol d\cdot (\ner-\nex_0)}$, where  $\bol d\in\R^2$ is a constant unit vector 
and where $\nex_0$ is a given point on the curve $\Gamma$. The direction of propagation $\bol d$ of each one of the plane waves will be expressed as a linear combination of the vectors $\nor_0$ and $\bol\tau_0$ which denote the unit normal and unit tangent vectors at $\nex_0\in\Gamma$, respectively. More precisely, using the curve parametrization $\bnex:[0,2\pi)\to\Gamma$  we have $\nex_0=\bnex(t_0)$, $\nor_0=(-x_2'(t_0),x_1'(t_0))/|\bnex'(t_0)|$ and $\bol\tau_0=\bnex'(t_0)/|\bnex'(t_0)|$ for some $t_0\in[0,2\pi)$.  Here we note that any LCPW is indeed  a smooth homogeneous solution of the Helmholtz equation in all of~$\R^2$.

We  begin then by constructing a LCPW $f_0$ that satisfies  the point conditions
 \begin{equation}
\begin{gathered}
f_0(\ner|\nex_0)=1,\quad  \p_n f_0(\ner|\nex_0) = 0,\quad\p_s f_0(\ner|\nex_0) = 0,\quad \p_s\p_n f_0(\ner|\nex_0) = 0\quad\mbox{at}\quad\ner=\nex_0.
\end{gathered}\label{eq:def_ffp}
\end{equation} 
From the law of reflections  we have that a LCPF of the form  $f_0(\ner|\nex_0)=c\{\e^{ik\nor_0\cdot(\ner-\nex_0)}+\e^{-ik\nor_0\cdot(\ner-\nex_0)}\}$ satisfies the homogeneous Neumann boundary condition $\p_n f_0=0$ on the line tangent to $\Gamma$ at the point~$\nex_0$, and, furthermore, it remains constant along the tangent direction $\bol \tau_0$. Therefore, enforcing the condition  $f_0(\nex_0|\nex_0)=1$ we obtain that the  elementary function
$f_0(\ner|\nex_0) =\cos(k\nor_0\cdot(\ner-\nex_0))$ satisfies all the required point conditions~\eqref{eq:def_ffp}.

 Similarly,   considering now a LCPW $g_0$ that satisfies the homogeneous Dirichlet boundary condition $g_0=0$ on the line tangent to $\Gamma$ at the point~$\nex_0$, we find that  
$g_0(\ner|\nex_0) = \sin(k\nor_0\cdot (\ner-\nex_0))/k$
fulfills all  the point conditions
 \begin{equation*}
\begin{gathered}
g_0(\ner|\nex_0)=0,\quad  \p_n g_0(\ner|\nex_0) = 1,\quad\p_s g_0(\ner|\nex_0) = 0,\quad \p_s\p_n g_0(\ner|\nex_0) = 0\quad\mbox{at}\quad\ner=\nex_0.
\end{gathered}\label{eq:def_ggp}
\end{equation*} 
Combining $f_0$ and $g_0$ we thus obtain the  expression 
\begin{equation}
p_0(\ner|\nex_0) =f_0(\ner|\nex_0) + i\eta g_0(\ner|\nex_0)=\cos(k\nor_0\cdot(\ner-\nex_0))+\frac{i\eta}{k}\sin(k\nor_0\cdot(\ner-\nex_0)),\label{eq:p_func}
\end{equation}
for the smoothing function~\eqref{eq:funct_p}.

In order to construct the smoothing function  $p_1$ that satisfies the conditions in~\eqref{eq:funct_q}, on the other hand, we  consider the  LCPW that produces the function
\begin{eqnarray*}
 g_1(\ner|\nex_0) 
&=&\frac{2}{k^2}\sin\lf(\frac{k}{\sqrt{2}}\nor_0\cdot(\ner-\nex_0)\rg)\sin\lf(\frac{k}{\sqrt{2}}\bol \tau_0\cdot(\ner-\nex_0)\rg) \label{eq:q_func}
\end{eqnarray*}
which satisfies
 \begin{equation*}
\begin{gathered}
g_1(\ner|\nex_0)=0,\quad  \p_n g_1(\ner|\nex_0) = 0,\quad\p_s g_1(\ner|\nex_0) = 0,\quad \p_s\p_n g_1(\ner|\nex_0) = 1\quad\mbox{at}\quad\ner=\nex_0.
\end{gathered}\label{eq:def_gq}
\end{equation*}
Therefore, in order to construct $p_1$ it suffices to provide a LCPW $f_1$ such that
\begin{equation}
\begin{gathered}
f_1(\ner|\nex_0)=0,\quad  \p_s f_1(\ner|\nex_0) = 0,\quad\p_n f_1(\ner|\nex_0) = 1,\quad \p_s\p_n f_1(\ner|\nex_0) = 0\quad\mbox{at}\quad\ner=\nex_0.\end{gathered}\label{eq:def_ffq}
\end{equation}
In order to do so, we first consider the LCPW that produces the function
\begin{equation*}
\tilde  f_1(\ner|\nex_0) =\frac{\sqrt{2}}{k}\sin\lf(\frac{k}{\sqrt{2}}\bol\tau_0\cdot(\ner-\nex_0)\rg)\cos\lf(\frac{k}{\sqrt{2}}\bol n_0\cdot(\ner-\nex_0)\rg)\label{eq:q_func}
\end{equation*}
which satisfies the point conditions
\begin{equation*}
\begin{gathered}
\tilde f_1(\ner|\nex_0)=0,\ \  \p_n\tilde  f_1(\ner|\nex_0) = 0,\ \
\p_s \tilde f_1(\ner|\nex_0) = 1,\ \ \p_s\p_n \tilde f_1(\ner|\nex_0) = -\frac{\bol n_0\cdot\bnex''(t_0)}{|\bnex'(t_0)|^2} \ \ \mbox{at}\  \ \ner=\nex_0.
\end{gathered}\label{eq:def_f}
\end{equation*}
Therefore, clearly, a LCPW satisfying all the points conditions listed in~\eqref{eq:def_ffq} is given by 
 $$
 f_1(\ner|\nex_0) = \tilde f_1(\ner|\nex_0)+\frac{\bol n_0\cdot\bnex''(t_0)}{|\bnex'(t_0)|^2} g_1(\ner|\nex_0).
 $$
Finally, combining $f_1$ and $g_1$ we  obtain 
\begin{equation}\begin{split}
p_1(\ner|\nex_0) =&\  f_1(\ner|\nex_0)+i\eta g_1(\ner|\nex_0)\\
=&\ \lf\{\frac{\sqrt{2}}{k}\cos\lf(\frac{k}{\sqrt{2}}\bol n_0\cdot(\ner-\nex_0)\rg)+\left\{\frac{\bol n_0\cdot\bnex''(t_0)}{|\bnex'(t_0)|^2}+i\eta\rg\} \frac{2}{k^2}\sin\lf(\frac{k}{\sqrt{2}}\nor_0\cdot(\ner-\nex_0)\rg)\rg\}\\
&\ \times\sin\lf(\frac{k}{\sqrt{2}}\bol \tau_0\cdot(\ner-\nex_0)\rg).
\end{split}
\end{equation}

\section{Numerical examples\label{sec:numerics}}

This section presents  numerical examples that illustrate the properties of the smoothed CFIEs~\eqref{eq:IE_SDirichlet} and~\eqref{eq:IE_SNeumann} and as well as the  smoothed potential~\eqref{eq:reg_pot}.   

\subsection{Nystr\"om discretizations}
Three different Nystr\"om methods for the numerical solution of the smoothed CFIEs~\eqref{eq:IE_SDirichlet} and~\eqref{eq:IE_SNeumann} are briefly reviewed  in this section. In order to obtain Nystr\"om discretizations we need to provide quadrature rules for the numerical evaluation  of the smoothed integral operators~\eqref{eq:op_param_Dir} and~\eqref{eq:op_param_Neu}, which are here expressed as
\begin{equation}\int_0^{2\pi}\left\{A(t,\tau)\rho_D(\tau|t)+B(t,\tau)\rho_S(\tau|t)\rg\}\de\tau,\qquad t\in[0,2\pi],\label{eq:integral}\end{equation}
where the  kernels $A$ and $B$  correspond to $A=L$ and $B=M$ in the case of the SD-CFIE~\eqref{eq:op_param_Dir}, and  $A=H$ and $B=W$ in the case of the SN-CFIE~\eqref{eq:op_param_Neu}. We recall that the functions $\rho_D$ and $\rho_S$, introduced in~\eqref{eq:param_1} and ~\eqref{eq:param_2}, respectively, are given in terms of the density function $\phi(t)=\varphi(\bnex(t))$.

The first and simplest Nystr\"om method  considered  is based on the direct use of the classical trapezoidal rule (TR), which applied to  the  integral~\eqref{eq:integral} yields 
  \begin{equation}
  \int_0^{2\pi}\left\{A(t,\tau)\rho_D(\tau|t)+B(t,\tau)\rho_S(\tau|t)\rg\}\de\tau \approx \sum_{j=0}^{2n-1}\left\{A(t,t_j)\rho_D(t_j|t)+B(t,t_j)\rho_S(t_j|t)\rg\}w_j,\label{eq:trapezoidal_rule}
\end{equation}
  where the quadrature weights and quadrature points are given by  $w_j=h=\pi /n$ and $t_j=jh$, $j=0,\ldots,2n-1$, respectively.  A linear system of equations for the approximate values of the density function $\phi_j\approx \phi(t_j)=\varphi(\bnex(t_j))$, $j=0,\ldots,2n-1,$ is   obtained by equating the right-hand-side of~\eqref{eq:trapezoidal_rule} to the corresponding Dirichlet $f_D(t)=-u^\inc(\bnex(t))$ or Neumann $f_N(t)=-\p_n u^\inc(\bnex(t))$  data at the quadrature points $t=t_i$, $i=0,\ldots,2n-1$. Here, the first and second order derivatives of the unknown  function $\phi(t)=\varphi(\bnex(t))$ at $t=t_i$---which are needed for evaluation of 
\begin{subequations}\begin{equation}\rho_D(t_j|t_i)=\phi(t_j)-\phi(t_i)\tilde p_0(t_j|t_i)-|\bold x'(t_i)|^{-1}\phi'(t_i)\tilde p_1(t_j|t_i)\end{equation}
and 
\begin{equation}\rho_S(t_j|t_i)=i\eta\phi(t_j)-\phi(t_i)\p_n\tilde p_0(t_j|t_i)-|\bold x'(t_i)|^{-1}\phi'(t_i)\p_n\tilde p_1(t_j|t_i),\end{equation}
 as well as the diagonal term 
\begin{equation}H(t_i,t_i)\rho_D(t_i|t_i)= \frac{H_0(t_i,t_i)}{2}\left\{\phi''(t_i)-\phi(t_i)\tilde p_0''(t_i|t_i)-\frac{\phi'(t_i)}{|\bold x'(t_i)|}\tilde p_1''(t_i|t_i)\rg\}\label{eq:diag_hyper}\end{equation}\label{eq:der_approx}\end{subequations}
---are approximated by finite differences of $\phi_j$, $j=0,\ldots,2n-1$.   
  Note that although the functions $A(t_i,\tau)\rho_D(\tau|t_i)$ and $B(t_i,\tau)\rho_S(\tau|t_i)$ in~\eqref{eq:trapezoidal_rule} are continuously differentiable, they still have a polylogarithmic singularity  of the form $|t_i-\tau|^2\log|t_i-\tau|$ at $\tau=t_i$. As expected, this mild singularity limits to $O(h^3)$ the overall accuracy of the trapezoidal rule and associated Nystr\"om-method solution $\phi_j$, $j=0,\ldots,2n-1$, which would otherwise exhibit spectral (superalgebraic) accuracy if the integrands were smooth periodic functions.

To achieve higher order accuracy, of course, Nystr\"om methods that properly handle  logarithmic singularities can be utilized instead of the classical trapezoidal rule.  Following~\cite{COLTON:2012} then, we  consider the spectrally accurate  Nystr\"om method of Martensen and Kussmaul (MK) which is suited for the discretization of the smoothed integral equations~\eqref{eq:IE_SDirichlet} and~\eqref{eq:IE_SNeumann} by virtue of the fact that the integrand in~\eqref{eq:integral} can be expressed as 
\begin{equation}
A(t,\tau)\rho_D(\tau|t)+B(t,\tau)\rho_S(\tau|t) =  C_1(t,\tau)\log\lf(4\sin^2\frac{t-\tau}{2}\rg)+ C_2(t,\tau),\label{eq:splitting}
\end{equation}
where $C_1$ and $C_2$ are $2\pi$-biperiodic analytic functions. (The explicit form of the function $C_1$ and $C_2$  can be easily obtained from  identities~\eqref{eq:kernels_dirich} and~\eqref{eq:kernels_neumann}.) It thus follows from~\cite{COLTON:2012} that the quadrature rule\begin{equation}\begin{split}
\int_0^{2\pi}\lf\{C_1(t_i,\tau)\log\lf(4\sin^2\frac{t_i-\tau}{2}\rg)+ C_2(t_i,\tau)\rg\}\de\tau\approx&\ \sum_{j=0}^{2n-1}\left\{R_{|i-j|}C_1(t_i,t_j) +h\,C_2(t_i,t_j)\rg\},
\end{split}\label{eq:MK}\end{equation}
with quadrature weights given by 
$$
R_j = -\frac{2\pi}{n}\sum_{m=1}^{n-1}\frac{1}{m}\cos\frac{m j \pi}{n}-\frac{(-1)^j\pi}{n^2},\qquad j=0,\ldots,2n-1,
$$
provides a spectrally accurate approximation of the smoothed integral operators at the quadrature points $t_i=ih$, $i=0,\ldots,2n-1$. Equating the right-hand-side of~\eqref{eq:MK} to the corresponding boundary data $f_D$ or $f_N$ at the quadrature points $t_i$, $i=0,\ldots, 2n-1$, we obtain a linear system of equations for the quantities  $\phi_j$, $j=0,\ldots,2n-1,$ which approximate the desired unknown values $\phi(t_j),$  $j=0,\ldots, 2n-1$, respectively. The spectral accuracy of the Nystr\"om solution $\phi_j$, $j=0,\ldots,2n-1$, is achieved by approximating $\phi'(t_i)$  and $\phi''(t_i)$ in~\eqref{eq:der_approx} using FFT-based differentiation of the periodic sequence~$\phi_j$, $j=0,\ldots,2n-1$.

Alternatively, the spectrally accurate quadrature rule for integral kernels with $O(|t-\tau|^2\log|t-\tau|)$ singularities introduced in~\cite{boubendir2016high} could be utilized to produce high-order Nystr\"om discretizations of the smoothed CFIEs.

Finally, following~\cite{hao2014high} we also consider the high-order Nystr\"om method based on Kapur-Rokhlin quadrature rules. Utilizing the same quadrature points $t_j=j\pi/n$, $j=0,\ldots,2n-1$, this approach has an advantage over the aforementioned methods in that it does not entail evaluation  of the rather involved  diagonal term~\eqref{eq:der_approx}---as by construction the quadrature weight corresponding to the singular point $\tau=t$ is identically zero (see~\cite{hao2014high} for details).

\begin{figure}[h!]
\centering	
 \subfloat[][SD-CFIE~\eqref{eq:IE_SDirichlet}]{\includegraphics[scale=0.57]{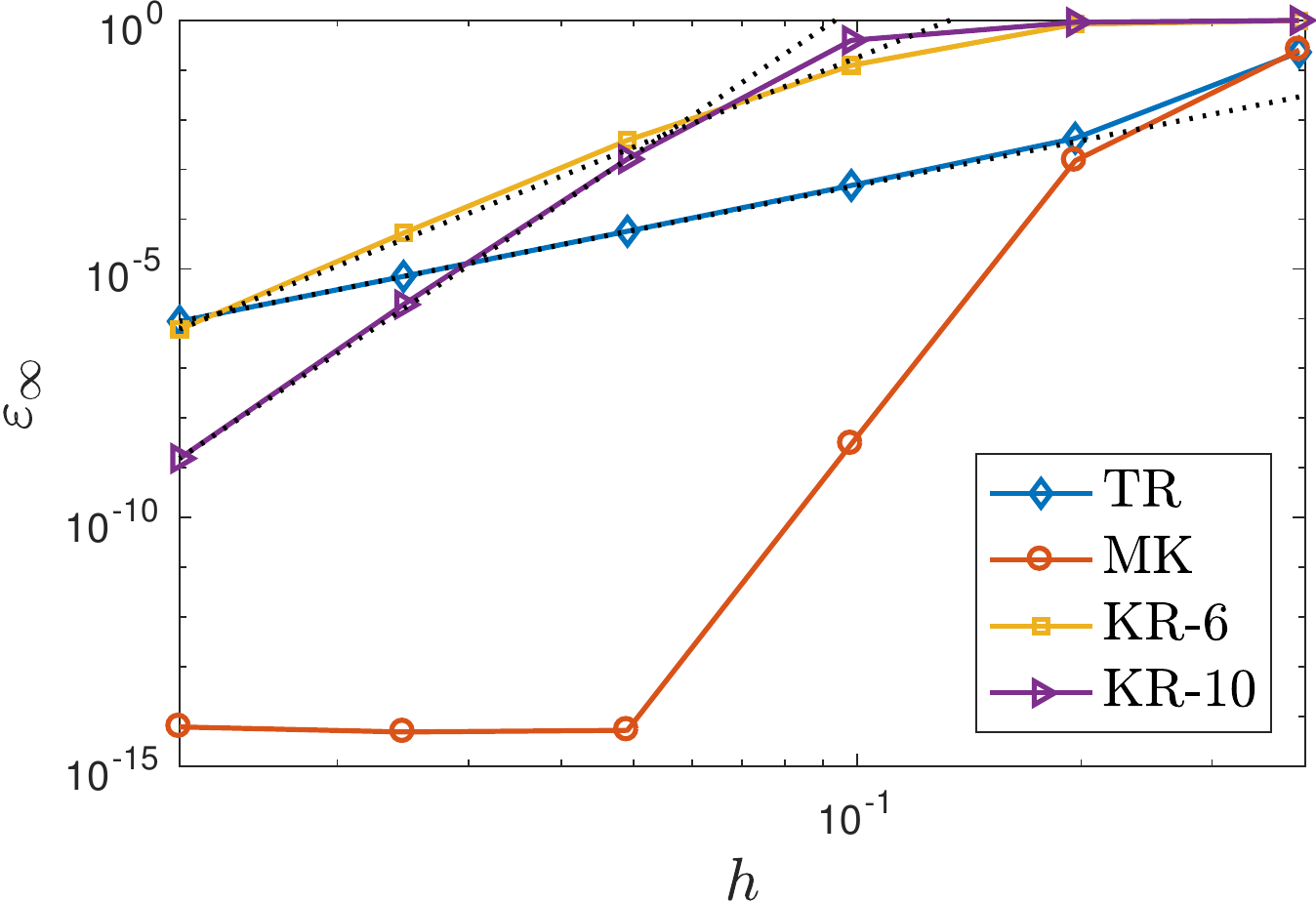}\label{fig_SD_CFIE}}\quad
 \subfloat[][SN-CFIE~\eqref{eq:IE_SNeumann}]{\includegraphics[scale=0.57]{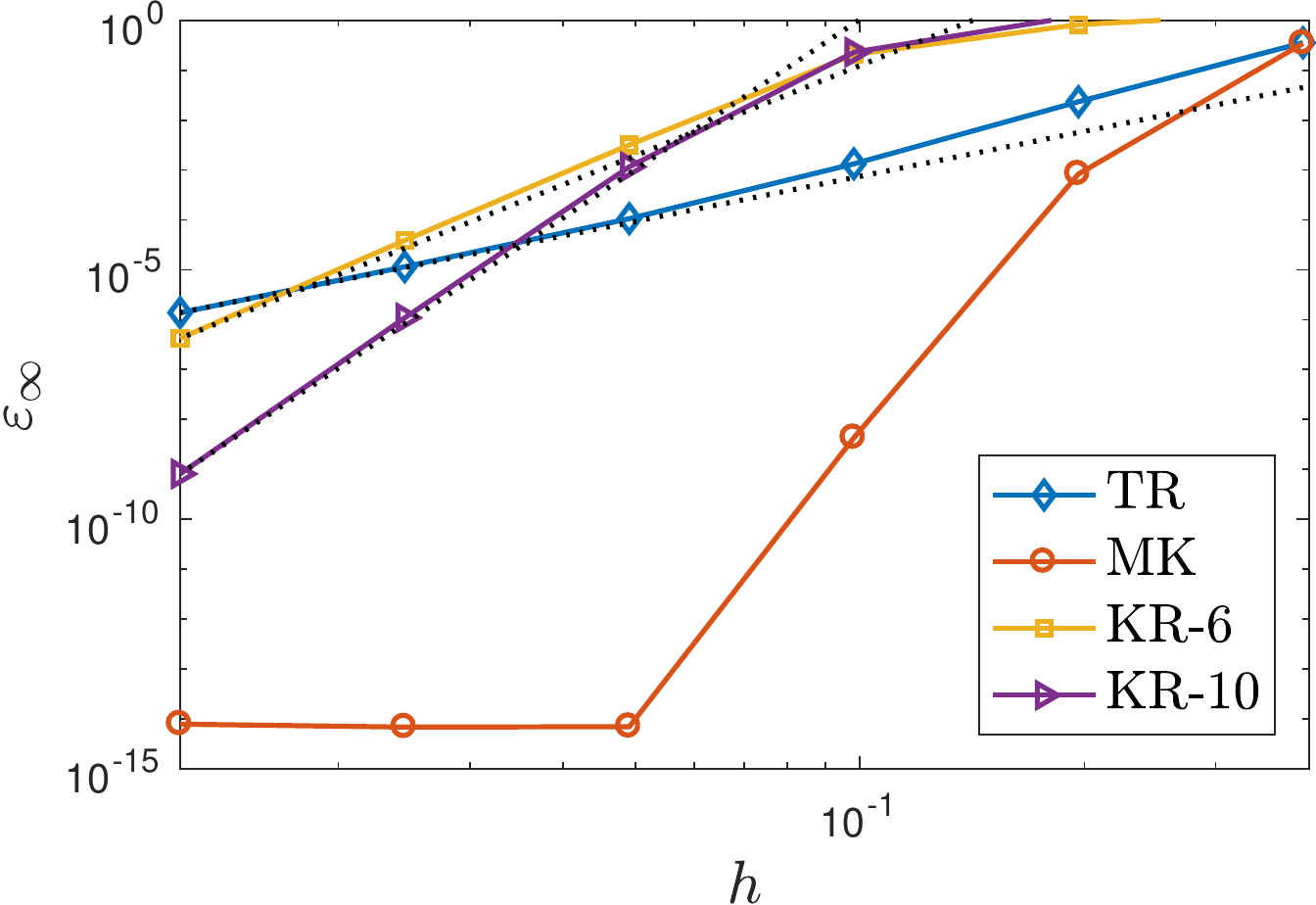}\label{fig_SN_CFIE}}
\caption{Far-field numerical errors for various grid sizes $h>0$ obtained from the smoothed integral equations SD-CFIE (left) and SN-CFIE (right)  using the various Nystr\"om discretization considered in this section;  trapezoidal rule (TR), Martensen \& Kussmaul (MK) and Kapur-Rokhlin of orders six (KR-6) and  ten (KR-10). The dotted lines have slopes three, six and ten. All solutions where computed for the fixed wavenumber $k=4$.}\label{fig:results_RCFiE}
\end{figure}  

\subsection{Smooth obstacles}
 In our first numerical experiment we compare the accuracy of the various Nystr\"om discretizations for the  solution of the smoothed CFIEs~\eqref{eq:IE_SDirichlet} and~\eqref{eq:IE_SNeumann} for the exterior boundary value problems~\eqref{eq:DEP} and~\eqref{eq:NEP}, respectively, that result from the scattering of an incident plane-wave $u^\inc(\nex) = \e^{ik\bol d\cdot\nex}$, $\bol d=(\cos\pi/8,\sin\pi/8)$, that impinges on a kite-shaped obstacle with boundary 
$\Gamma = \{(\cos t+0.65(\cos 2t-1),1.5\sin t)\in\R^2,t\in[0,2\pi)\}.$ The parameter value $\eta=k$ is utilized in all the numerical examples considered in this paper.

 Figure~\ref{fig:results_RCFiE}  displays the maximum relative errors
 \begin{equation*}\varepsilon_\infty=\max_{|\hat\nex|=1}\lf\{\frac{|\tilde u_{\infty}(\hat\nex)-u_{\infty}(\hat\nex)|}{|u_{\infty}(\hat\nex)|}\rg\},\label{eq:errors}\end{equation*}
 in the approximate far-field pattern $\tilde u_\infty$ obtained from the numerical solution~\eqref{eq:IE_SDirichlet} and~\eqref{eq:IE_SNeumann} for various grid sizes $h=\pi/n$. The numerical errors displayed in Figure~\ref{fig:results_RCFiE} are measured with respect to a reference highly accurate far-field pattern~$u_\infty$. Here we recall that the far-field pattern  corresponding to an integral equation solution~$\varphi$ is given~by~\cite{COLTON:2012}
$$ u_\infty(\hat\nex)=\frac{\e^{-i\pi/4}}{\sqrt{8\pi k}}\int_{\Gamma}\lf\{k \nor(\ney)\cdot\hat\nex+\eta\rg\}\e^{-ik\,\hat\nex\cdot\ney}\varphi(\ney)\de s(\ney)\quad (|\hat\nex|=1).
$$
As can be observed in Figure~\ref{fig:results_RCFiE}, the three Nystr\"om discretizations yield the expected convergence rates which are in turn determined by the order of the associated quadrature rules.  Interestingly,  it can be noted that for a wide range of grid sizes $h$  the low-order TR method produces more accurate results than the high-order Nystr\"om method based on the Kapur-Rokhlin quadrature rules or orders six (KR-6) and ten (KR-10).

 \begin{table}
   \begin{center}
   \scalebox{0.9}{
\begin{tabular}{c|c|c|c|c|c|c|c|c|c|c|c}
\toprule\rule{0pt}{2.5ex} 
$k$ &$N$& \multicolumn{2}{c|} {TR SD-CFIE}  & \multicolumn{2}{c|} {MK D-CFIE} & \multicolumn{2}{c|} {MK SD-CFIE} & \multicolumn{2}{c|} {KR-10 D-CFIE}& \multicolumn{2}{c} {KR-10 SD-CFIE }\\
\cline{3-12}
&$(\frac{2\pi}{h})$& It. &  $\varepsilon_\infty$ & It. & $\varepsilon_\infty$ & It.& $\varepsilon_\infty$ & It.&$\varepsilon_\infty$& It.&$\varepsilon_\infty$\\
\hline
1&30&8&$7.15\times 10^{-4}$&8&$1.01\times 10^{-5}$&8&$6.08\times 10^{-4}$&58&$9.24\times10^{-1}$&58&$5.59\times 10^{-1}$\\
4&120&10&$9.52\times 10^{-5}$&10&$2.05\times 10^{-5}$&10&$2.05\times 10^{-5}$&83&$3.65\times10^{-3}$&83&$4.21\times 10^{-3}$\\
16&480&12&$8.15\times 10^{-5}$&12&$2.19\times 10^{-5}$&12&$2.19\times 10^{-5}$&214&$2.38\times10^{-3}$&215&$2.13\times 10^{-3}$\\
64&1920&16&$1.05\times 10^{-4}$&16&$1.39\times 10^{-5}$&16&$1.39\times 10^{-5}$&793&$1.73\times 10^{-3}$&809&$1.67\times 10^{-3}$\\
\toprule\rule{0pt}{2.5ex} 
$k$ &$N$& \multicolumn{2}{c|} {TR SN-CFIE}  & \multicolumn{2}{c|} {MK N-CFIE} & \multicolumn{2}{c|} {MK SN-CFIE} & \multicolumn{2}{c|} {KR-10 N-CFIE}& \multicolumn{2}{c} {KR-10 SD-CFIE }\\
\cline{3-12}
&$(\frac{2\pi}{h})$& It. &  $\varepsilon_\infty$ & It. & $\varepsilon_\infty$ & It.& $\varepsilon_\infty$ & It.&$\varepsilon_\infty$& It.&$\varepsilon_\infty$\\
\hline
1&30&17&$3.35\times 10^{-4}$&20&$ 2.16\times10^{-5}$&20&$7.32\times 10^{-5}$&28&$2.47\times 10^{-1}$&58&$2.88\times 10^{-0}$\\
4&120&24&$1.80\times 10^{-4}$&29&$5.49\times10^{-5}$&29&$4.74\times 10^{-5}$&103&$3.32\times 10^{-3}$&84&$3.41\times 10^{-3}$\\
16&480&23&$1.00\times 10^{-4}$&27&$3.04\times10^{-5}$&27&$3.03\times 10^{-5}$&346&$1.45\times 10^{-3}$&230&$9.61\times 10^{-4}$\\
64&1920&12&$1.03\times 10^{-4}$&14&$2.22\times10^{-5}$&14&$2.19\times 10^{-5}$&998&$1.29\times 10^{-3}$&719&$7.97\times 10^{-4}$\\
\bottomrule
\end{tabular}}
\caption{\label{tab:GMRES} Number of  GMRES iterations needed to solve of the linear systems  resulting from the TR, MK and KR-10 Nystr\"om discretizations of the smoothed~(SD-CFIE~\eqref{eq:IE_SDirichlet}  and SN-CFIE~\eqref{eq:IE_SNeumann})  and non-smoothed (D-CFIE~\eqref{eq:IE_Dirichlet} and N-CFIE~\eqref{eq:IE_Neumann}) integral equations. The upper panel displays the results for the Dirichlet integral equations, D-CFIE and SD-CFIE, and the lower panel displays the results for the Neumann integral equations, N-CFIE and SN-CFIE. The GMRES tolerance $10^{-6}$ was utilized in all the examples considered in this table.}
\end{center}
 \end{table}

Table~\ref{tab:GMRES}, on the other hand, displays the number of GMRES iterations required to solve the linear systems resulting from the three  Nystr\"om discretizations of the smoothed and non-smoothed integral equations. 
These results show that the numbers of GMRES iterations required to solve the smoothed integral equations (SD-CFIE~\eqref{eq:IE_SDirichlet}  and SN-CFIE~\eqref{eq:IE_SNeumann}) do not differ considerably from those required to solve the non-smoothed integral equations (D-CFIE~\eqref{eq:IE_Dirichlet} and N-CFIE~\eqref{eq:IE_Neumann}---these are, of course, expected results in view of the operator identities established in Theorems~\ref{lem:equal} and~\ref{lem:equal_N}. Here we point out that the inordinate large number of GMRES iteration required by the KR-10 method makes it in practice not amenable to iterative linear algebra solvers~\cite{hao2014high}. 
 
 \begin{figure}[h!]
\centering	
 \subfloat[][Combined potential (left) and smoothed-combined potential (right) applied to  SD-CFIE solution obtained via  MK Nystr\"om method.]{\includegraphics[scale=0.0550]{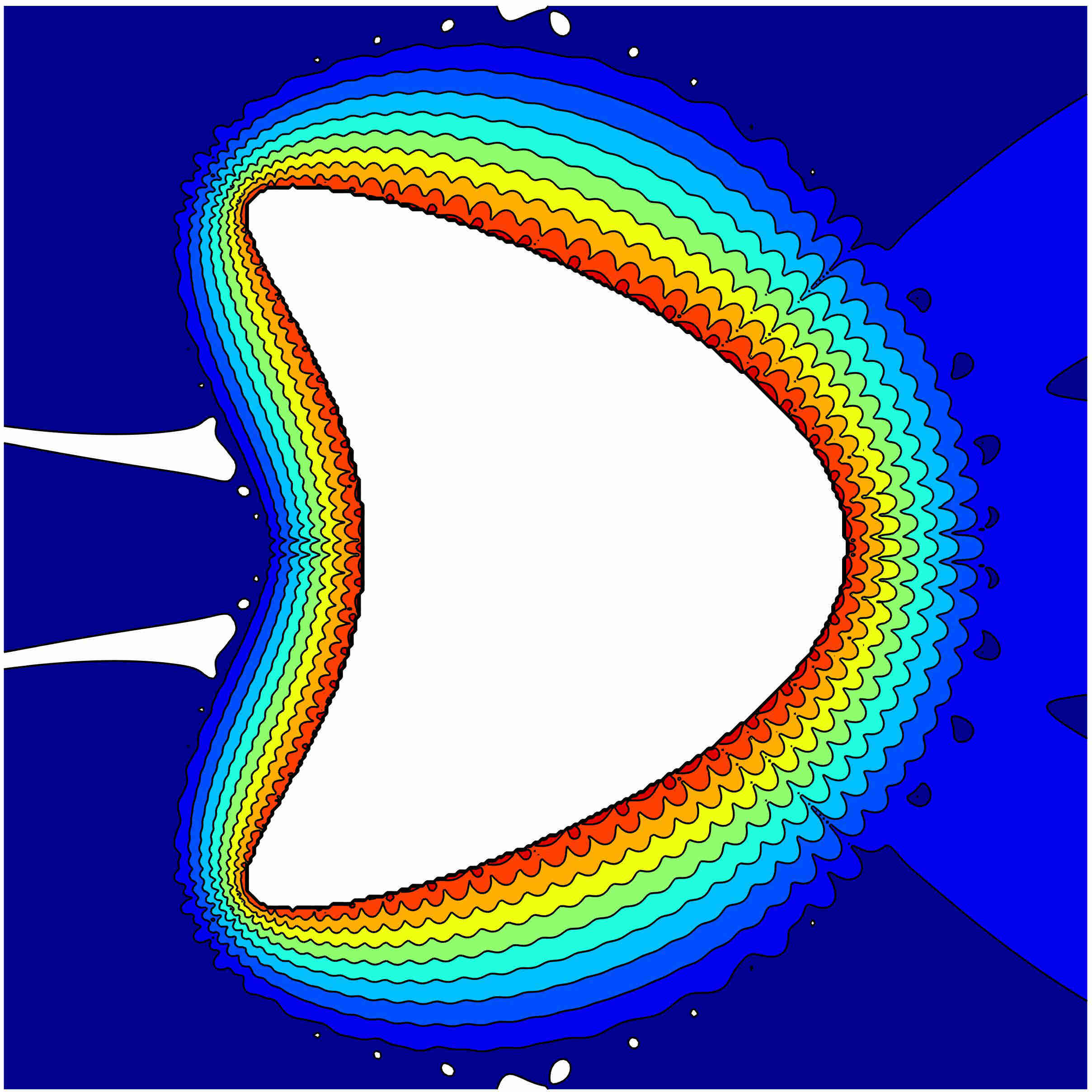}\qquad \includegraphics[scale=0.0550]{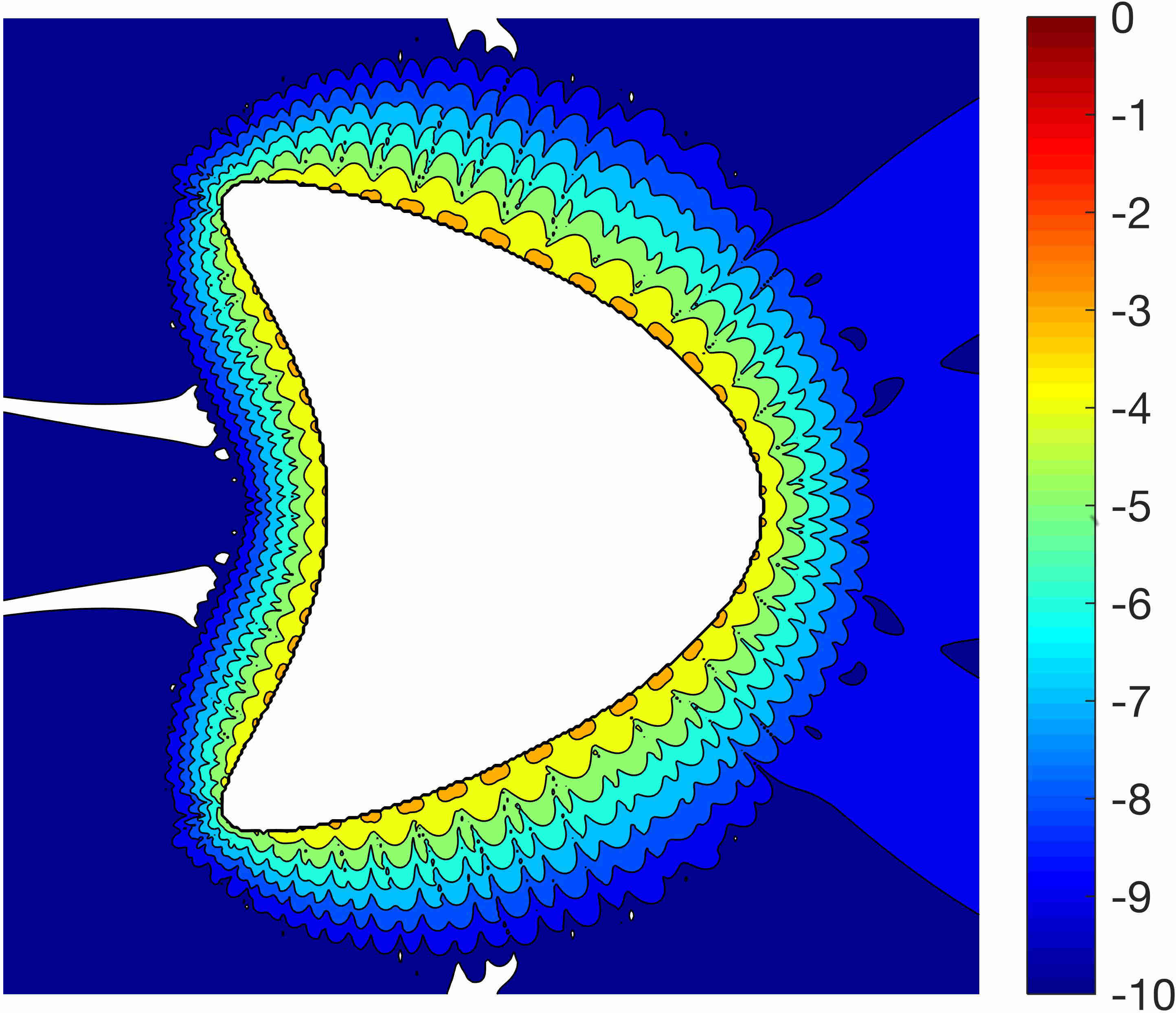}\label{5a}}\\
  \subfloat[][Combined potential (left) and smoothed-combined potential (right) applied to SD-CFIE solution obtained via  TR Nystr\"om method.]{\includegraphics[scale=0.0550]{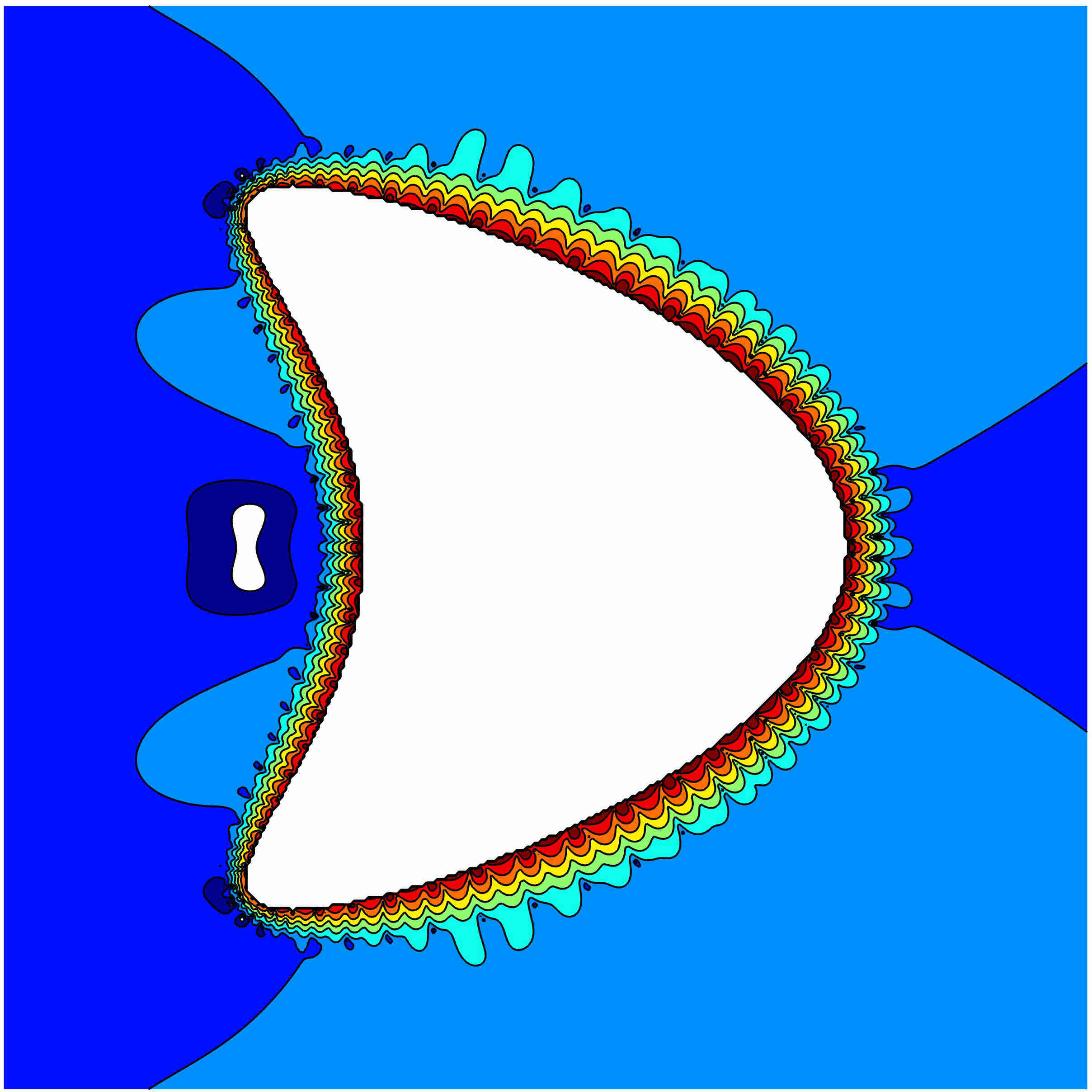}\qquad \includegraphics[scale=0.0550]{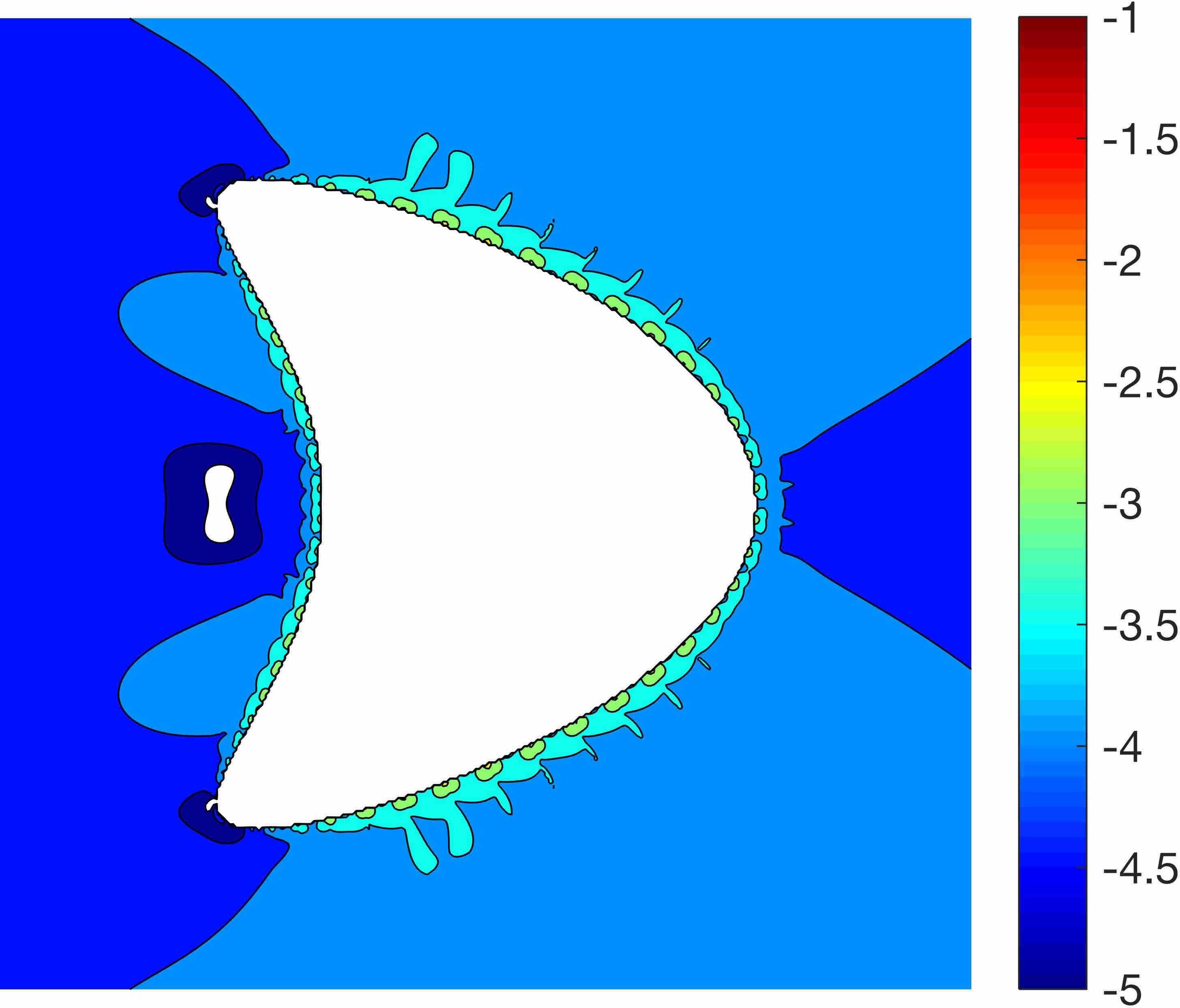}\label{5b}}
\caption{Errors ($\log_{10}|\tilde u(\nex)-u^{\mathrm{exact}}(\nex)|$) in the approximate fields $\tilde u$ produced by use of the non-smoothed~\eqref{eq:CFP} and smoothed~\eqref{eq:reg_pot} combined potentials applied to SD-CFIE solutions obtained via (a)~MK  and (b)~TR  Nystr\"om methods. The SD-CFIE~\eqref{eq:IE_SDirichlet} was discretized using $N=2\pi/h=64$ quadrature points, and the wavenumber $k=4$ was utilized in all the examples considered in these figures.}\label{fig:results_pot_close}
\end{figure}  
 Our second numerical experiment illustrates the properties of the smoothed combined potential~\eqref{eq:reg_pot} when evaluated at target points close to boundaries.  In this experiment we consider the exterior Dirichlet problem~\eqref{eq:DEP} with boundary data $u^\inc(\nex)=-H_0^{(1)}(k|\nex|)$ on $\Gamma$ so that the exact solution of~\eqref{eq:DEP} is $u^{\mathrm{exact}}(\nex)=H_0^{(1)}(k|\nex|)$ in $\R^2\setminus\overline\Omega$. Four different approximate solutions~$\tilde u$ of~\eqref{eq:DEP} are obtained by solving the smoothed integral equation~\eqref{eq:IE_SDirichlet} by means of the MK and TR methods, and then utilizing the resulting  density functions to produce the near fields via the combined and smoothed-combined potentials given in~\eqref{eq:CFP} and \eqref{eq:reg_pot}, respectively.  Figure~\ref{fig:results_pot_close} displays the logarithm (in base 10) of the absolute errors $|\tilde u(\nex)-u^{\mathrm{exact}}(\nex)|$ within a region near the boundary $\Gamma$ of the kite-shaped obstacle. This figure  shows that, for the given discretization level, the smoothed potential produces near fields that are at least three digits more accurate than those obtained through the non-smoothed combined potential, in both MK and TR cases.

In the final numerical experiment of this section we illustrate the advantages of the smoothed operators~\eqref{eq:dir_near} and~\eqref{eq:neu_near}, presented in Section~\ref{sec:multiple}, for the solution problems involving obstacles that are close to each other.  This experiment considers a geometric configuration consisting of two kite-shaped obstacles with boundaries $\Gamma_1= \{(\cos t+0.65(\cos 2t-1)-1-d/2,1.5\sin t)\in\R^2,t\in[0,2\pi)\}$ and $\Gamma_2= \{(-\cos t-0.65(\cos 2t-1)+1+d/2,1.5\sin t)\in\R^2,t\in[0,2\pi)\}$ separated by a distance $d>0$. In order to assess the numerical errors we construct exact solutions to both Dirichlet~\eqref{eq:DEP} and Neumann~\eqref{eq:NEP} problems by imposing the boundary conditions $u^\inc(\nex)=-(H_0^{(1)}(k|\nex-\nex_1|)+H_0^{(1)}(k|\nex-\nex_2|))$  and $\p_n u^\inc=-\p_n(H_0^{(1)}(k|\nex-\nex_1|)+H_0^{(1)}(k|\nex-\nex_2|))$ on $\Gamma_1\cup\Gamma_2$, respectively, where  $\nex_1=-(1+d/2,0)\in\Gamma_1$ and $\nex_2=(1+d/2,0)\in\Gamma_2$ . The exact solution to both problems~\eqref{eq:DEP} and~\eqref{eq:NEP} is then given by $u_{\mathrm{exact}}(\nex)=H_0^{(1)}(k|\nex-\nex_1|)+H_0^{(1)}(k|\nex-\nex_2|)$. 

Figure~\ref{fig:results_close_convergence} displays the far-field errors resulting from application of the MK and TR methods to the Dirchlet and Neumann integral equations with smoothed integral operators~\eqref{eq:dir_near} and~\eqref{eq:neu_near}, respectively, and the far-field errors resulting from application of the MK method to the non-smoothed integral equations D-CFIE~\eqref{eq:IE_Dirichlet} and N-CFIE~\eqref{eq:IE_Neumann}. Figure~\ref{fig_fix_h}, in particular,  displays the errors for various separations $d>0$ and a fixed grid size $h=\pi/32$. This figure  reveals the pronounced accuracy deterioration of the MK solutions (blue  and red curves) as  $d\to 0$ in both Dirichlet (left) and Neumann (right) cases. This  deterioration is significantly milder in the case of the smoothed integral equations (red curves). The accuracy of the TR solutions (yellow curves), in turn,  does not seem to deteriorate $d\to 0$. Figure~\ref{fig_fix_d}, on the other hand, displays the far-field errors for a fixed separation $d=10^{-5}$ and various grid sizes $h>0$. This figure demonstrates that the use the smoothed integral equation (red and yellow curves) improves significantly the accuracy of the far-field in the Neumann case (right). In the Dirichlet case (left), however, the far-fields resulting from smoothed and non-smoothed integral equations exhibit similar convergence rates and accuracies.

 \begin{figure}[h!]
\centering	
 \subfloat[][Far-field errors in  Dirichlet (left) and Neumann (right) solutions for various separations $d>0$ and a fixed grid size $h=\pi/32$.]{\includegraphics[scale=0.57]{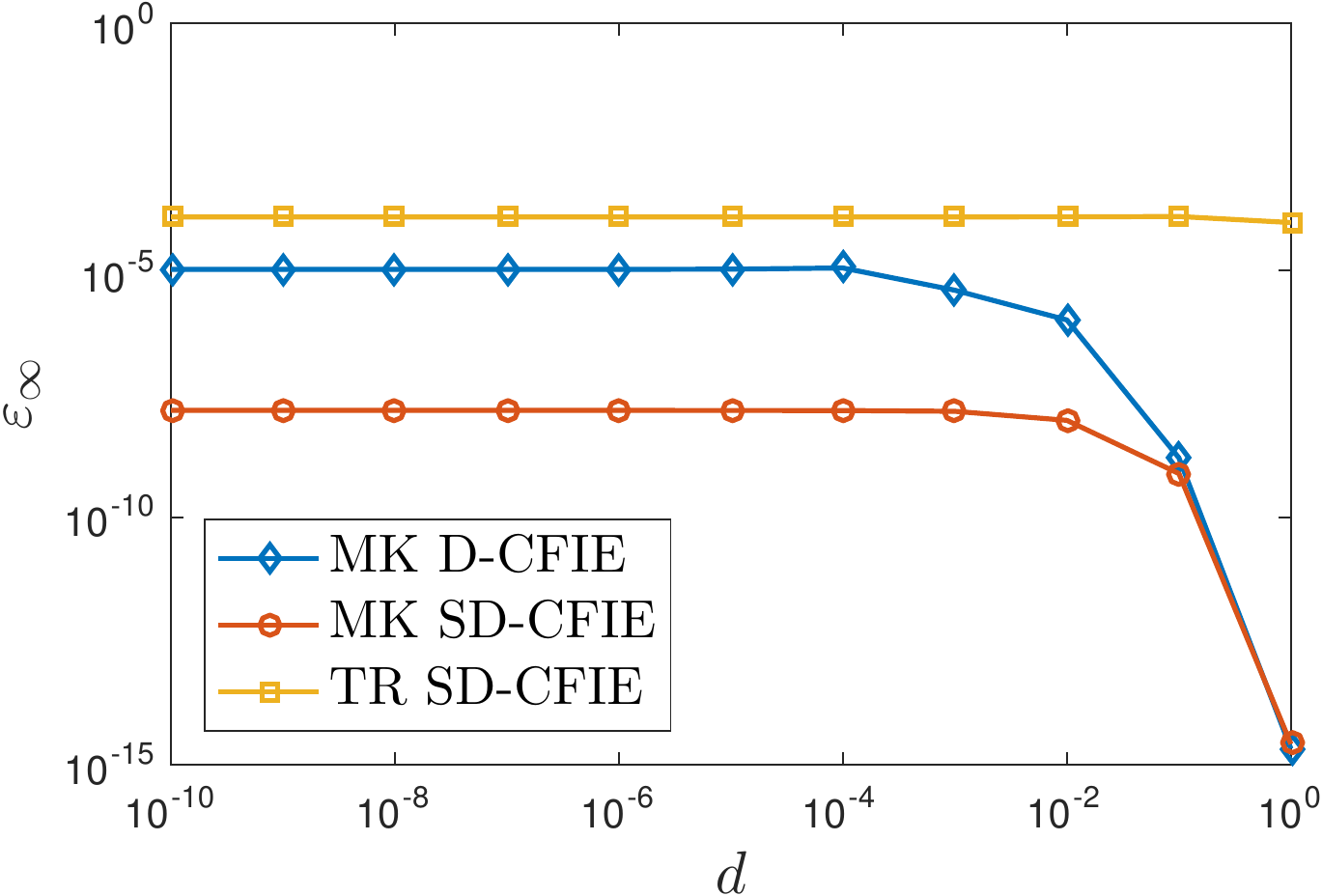}\qquad \includegraphics[scale=0.57]{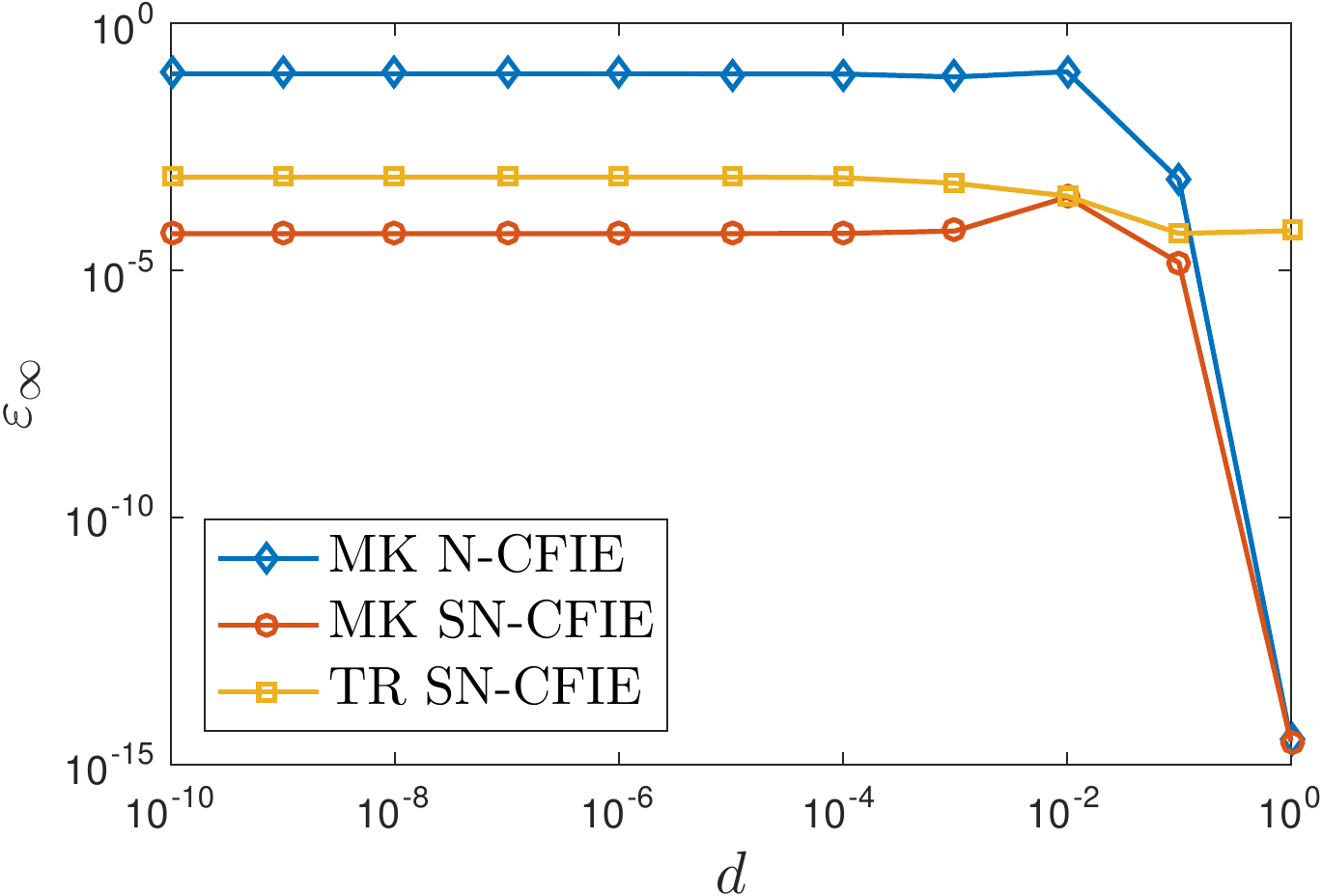}\label{fig_fix_h}}\\
  \subfloat[][Far-field errors in  Dirichlet (left) and Neumann (right) solutions for various gird sizes $h>0$ and a fixed separation $d=10^{-5}$.]{\includegraphics[scale=0.57]{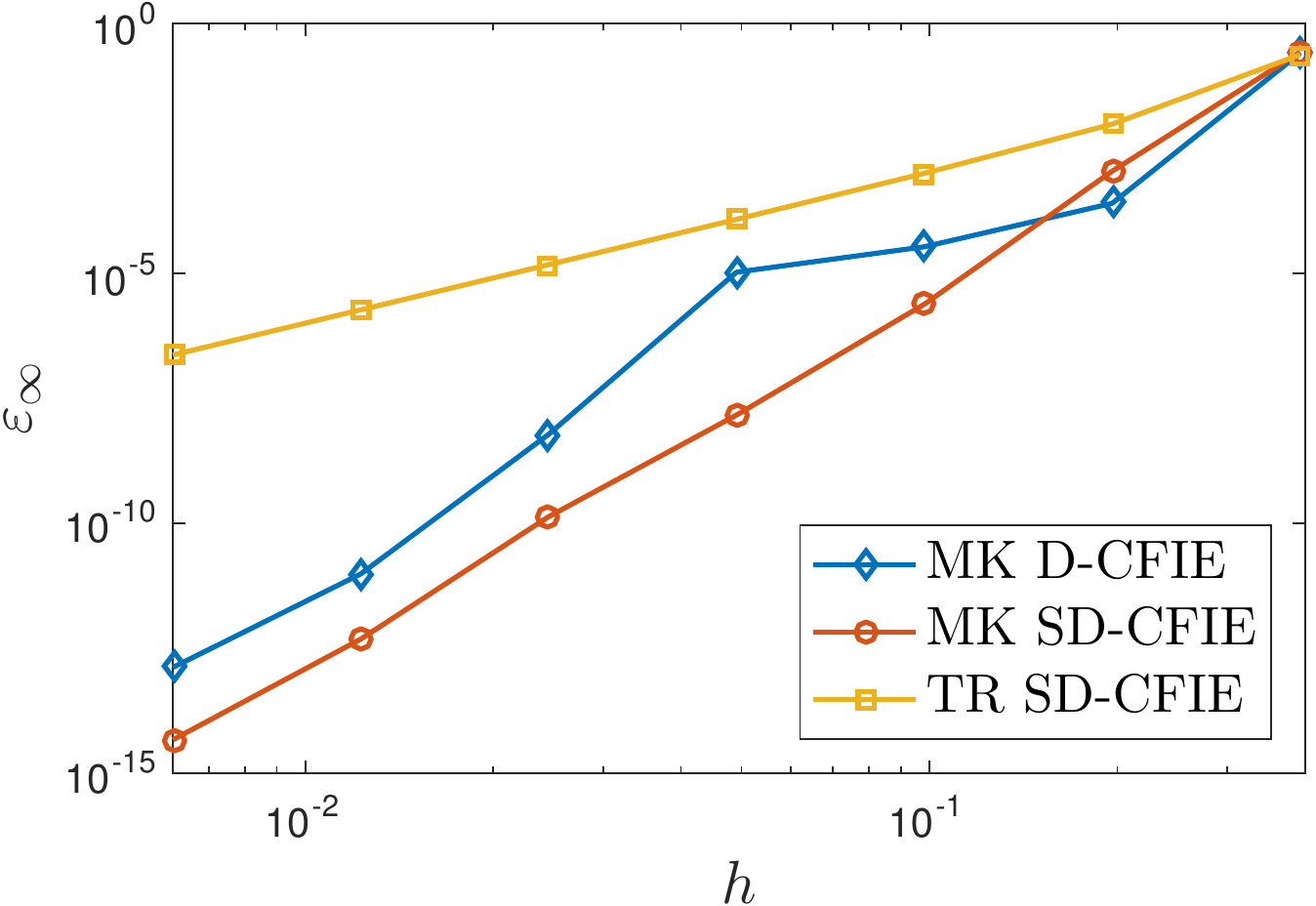}\qquad \includegraphics[scale=0.57]{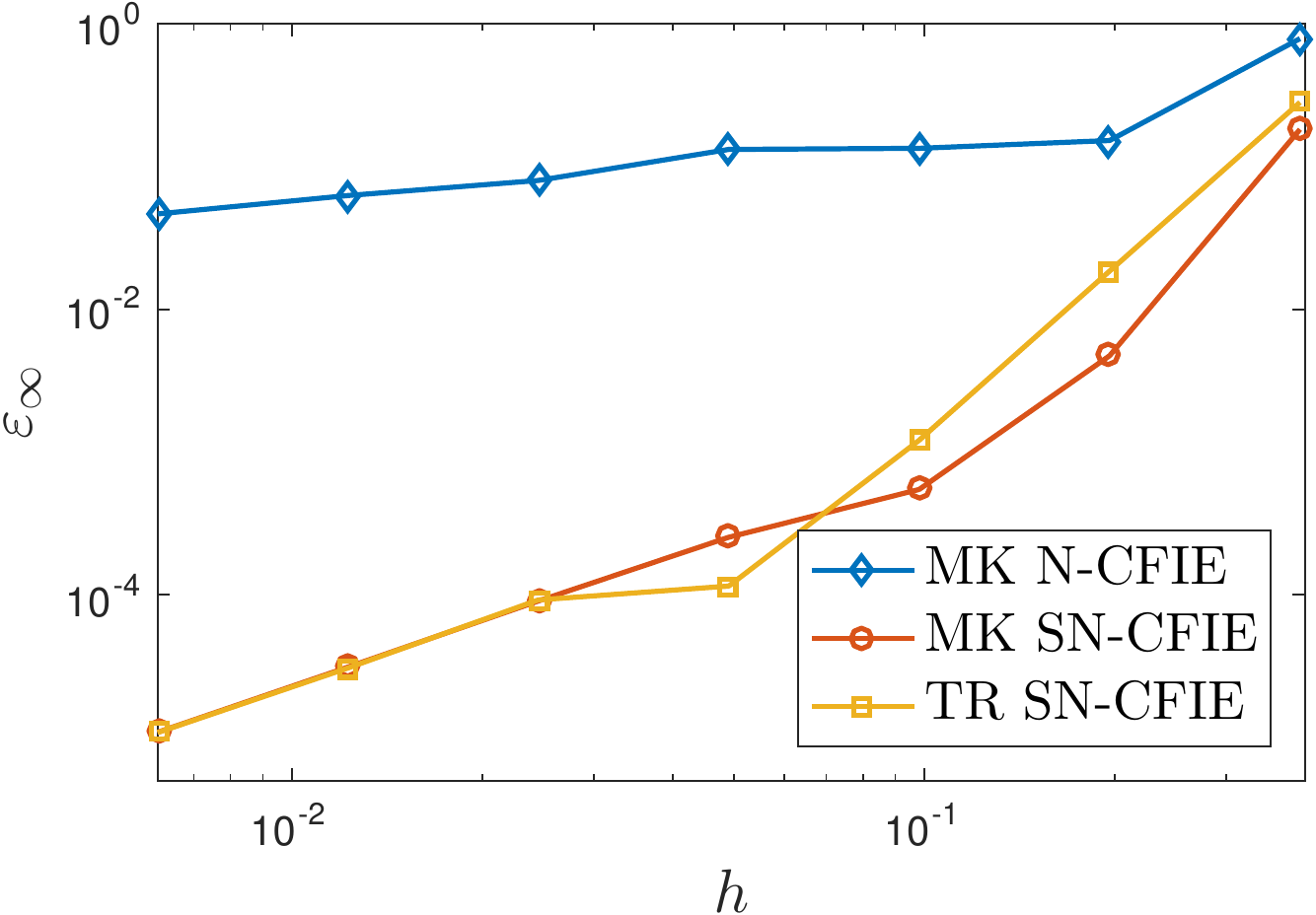}\label{fig_fix_d}}
\caption{Far-field errors in solutions of Dirichlet~\eqref{eq:DEP} and Neumann~\eqref{eq:NEP} problems involving two close kite-shaped obstacles, produced by smoothed and non-smoothed integral equation solutions obtained by means of MK and TR Nystr\"om methods. The wavenumber $k=4$ was utilized in all the examples considered in these plots.}\label{fig:results_close_convergence}
\end{figure}

Figures~\ref{fig:results_near_close_dir} and~\ref{fig:results_near_close_neu}, finally, show the logarithm (in base 10) of the absolute errors $|\tilde u(\nex)-u^{\mathrm{exact}}(\nex)|$ for the Dirichlet and Neumann problems, respectively, where $\tilde u$ denotes any of the approximate solutions obtained by means of the smoothed and non-smoothed integral equations using the smoothed and non-smoothed combined potentials. 

  \begin{figure}[h!]
\centering	
 \subfloat[][D-CFIE and combined potential]{\includegraphics[scale=0.0550]{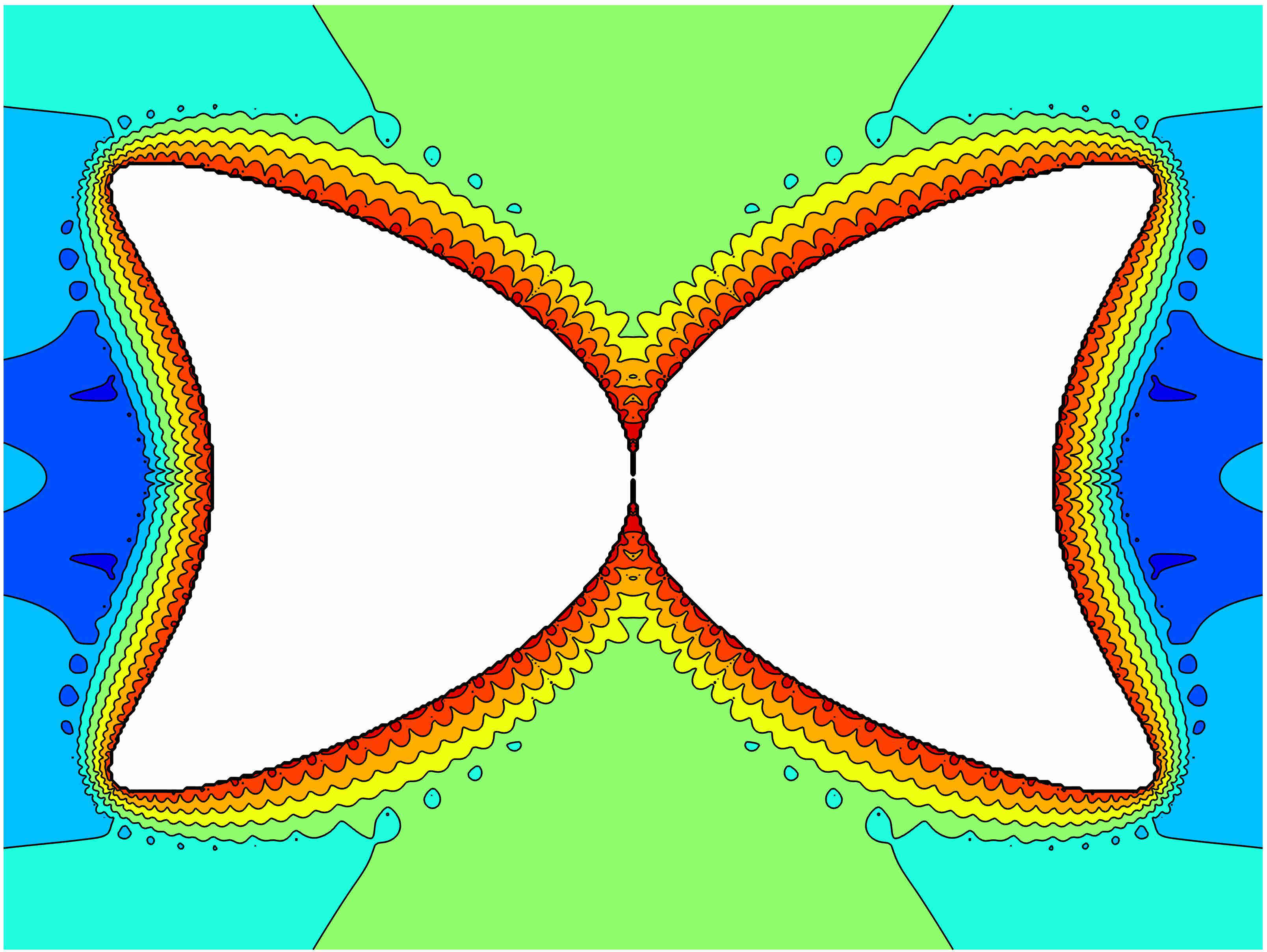}\label{7a}}\qquad
 \subfloat[][D-CFIE and smoothed combined potential]{\includegraphics[scale=0.0550]{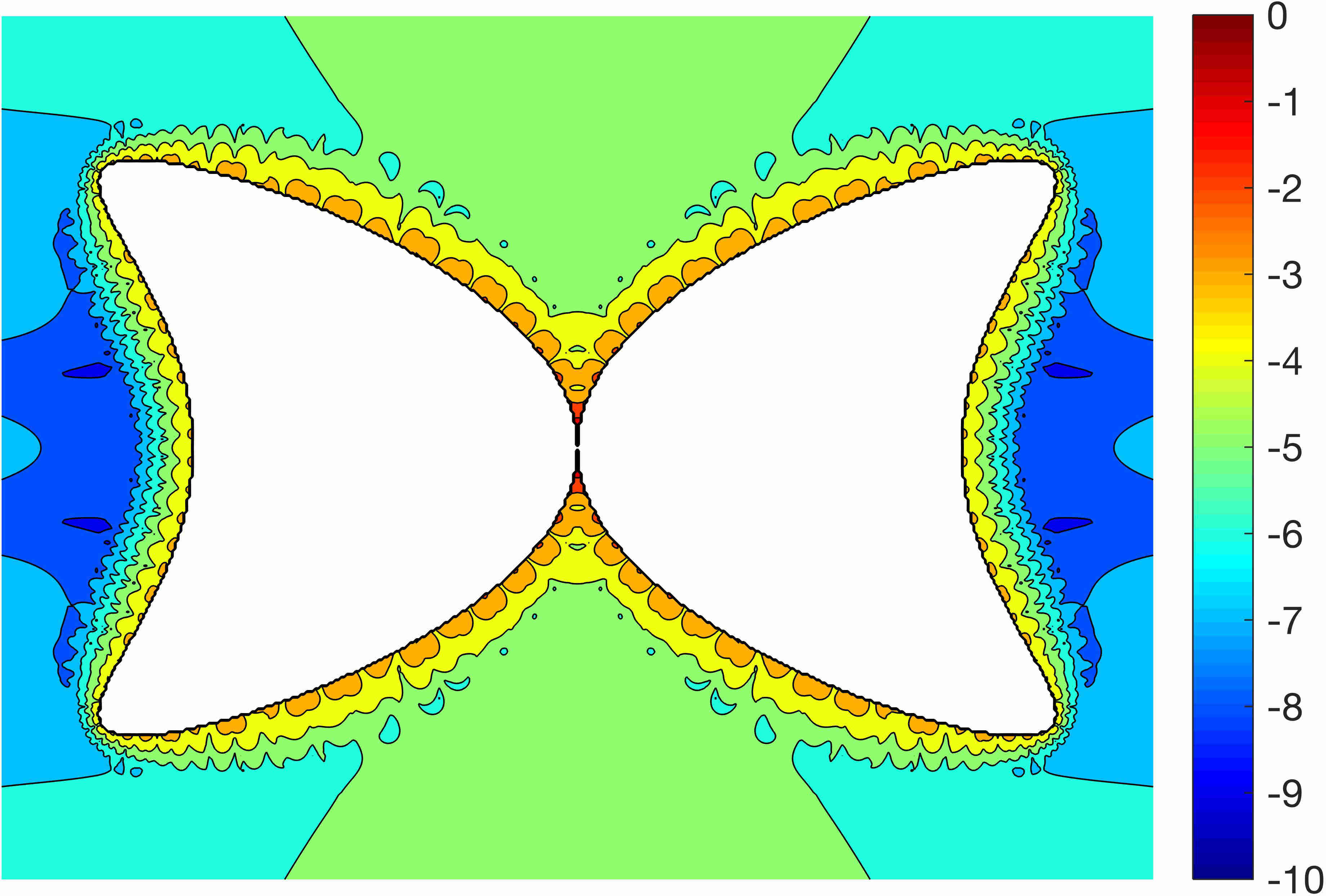}\label{8b}}\\
  \subfloat[][SD-CFIE and combined potential]{\includegraphics[scale=0.0550]{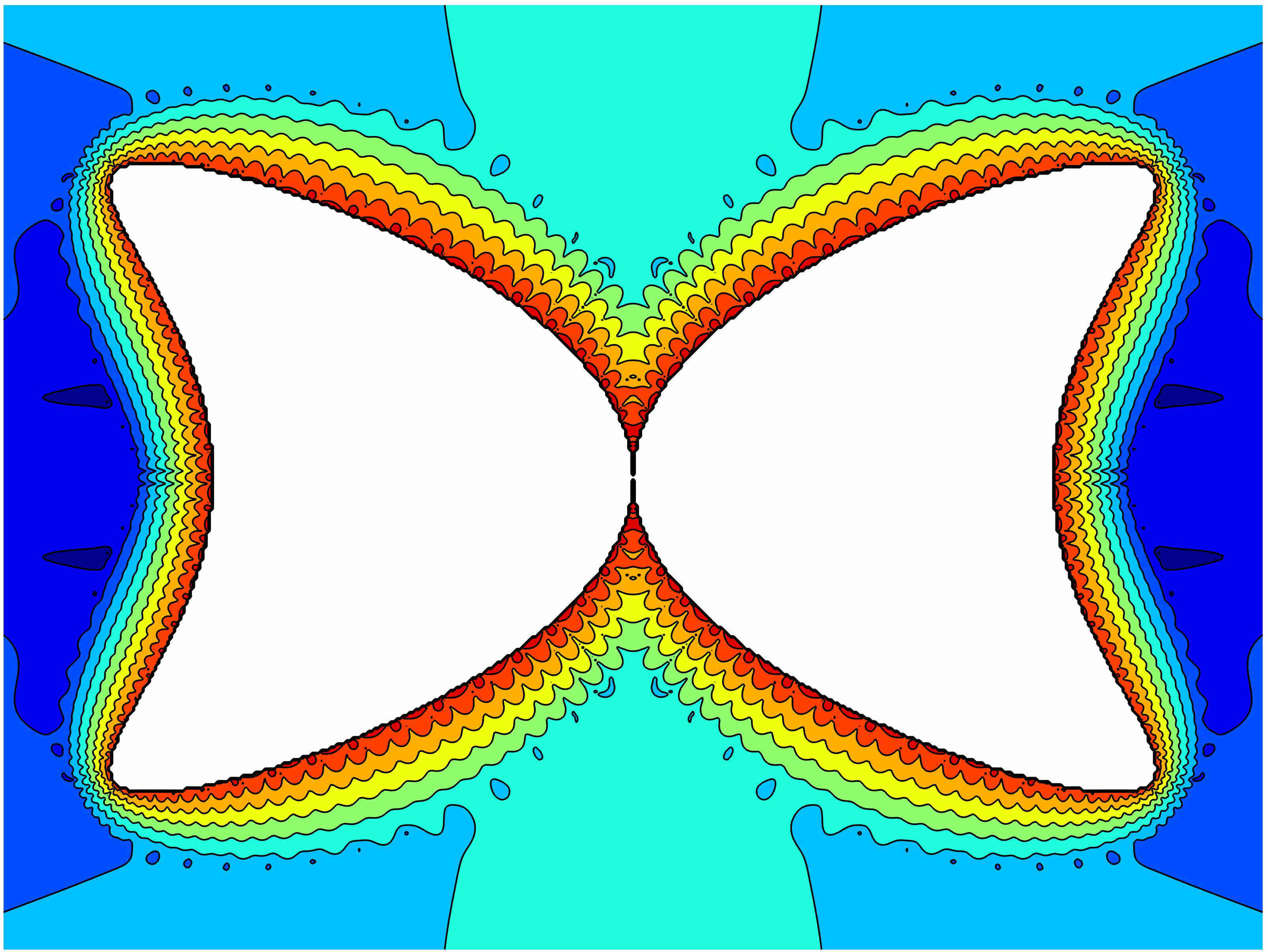}\label{7c}}\qquad
 \subfloat[][SD-CFIE and smoothed combined potential]{\includegraphics[scale=0.0550]{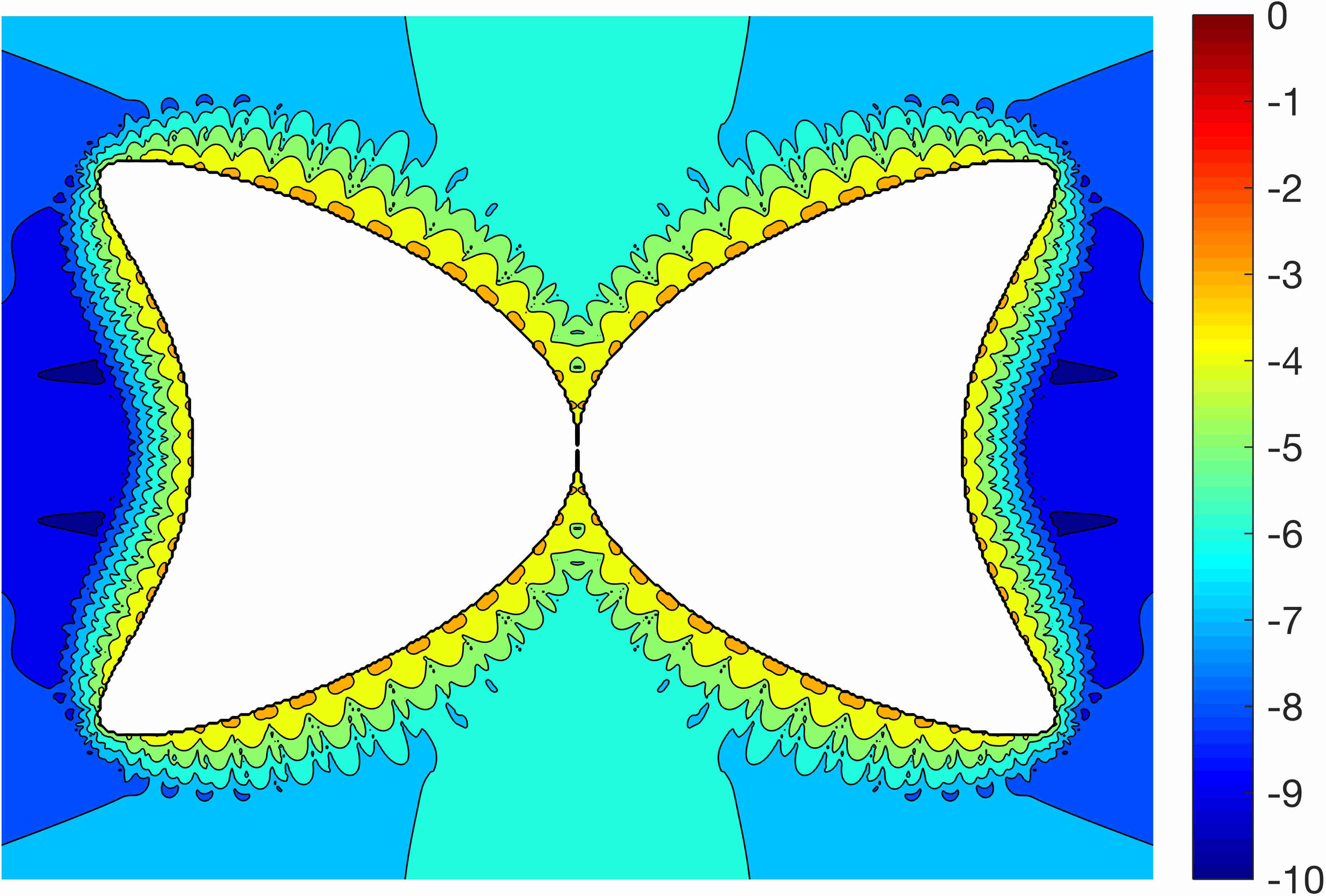}\label{8d}}
\caption{Errors ($\log_{10}|\tilde u(\nex)-u^{\mathrm{exact}}(\nex)|$) in the approximate solutions $\tilde u$ of the Dirichlet problem~\eqref{eq:DEP} involving two close kite-shaped obstacles separated by a distance $d=10^{-5}$. (a) Combined potential applied to D-CFIE solution; (b) smoothed combined field potential applied to D-CFIE soltuion; (c) combined potential applied to SD-CFIE solution, and (d) smoothed combined potential applied to SD-CFIE solution. The integral equations were discretized using the MK Nystr\"om method with $N=2\pi/h=64$ quadrature points on each curve. The wavenumber $k=4$ was utilized in all the examples considered in these plots.}\label{fig:results_near_close_dir}
\end{figure}  

 \begin{figure}[h!]
\centering	
 \subfloat[][N-CFIE and combined potential]{\includegraphics[scale=0.0550]{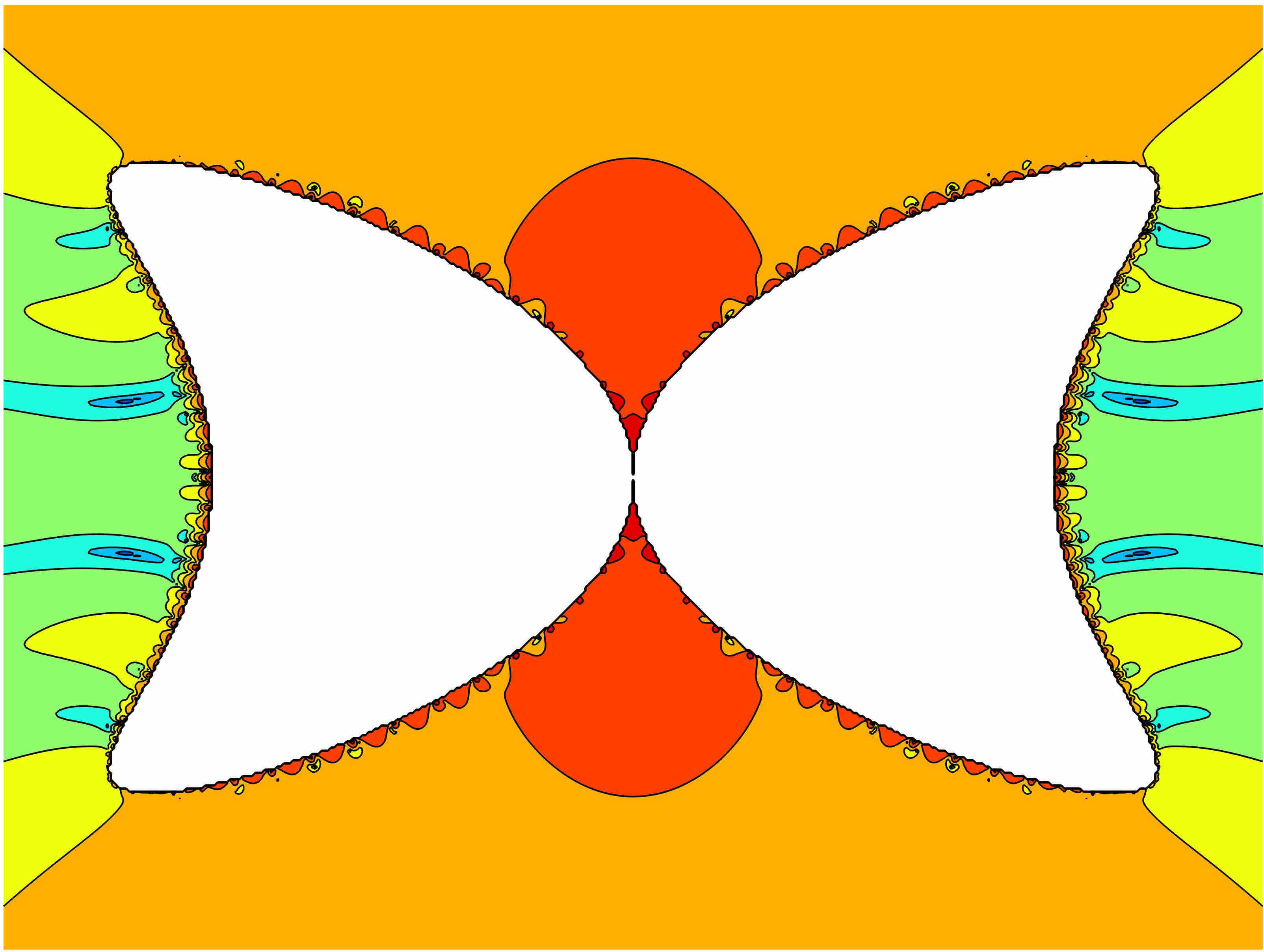}\label{8a}}\qquad
 \subfloat[][N-CFIE and smoothed combined potential]{\includegraphics[scale=0.0550]{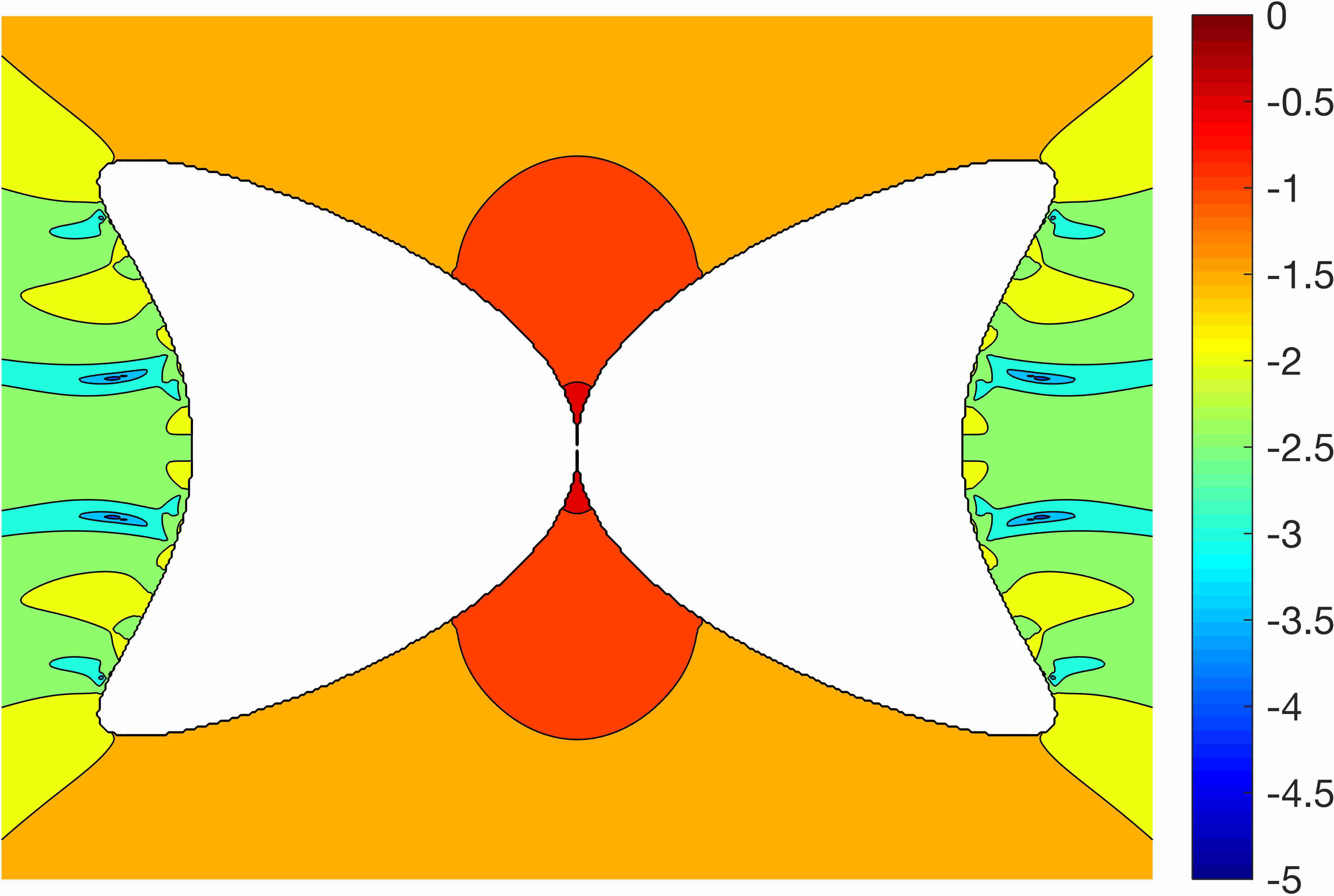}\label{8b}}\\
  \subfloat[][SN-CFIE and combined potential]{\includegraphics[scale=0.0550]{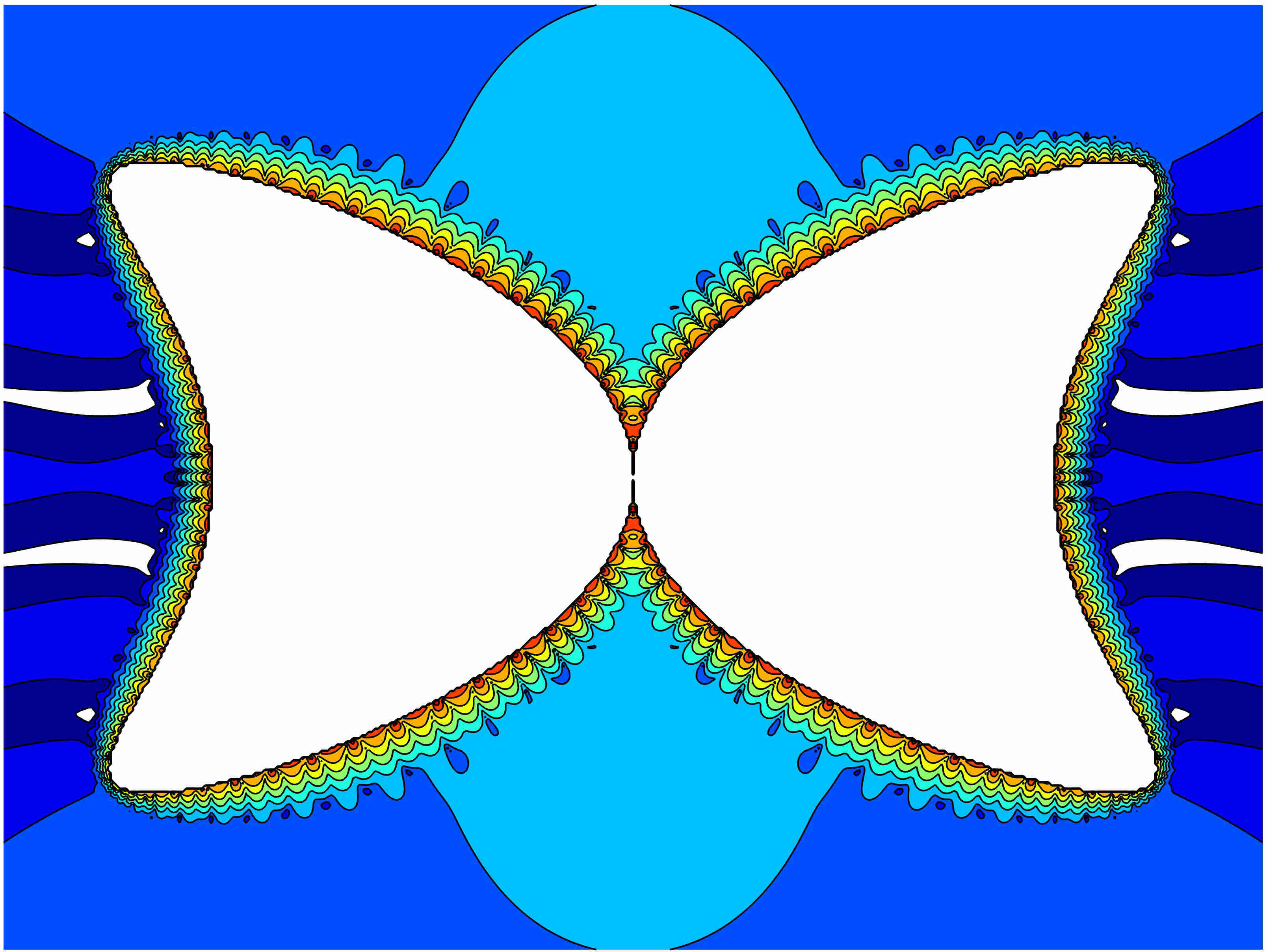}\label{8c}}\qquad
 \subfloat[][SN-CFIE and smoothed combined potential]{\includegraphics[scale=0.0550]{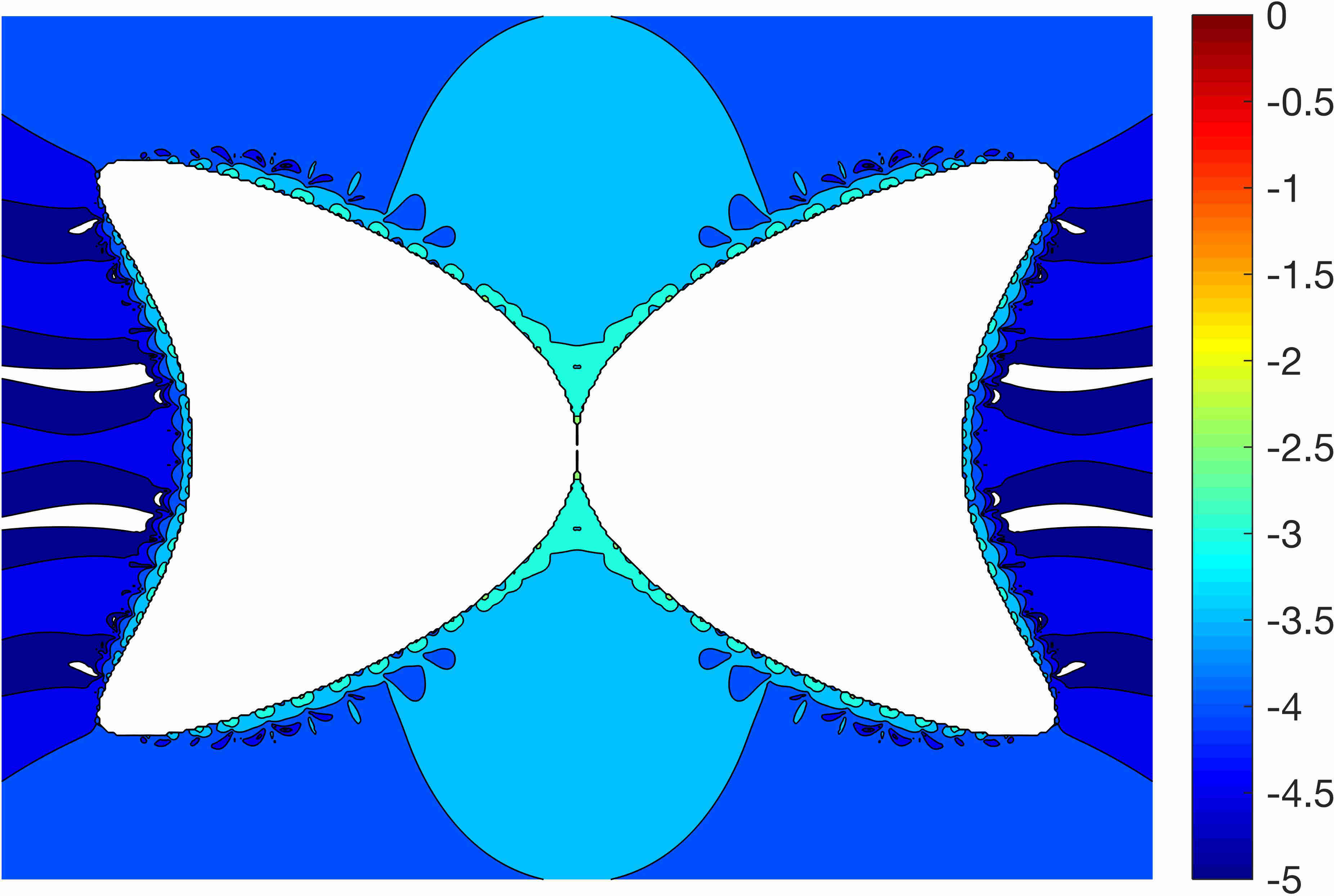}\label{8d}}
\caption{Errors ($\log_{10}|\tilde u(\nex)-u^{\mathrm{exact}}(\nex)|$) in the approximate solutions $\tilde u$ of the Neumann problem~\eqref{eq:NEP} involving two close kite-shaped obstacles separated by a distance $d=10^{-5}$.  (a) Combined potential applied to D-CFIE solution; (b) smoothed combined field potential applied to D-CFIE solution; (c) combined potential applied to SD-CFIE solution, and (d) smoothed combined potential applied to SN-CFIE solution. The integral equations were discretized using the MK Nystr\"om method with $N=2\pi/h=64$ quadrature points on each curve. The wavenumber $k=4$ was utilized in all the examples considered in these plots. }\label{fig:results_near_close_neu}
\end{figure}  
 
  \subsection{Obstacles with corners}\label{sec:corners}
 In this section we  extend  the smoothing procedure to problems of scattering involving obstacles with corners. For presentation simplicity we assume that the boundary~$\Gamma$, which is parametrized by $\bnex:[0,2\pi]\to\Gamma$, has only one corner at the point $\bnex(0)=\bnex(2\pi)$.  
 
In order to deal with the corner singularity of the integral equations solutions~(cf.~\cite{bruno2009high,grisvard2011elliptic,zargaryan1984asymptotic}) we introduce graded meshes generated by use of the change of variable $t = w(s)$, where letting $p\geq 2$ and 
\begin{eqnarray}
v(s)& =& \left(\frac{1}{p}-\frac{1}{2}\right)\left(\frac{\pi-s}{\pi}\right)^3+\frac{1}{p}\frac{s-\pi}{\pi}+\frac{1}{2},\nonumber
\end{eqnarray}
the function $w:[0,2\pi]\to[0,2\pi]$ is given by
\begin{eqnarray}
w(s)&=&2\pi\frac{[v(s)]^p}{[v(s)]^p+[v(2\pi-s)]^p}.\label{eq:change_variable}
\end{eqnarray}
Note that  $w$ is a smooth and monotonically increasing function on the interval~$[0,
2\pi]$ and its derivatives vanish algebraically at  $s=0$ and $s=2\pi$, i.e., $w^{(q)}(0) = w^{(q)}(2\pi)= 0$ for $1 \leq q \leq p -
1$. This change of variable, which was originally introduced by Kress~\cite{Kress:1990vm}, has been extensively utilized to produce high-order Nystr\"om discretizations of boundary integral equations involving domains with corners~\cite{anand2012well,dominguez2015well,turc2016well}.

Using this transformation then, the integral~\eqref{eq:integral} is expressed as
\begin{equation}\int_0^{2\pi}\left\{ A(w(s),w(\sigma))\rho_D(w(\sigma)|w(s))+B(w(s),w(\sigma))\rho_S(w(\sigma)|w(s))\rg\}w'(\sigma)\de \sigma,\quad s\in[0,2\pi].\label{eq:integral_2}\end{equation}
 Since the function  $w'(\sigma)$ in~\eqref{eq:integral_2} vanishes algebraically at $\sigma=0$ and $\sigma=2\pi$, the integrand in~\eqref{eq:integral_2} can be regarded as a $2\pi$-periodic function which exhibits a certain degree of smoothness around the endpoints $\sigma=0$ and $\sigma=2\pi$ that can be controlled  by the parameter $p$. Therefore, for $p$ sufficiently large the smoothed integral equations~\eqref{eq:IE_SDirichlet} and~\eqref{eq:IE_SNeumann} resulting from use of this change of variable can be discretized using any of the Nystr\"om methods discussed  in the previous section. 
 
Utilizing the quadrature points $s_j=\pi/n(j+1/2)$, $j=0,\ldots,2n-1$---which do not include the endpoints 0 and $2\pi$---we  obtain a linear system for the approximate the values $\psi_j$, $j=0,\ldots, 2n-1,$ of the unknown function $\psi(s)=\phi(w(s))$ at the quadrature points $s=s_j$, $j=0,\ldots, 2n-1$, respectively. 
 
 Clearly,  the Nystr\"om discretizations of the smoothed integral equations presented in the previous section require approximate expressions for 
 \begin{equation*}\begin{split}
 \rho_D(w(s_j)|w(s_i)) 
=&\  \psi(s_j) - \tilde  p_0(w(s_j)|w(s_i))\psi(s_i)-\frac{\tilde p_1(w(s_j)|w(s_i))}{|\bnex'(w(s_i))|}\frac{\psi'(s_i)}{w'(s_i)}
\end{split}
\end{equation*}
and
 \begin{equation*}\begin{split}
 \rho_S(w(s_j)|w(s_i)) 
=&\  i\eta\psi(s_j) - \p_n\tilde p_0(w(s_j)|w(s_i))\psi(s_i)-\frac{\p_n\tilde p_1(w(s_j)|w(s_i))}{|\bnex'(w(s_i))|}\frac{\psi'(s_i)}{w'(s_i)} 
\end{split}
\end{equation*} 
in terms of the discrete values $\psi_j$, $j=0,\ldots,2n-1$. (Note that $w'(s_i)\neq 0$ for all $i=0,\ldots, 2n-1$.) These expressions are here obtained by approximating $\psi'(s_i)$ using finite differences of $\psi_j$ in the case of the low-order TR  method, and using FFT-based differentiation of  $\psi_j$  in the case of the higher order MK and KR methods. High accuracy is  achieved in the latter case by utilizing the identity
$$
\psi'(s_i) = \frac{(w'\psi)'(s_i)-\psi(s_i)w''(s_i)}{w'(s_i)},
$$
which allows the desired quantity $\psi'(s_i)$ to be approximated using FFF-based differentiation of the sequences $w'(s_j)\psi_j$ and $w'(s_j)$, $j=0,\ldots,2n-1$, which are assumed periodic. 
We proceed similarly to approximate the second-order derivative $\psi''(s_i)$ that is needed for evaluation of the diagonal term $H(w(s_i),w(s_i))\rho_D(w(s_i)|w(s_i))$.

Figure~\ref{fig:results_RCFiE_corners} display the far-field errors in the solution of the problem of scattering of a horizontal plane-wave $u^\inc(\nex)=\e^{ik(x\cos\alpha+y\sin\alpha)}$, $\alpha=0$, $k=4$, by a drop-shaped obstacle parametrized by $\bnex(t)=(2\sin\frac{t}{2},-\sin t)$~\cite{Kress:1990vm} that features a convex corner (Figures~\ref{fig_SD_CFIE_drop} and~\ref{fig_SN_CFIE_drop}), and by a boomerang-shaped obstacle parametrized by $\bnex(t)=(-\frac{2}{3}\sin\frac{3t}{2},-\sin t)$~\cite{anand2012well} that features a concave (reentry) corner (Figures~\ref{fig_SD_CFIE_boom} and~\ref{fig_SN_CFIE_boom}). The smoothed integral equations (SD-CFIE~\eqref{eq:IE_SDirichlet} and SN-CFIE~\eqref{eq:IE_SNeumann}) are discretized using the TR, MK and KR-10 Nystr\"om methods described in the previous section, using the change of variable described above with the parameter value $p=4$. The error curves displayed  in these figures demonstrate the accuracy of the three Nystr\"om method considered.  The numbers of the GMRES iterations required for the solution of the resulting linear systems are comparable to those reported in Table~\ref{tab:GMRES}.
\begin{figure}[h!]
\centering	
 \subfloat[][Dirichlet problem for drop-shaped obstacle]{\includegraphics[scale=0.57]{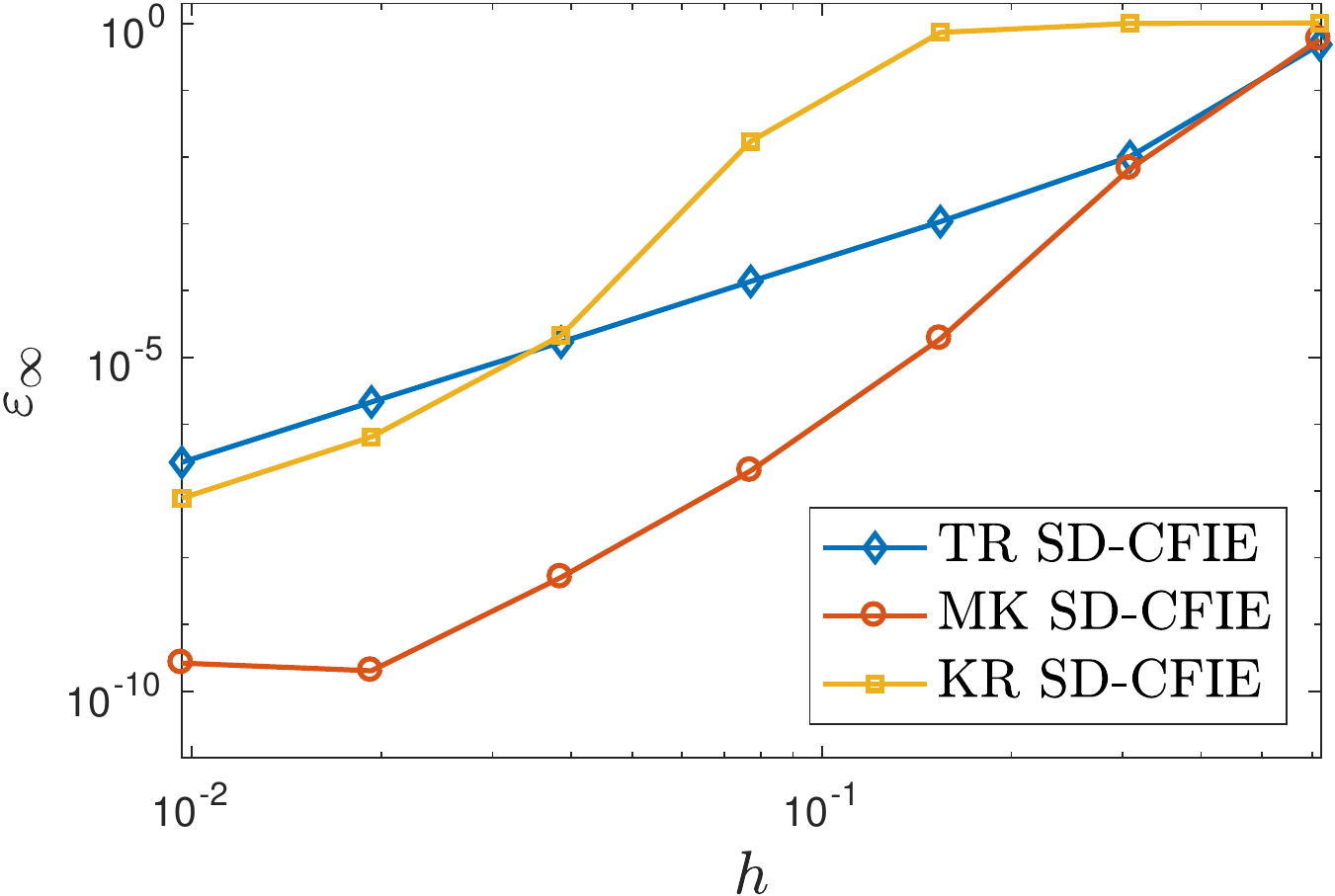}\label{fig_SD_CFIE_drop}}\quad
 \subfloat[][Neumann problem for drop-shaped obstacle]{\includegraphics[scale=0.57]{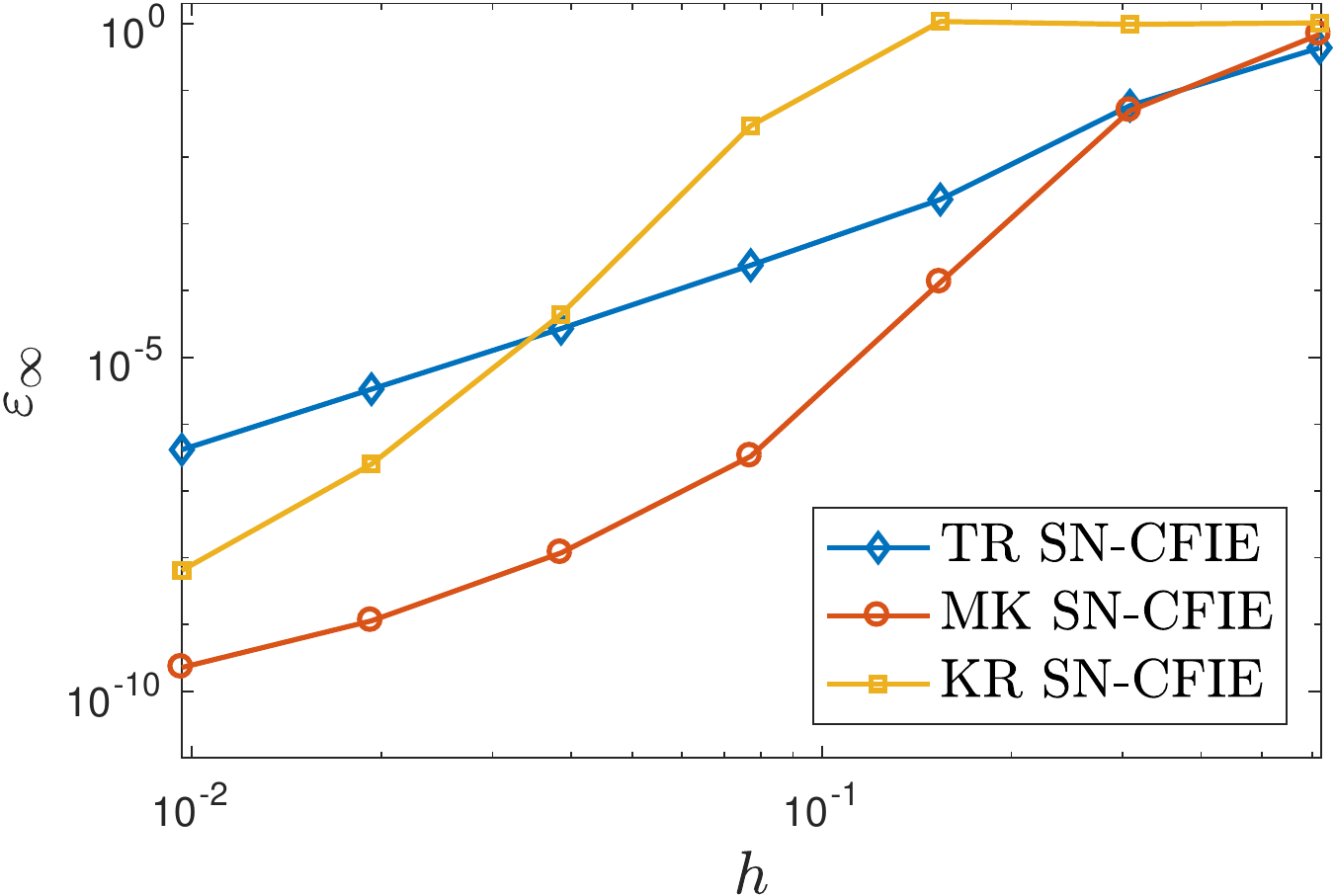}\label{fig_SN_CFIE_drop}}\\
  \subfloat[][Dirichlet problem for boomerang-shaped obstacle]{\includegraphics[scale=0.57]{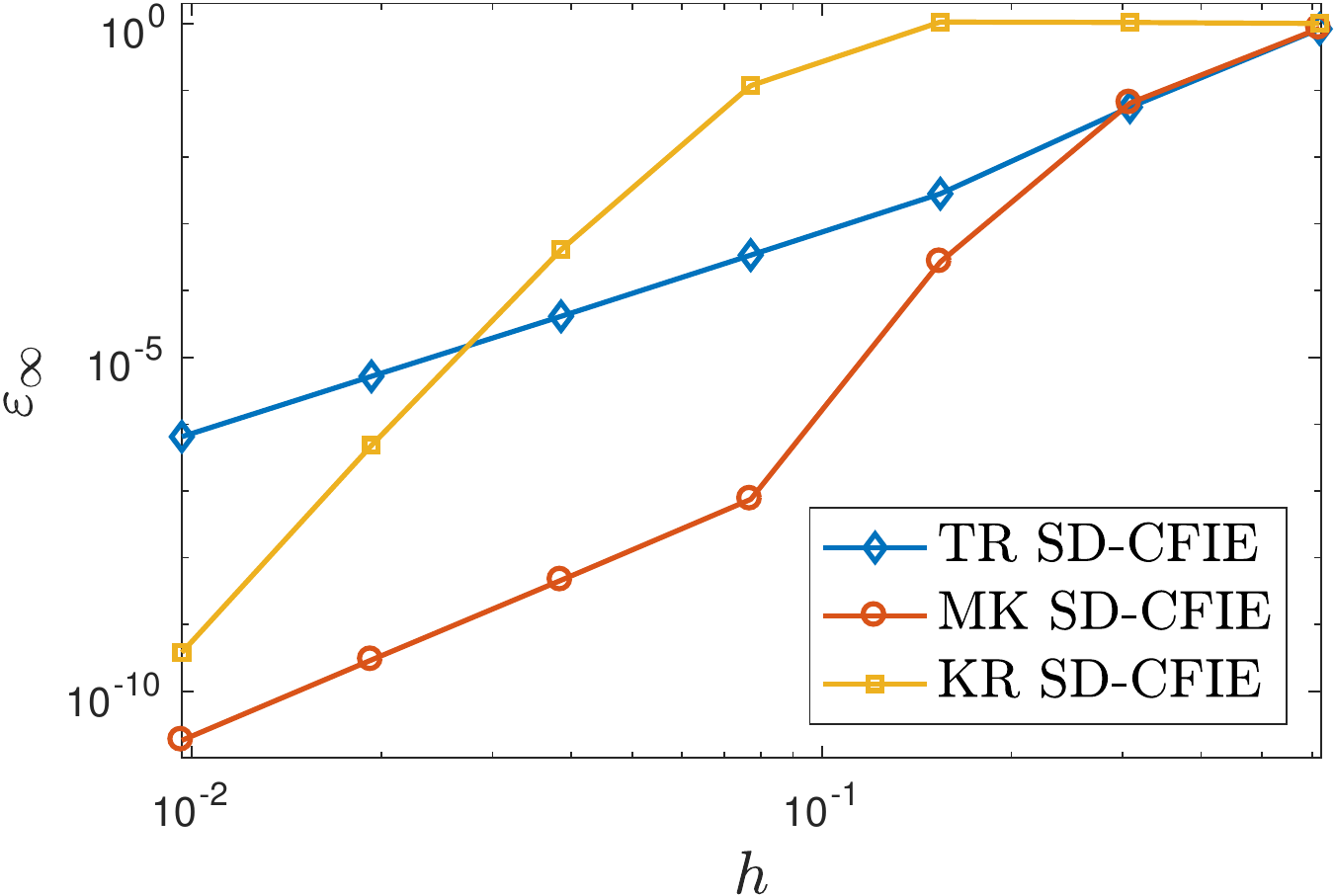}\label{fig_SD_CFIE_boom}}\quad
 \subfloat[][Neumann problem for boomerang-shaped obstacle]{\includegraphics[scale=0.57]{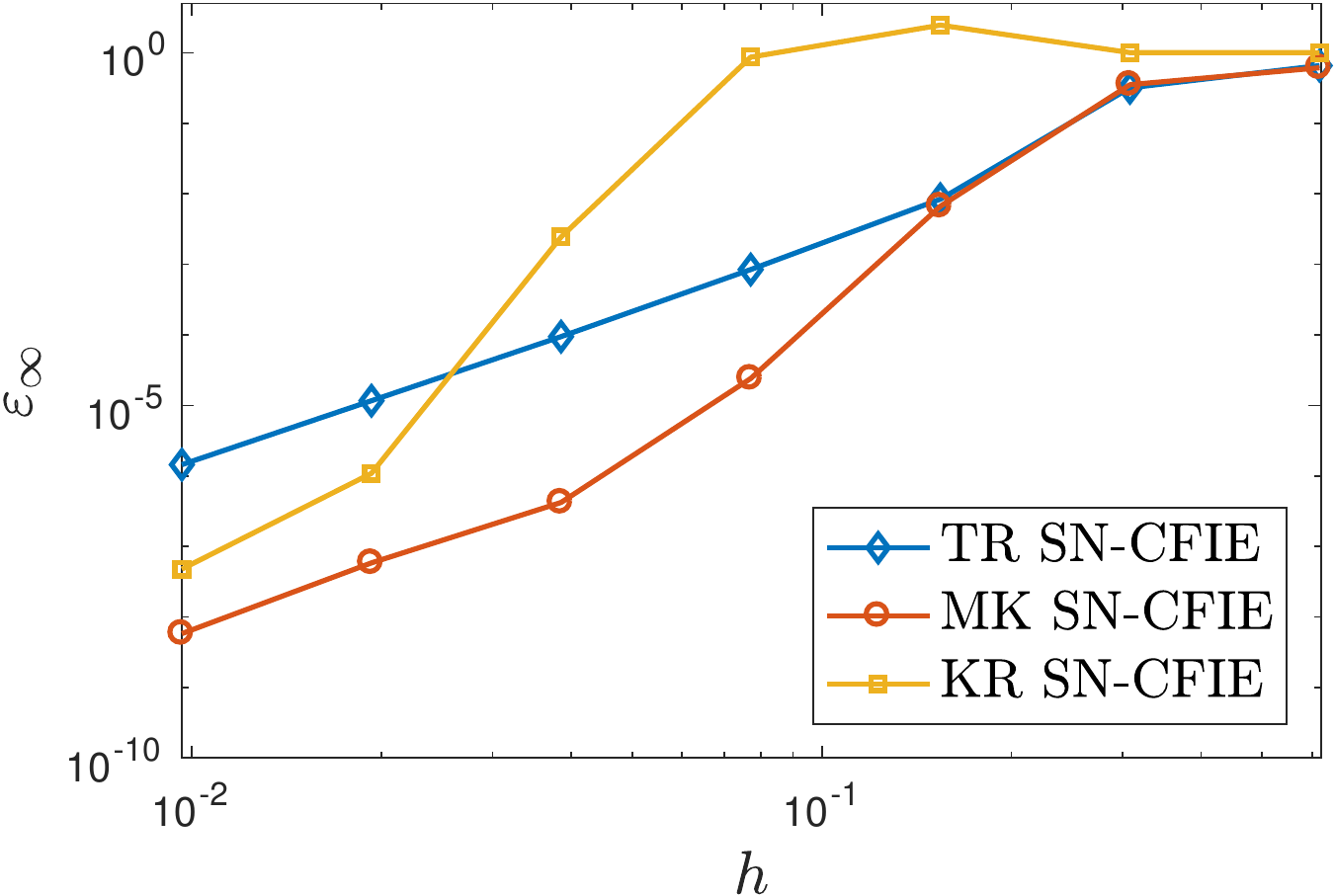}\label{fig_SN_CFIE_boom}}\\
\caption{Far-field  errors for various grid sizes $h$ obtained from various Nystr\"om discretizations of the smoothed integral equations:  Trapezoidal rule (TR), Martensen \& Kussmaul (MK) and Kapur-Rokhlin of order ten (KR-10).}\label{fig:results_RCFiE_corners}
\end{figure}  

Figure~\ref{fig:results_near_close_neu_corner} (resp. Figure~\ref{fig:results_near_close_neu_corner_boom}), finally,  displays the real part of total field (incident field plus scattered field)  for the solution of the problem of scattering of a horizontal high-frequency ($k=64$) plane-wave by a drop-shaped (resp. boomerang-shaped) obstacle with Dirichlet and Nuemann boundary conditions. The integral equation solutions were obtained via the MK Nystr\"om method and the near-field were produced through the smoothed potential~\eqref{eq:reg_pot}.

 \begin{figure}[h!]
\centering	
 \subfloat[][Drop-shaped obstacle with Dirichlet boundary condition.]{\includegraphics[scale=0.4]{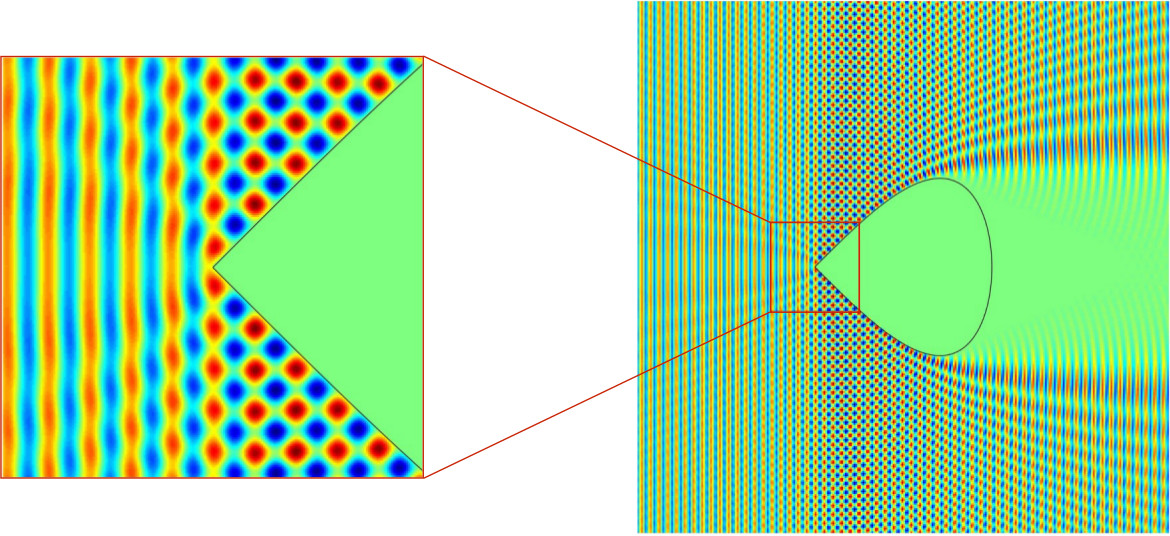}\label{10a}}\\
  \subfloat[][Drop-shaped obstacle with Neumann boundary condition.]{\includegraphics[scale=0.4]{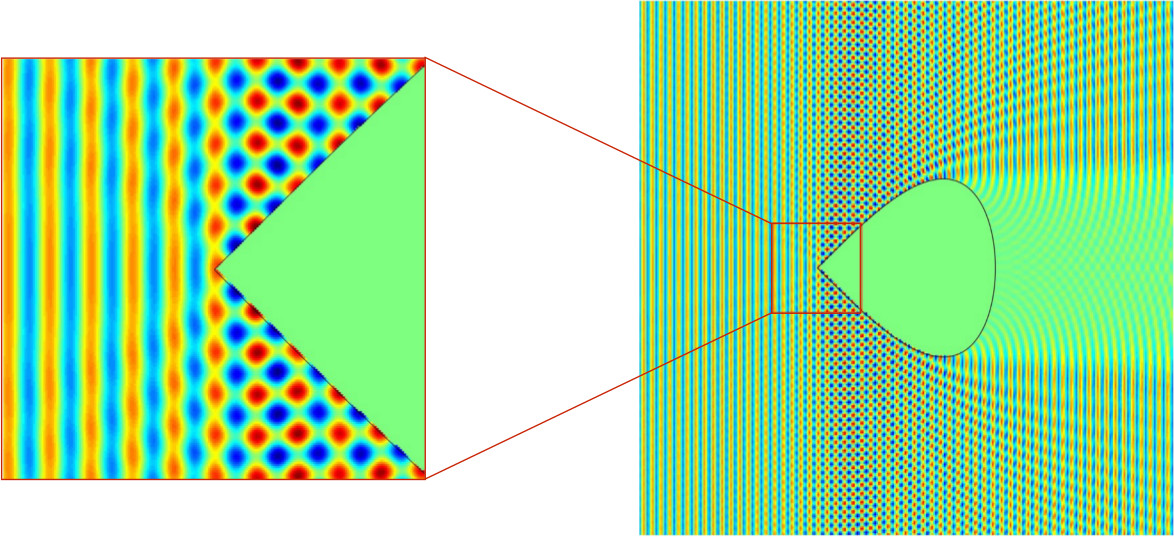}\label{10b}}
\caption{Total fields resulting from the high-frequency scattering of a horizontal plane-wave $u^\inc(\nex) = \e^{ikx}$, $k=64$, by an obstacle with a corner. A region around the corner point is zoomed-in  and displayed on the left-hand-side of the figure.}\label{fig:results_near_close_neu_corner}
\end{figure}

 \begin{figure}[h!]
\centering	
 \subfloat[][Boomerang-shaped obstacle with Dirichlet boundary condition.]{\includegraphics[scale=0.4]{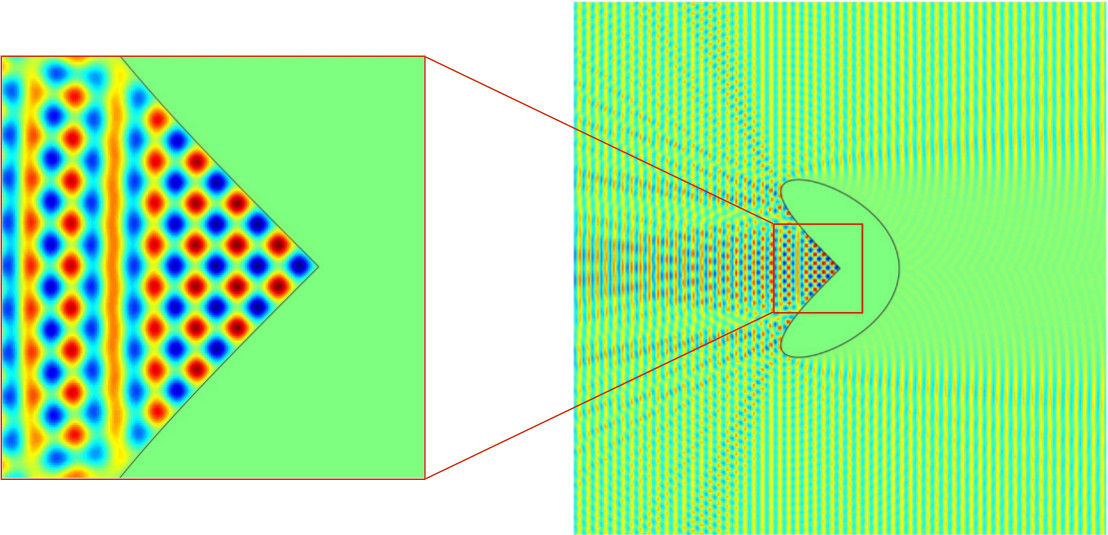}\label{11a}}\\
  \subfloat[][Boomerang-shaped obstacle with Neumann boundary condition.]{\includegraphics[scale=0.4]{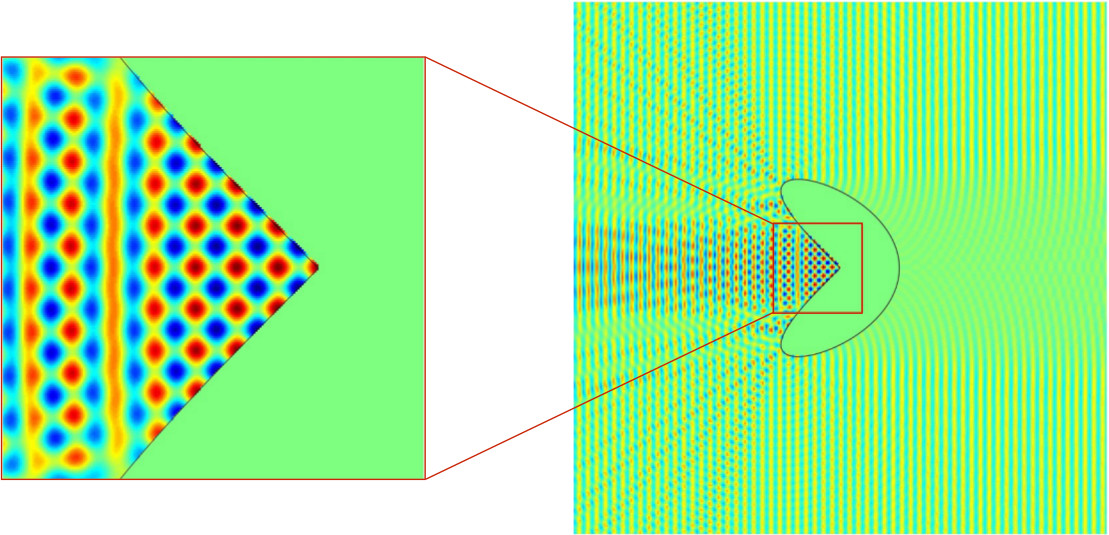}\label{11b}}
\caption{Total fields resulting from the high-frequency scattering of a horizontal plane-wave $u^\inc(\nex) = \e^{ikx}$, $k=64$, by an obstacle with a corner. A region around the corner point is zoomed-in  and displayed on the left-hand-side of the figure.}\label{fig:results_near_close_neu_corner_boom}
\end{figure}  
\section*{Acknowledgement} The author would like to thank Prof. Catalin Turc with the Department of Mathematics of the New Jersey Institute of Technology, for interesting discussion and useful comments on the subject of this manuscript.
\bibliographystyle{abbrv}

\bibliography{References}

\begin{thebibliography}{10}

\bibitem{anand2012well}
A.~Anand, J.~Ovall, C.~Turc, et~al.
\newblock Well-conditioned boundary integral equations for two-dimensional
  sound-hard scattering problems in domains with corners.
\newblock {\em Journal of Integral Equations and Applications}, 24(3):321--358,
  2012.

\bibitem{atkinson1997numerical}
K.~E. Atkinson.
\newblock {\em The numerical solution of integral equations of the second
  kind}, volume~4.
\newblock Cambridge university press, 1997.

\bibitem{Barnett:2014tq}
A.~Barnett.
\newblock {Evaluation of layer potentials close to the boundary for Laplace and
  Helmholtz problems on analytic planar domains}.
\newblock {\em SIAM Journal on Scientific Computing}, 36(2):A427--A451, 2014.

\bibitem{Bebendorf2008}
M.~Bebendorf.
\newblock {\em Hierarchical Matrices}.
\newblock Springer, 2008.

\bibitem{bonnet1999boundary}
M.~Bonnet.
\newblock {\em Boundary integral equation methods for solids and fluids}.
\newblock John Wiley, 1995.

\bibitem{boubendir2016high}
Y.~Boubendir, C.~Turc, and V.~Dom{\'\i}nguez.
\newblock High-order {N}ystr{\"o}m discretizations for the solution of integral
  equation formulations of two-dimensional {H}elmholtz transmission problems.
\newblock {\em IMA Journal of Numerical Analysis}, 36(1):463--492, 2016.

\bibitem{brakhage1965dirichletsche}
H.~Brakhage and P.~Werner.
\newblock {\"U}ber das dirichletsche aussenraumproblem f{\"u}r die
  {H}elmholtzsche schwingungsgleichung.
\newblock {\em Archiv der Mathematik}, 16(1):325--329, 1965.

\bibitem{Bruno:2012dx}
O.~P. Bruno, T.~Elling, and C.~Turc.
\newblock {Regularized integral equations and fast high-order solvers for
  sound-hard acoustic scattering problems}.
\newblock {\em International Journal for Numerical Methods in Engineering},
  91(10):1045--1072, June 2012.

\bibitem{Bruno:2001ima}
O.~P. Bruno and L.~A. Kunyansky.
\newblock {A fast, high-order algorithm for the solution of surface scattering
  problems: basic implementation, tests, and applications}.
\newblock {\em Journal of Computational Physics}, 2001.

\bibitem{bruno2009high}
O.~P. Bruno, J.~S. Ovall, and C.~Turc.
\newblock A high-order integral algorithm for highly singular pde solutions in
  lipschitz domains.
\newblock {\em Computing}, 84(3-4):149--181, 2009.

\bibitem{Burton1971Application}
A.~Burton and G.~Miller.
\newblock The application of integral equation methods to the numerical
  solution of some exterior boundary-value problems.
\newblock {\em Proceedings of the Royal Society of London. Series A,
  Mathematical and Physical Sciences}, pages 201--210, 1971.

\bibitem{COLTON:2012}
D.~Colton and R.~Kress.
\newblock {\em Inverse Acoustic and Electromagnetic Scattering Theory},
  volume~93.
\newblock Springer, third edition, 2012.

\bibitem{COLTON:1983}
D.~L. Colton and R.~Kress.
\newblock {\em Integral Equation Methods in Scattering Theory}.
\newblock Pure and Applied Mathematics. John Wiley \& Sons Inc., first edition,
  1983.

\bibitem{dominguez2014nystrom}
V.~Dom{\'\i}nguez, S.~L. Lu, and F.-J. Sayas.
\newblock A nystr{\"o}m flavored calder{\'o}n calculus of order three for two
  dimensional waves, time-harmonic and transient.
\newblock {\em Computers \& Mathematics with Applications}, 67(1):217--236,
  2014.

\bibitem{dominguez2015well}
V.~Dominguez, M.~Lyon, and C.~Turc.
\newblock Well-posed boundary integral equation formulations and
  nystr$\backslash$" om discretizations for the solution of helmholtz
  transmission problems in two-dimensional lipschitz domains.
\newblock {\em arXiv preprint arXiv:1509.04415}, 2015.

\bibitem{Givoli2013}
D.~Givoli.
\newblock {\em Numerical methods for problems in infinite domains}, volume~33.
\newblock Elsevier, 2013.

\bibitem{Greengard1998}
L.~Greengard, J.~Huang, V.~Rokhlin, and S.~Wandzura.
\newblock Accelerating fast multipole methods for the helmholtz equation at low
  frequencies.
\newblock {\em IEEE Computational Science and Engineering}, 5(3):32--38, 1998.

\bibitem{grisvard2011elliptic}
P.~Grisvard.
\newblock {\em Elliptic problems in nonsmooth domains}, volume~69.
\newblock SIAM, 2011.

\bibitem{hao2014high}
S.~Hao, A.~Barnett, P.-G. Martinsson, and P.~Young.
\newblock High-order accurate methods for nystr{\"o}m discretization of
  integral equations on smooth curves in the plane.
\newblock {\em Advances in Computational Mathematics}, 40(1):245--272, 2014.

\bibitem{JIN:2002}
J.~Jin.
\newblock {\em The Finite Element Method in Electromagnetics`}, chapter Finite
  elements-boundary integral methods, pages 407--486.
\newblock Wiley, 2002.

\bibitem{Klaseboer:2012ks}
E.~Klaseboer, Q.~Sun, and D.~Y.~C. Chan.
\newblock {Non-singular boundary integral methods for fluid mechanics
  applications}.
\newblock {\em Journal of Fluid Mechanics}, 696:468--478, 2012.

\bibitem{klockner2013quadrature}
A.~Kl{\"o}ckner, A.~Barnett, L.~Greengard, and M.~OʼNeil.
\newblock Quadrature by expansion: A new method for the evaluation of layer
  potentials.
\newblock {\em Journal of Computational Physics}, 252:332--349, 2013.

\bibitem{Kress:1990vm}
R.~Kress.
\newblock {A Nystr{\"o}m method for boundary integral equations in domains with
  corners}.
\newblock {\em Numerische Mathematik}, 58(1):145--161, 1990.

\bibitem{Kress:1995}
R.~Kress.
\newblock {On the numerical solution of a hypersingular integral equation in
  scattering theory}.
\newblock {\em Journal of computational and applied mathematics},
  61(3):345--360, 1995.

\bibitem{kress2012linear}
R.~Kress.
\newblock {\em Linear Integral Equations}, volume~82.
\newblock Springer, 3rd edition, 2014.

\bibitem{KUSSMAUL:1969}
R.~Kussmaul.
\newblock Ein numerisches {V}erfahren zur {L}\"osung des {N}eumannschen
  {A}ussenraumproblems f\"ur die {H}elmholtzsche {S}chwingungsgleichung.
\newblock {\em Computing (Arch. Elektron. Rechnen)}, 4:246--273, 1969.

\bibitem{leis1965dirichletschen}
R.~Leis.
\newblock Zur {D}irichletschen randwertaufgabe des aussenraumes der
  schwingungsgleichung.
\newblock {\em Mathematische Zeitschrift}, 90(3):205--211, 1965.

\bibitem{MARTENSEN:1963}
E.~Martensen.
\newblock {\"U}ber eine methode zum r\"aumlichen {N}eumannschen problem mit
  einer anwendung f\"ur torusartige berandungen.
\newblock {\em Acta Mathematica}, 109:75--135, 1963.

\bibitem{Maue:1949in}
A.~W. Maue.
\newblock {Zur Formulierung eines allgemeinen Beugungs-problems durch eine
  Integralgleichung}.
\newblock {\em Zeitschrift f�r Physik}, 126(7-9):601--618, July 1949.

\bibitem{Mclean2000Strongly}
W.~C.~H. McLean.
\newblock {\em Strongly elliptic systems and boundary integral equations}.
\newblock Cambridge University Press, 2000.

\bibitem{panish1965question}
I.~Panich.
\newblock On the question of the solvability of the exterior boundary-value
  problems for the wave equation and {Maxwell's} equations.
\newblock {\em Russian Mathematical Surveys}, 20:221--226, 1965.

\bibitem{phillips1997precorrected}
J.~R. Phillips and J.~K. White.
\newblock A precorrected-{FFT} method for electrostatic analysis of complicated
  {3-D} structures.
\newblock {\em IEEE Transactions on Computer-Aided Design of Integrated
  Circuits and Systems}, 16(10):1059--1072, 1997.

\bibitem{saranen2013periodic}
J.~Saranen and G.~Vainikko.
\newblock {\em Periodic integral and pseudodifferential equations with
  numerical approximation}.
\newblock Springer Science \& Business Media, 2013.

\bibitem{Sauter2010}
S.~A. Sauter and C.~Schwab.
\newblock {\em Boundary Element Methods}.
\newblock Springer, 2010.

\bibitem{sun2015boundary}
Q.~Sun, E.~Klaseboer, B.-C. Khoo, and D.~Y. Chan.
\newblock Boundary regularized integral equation formulation of the {H}elmholtz
  equation in acoustics.
\newblock {\em Royal Society open science}, 2(1):140520, 2015.

\bibitem{taflove1995computational}
A.~Taflove and S.~H. Hagness.
\newblock {\em Computational electrodynamics: {T}he finite-difference
  time-domain method}.
\newblock Artech House, 2005.

\bibitem{turc2016well}
C.~Turc, Y.~Boubendir, and M.~K. Riahi.
\newblock Well-conditioned boundary integral equation formulations and
  nystr$\backslash$" om discretizations for the solution of helmholtz problems
  with impedance boundary conditions in two-dimensional lipschitz domains.
\newblock {\em arXiv preprint arXiv:1607.00769}, 2016.

\bibitem{zargaryan1984asymptotic}
S.~Zargaryan and V.~Maz'ya.
\newblock The asymptotic form of the solutions of the integral equations of
  potential theory in the neighbourhood of the corner points of a contour.
\newblock {\em Journal of Applied Mathematics and Mechanics}, 48(1):120--124,
  1984.

\end{thebibliography}

\end{document}